\documentclass[reqno,a4paper]{amsart}

\title{Affine Deformations of Divisible Convex Cones and Affine Spacetimes}

\author{Antoine Ablondi}
\address{Antoine Ablondi: IMAG, Universit\'e de Montpellier, CNRS, Montpellier, France.} 
\email{antoine.ablondi@umontpellier.fr}
\date{}

\usepackage[margin=2.5cm]{geometry}
\usepackage{setspace}

\usepackage{enumitem}

\usepackage{amsmath,amssymb,amscd,amsthm,mathtools}
\usepackage[T1]{fontenc}
\usepackage[english]{babel}

\usepackage{tikz-cd}
\usepackage{pgfplots}
\pgfplotsset{compat=1.16}
\usetikzlibrary{arrows}

\usepackage{xcolor}
\definecolor{newred}{RGB}{200, 30, 45}
\definecolor{newyellow}{RGB}{255, 222, 23}
\definecolor{newblue}{RGB}{25, 48, 126}
\definecolor{newgrey}{RGB}{115, 110, 95}
\definecolor{newbrown}{RGB}{160, 62, 45}

\usepackage{hyperref}
\hypersetup{colorlinks, linkcolor={newred}, citecolor={newbrown}, 
 pdftitle={Affine Deformations of Divisible Convex Cones and Affine Spacetimes},
	pdfauthor={Antoine Ablondi}}
\usepackage{pdfpages}

\newtheorem{maintheorem}{Theorem}

\newtheorem{maincorollary}[maintheorem]{Corollary}

\newtheorem{theorem}{Theorem}[section]
\newtheorem*{theorem*}{Theorem}
\newtheorem{corollary}[theorem]{Corollary}
\newtheorem{lemma}[theorem]{Lemma}
\newtheorem{proposition}[theorem]{Proposition}

\theoremstyle{definition}
\newtheorem{definition}[theorem]{Definition}

\theoremstyle{remark}
\newtheorem{remark}[theorem]{Remark}

\newtheorem*{remark*}{Remark}

\usepackage[justification=centering,labelsep=period]{caption}
\numberwithin{equation}{section}
\numberwithin{figure}{section}

\newcommand{\R}{\mathbb{R}} 
\newcommand{\N}{\mathbb{N}} 
\newcommand{\V}{\mathbb{V}} 
\newcommand{\A}{\mathbb{A}} 
\renewcommand{\P}{\mathbb{P}} 
\renewcommand{\H}{\mathbb{H}} 
\newcommand{\M}{\mathbb{M}} 
\newcommand{\B}{\mathbb{B}} 
\newcommand{\Sph}{\mathbb{S}} 
\newcommand{\Id}{\mathrm{Id}} 

\newcommand{\dif}{\mathrm{d}} 
\DeclareMathOperator{\grad}{grad} 
\DeclareMathOperator{\Hess}{Hess} 

\DeclareMathOperator{\graph}{graph} 
\DeclareMathOperator{\epi}{epigraph} 
\DeclareMathOperator{\dom}{dom} 

\newcommand{\SL}{\mathrm{SL}} 
\newcommand{\SO}{\mathrm{SO}} 
\newcommand{\SA}{\mathrm{SA}} 
\newcommand{\PGL}{\mathrm{PGL}} 

\newcommand{\C}{\mathcal{C}} 
\renewcommand{\S}{\Sigma_\C} 
\newcommand{\G}{\mathcal{G}} 

\newcommand{\vol}{\mathrm{vol}} 

\newcommand{\dev}{\mathrm{dev}} 

\newcommand{\Pg}{\mathcal{P}_g} 
\newcommand{\Pgaff}{\widehat{\mathcal{P}_g}} 

\newcommand{\st}{\; \middle\vert \;} 

\newcommand{\namelessfunction}[4]{ \left\lbrace \begin{matrix}
{#1} &\longrightarrow& {#2}\\
{#3} & \longmapsto & {#4}
\end{matrix} \right. }
\newcommand{\function}[5]{{#1}: \namelessfunction{#2}{#3}{#4}{#5}}

\newcommand{\namelessembedding}[4]{ \left\lbrace \begin{matrix}
{#1} & \lhook\joinrel\longrightarrow& {#2}\\
{#3} & \longmapsto & {#4}
\end{matrix} \right. }

\newcommand{\vp}[1]{{\left( {#1} \right)}} 
\newcommand{\bp}[1]{{\bigl( {#1} \bigr)}} 

\begin{document}

\begin{abstract}
Let $G$ be a subgroup of $\SL (\R^{d+1}) \ltimes \R^{d+1}$ obtained by adding a translation part to a torsion-free discrete subgroup of $\SL (\R^{d+1})$ dividing a convex cone in the sense of Benoist. We consider the maximal convex domains in $\R^{d+1}$ on which the affine action of $G$ is free and properly discontinuous, and show its quotient by $G$ is naturally endowed with an ‘‘affine spacetime'' structure, which is a generalisation of the notion of flat Lorentzian spacetime. More precisely, we show that this quotient is a Maximal Globally Hyperbolic affine spacetimes admitting a $C^2$ locally uniformly Convex and Compact Cauchy surface (denoted as a MGHCC affine spacetimes), and that it comes with a cosmological time function with Cauchy hypersurfaces foliating the quotient affine spacetime as level sets. Finally, we show such quotients are the only examples of MGHCC affine spacetimes. All these results generalise the work of Mess, Barbot and Bonsante on affine deformations of uniform lattices of $\SO_0(d,1)$.
\end{abstract}

\maketitle

{\hypersetup{linkcolor=black} \hypersetup{bookmarksdepth=3} \setcounter{tocdepth}{1} \tableofcontents}

\begin{spacing}{1.2}

\section*{Introduction}

In this article we study some particular discrete subgroups of $\SA(\R^{d+1}) = \SL(\R^{d+1}) \ltimes \R^{d+1}$, the group of \emph{special affine transformations of} $\R^{d+1}$. Our interest for such groups comes from previous works in two different settings generalising uniform lattices of $\SO_0(d,1)$.

\subsection*{\texorpdfstring{Affine deformations of uniform lattices of $\SO_0(d,1)$}{Affine deformations of uniform lattices of SO(d,1)}}

Let $d \geq 2$. The $(d + 1)$-dimensional \emph{Minkowski vector space} $\R^{d,1}$ is the oriented vector space $\R^{d+1}$ endowed with a non-degenerate bilinear form $\langle\cdot,\cdot \rangle_{d,1}$ of signature $(d,1)$ and a \emph{time orientation on $\R^{d,1}$}, i.e. a favourite component  $\C^{d,1}$ among the two connected component of the cone of non-trivial timelike vectors, called the \emph{future cone}. The bilinear form $\langle\cdot,\cdot \rangle_{d,1}$ induces a Riemannian metric on $\H^{d} \subset \R^{d,1}$, the one-sheeted hyperboloid of all unit future timelike vectors, giving the classical hyperboloid model of the $d$-dimensional \emph{hyperbolic space}. A model for the group of orientation preserving isometries of the hyperbolic space $\H^d$ is then given by $\SO_0(d,1)$, the group of linear transformation of $\R^{d,1}$ preserving its orientation, its bilinear form and its time orientation.

The $(d + 1)$-dimensional \emph{Minkowski affine space} $\M^{d,1}$ is the oriented real affine space endowed with a flat Lorentzian metric and a time orientation. Thus, its vector space of translations is the Minkowski vector space $\R^{d,1}$, and a model for $\M^{d,1}$ is the manifold $\R^{d+1}$ (seen as an affine space) endowed with the flat Lorentzian metric induced by the bilinear form $\langle\cdot,\cdot \rangle_{d,1}$. Its group of orientation and time-orientation preserving isometries is $\SO_0(d,1) \ltimes \R^{d,1}$, the group of affine transformations of $\M^{d,1}$ having linear part in $\SO_0(d,1)$.

A \emph{uniform lattice of $\SO_0(d,1)$} is discrete subgroup $\Lambda < \SO_0(d,1)$ such that the quotient $\H^d/\Lambda$ is compact. Thus, torsion-free uniform lattices of $\SO_0(d,1)$ are precisely the holonomy groups of closed hyperbolic manifolds. An \emph{affine deformation of a uniform lattice of $\SO_0(d,1)$} is a subgroup of $\SO_0(d,1) \ltimes \R^{d+1}$ obtained by adding translation parts to a uniform lattice. In particular, any uniform lattice of $\SO_0(d,1)$ can be seen as an affine deformation of itself (by adding a trivial translation part).

Those subgroups of affine transformations have been the focus of the work of Mess \cite{Mess_07,Mess_notes} for $d=2$, as well as the ones of Barbot \cite{Barbot_05} and Bonsante \cite{Bonsante_05} in higher dimension. They have proved a correspondence between affine deformations of torsion-free uniform lattices of $\SO_0(d,1)$ and \emph{flat maximal globally hyperbolic spacetimes (denoted MGHC) with hyperbolic compact Cauchy surfaces}. They are time-oriented flat Lorentzian manifolds $M$ admitting a Cauchy surface $S$, that is an hypersurface intersecting every inextensible causal curve exactly once, such that $S$ is compact and admits a hyperbolic metric, and $M$ satisfies some maximality condition (see Definition~\ref{def Maximal}).

When $d=2$, that setting is closely related to \emph{classical Teichm\"uller theory}. Indeed, if $S_g$ is a closed oriented and connected surface of genus $g \geq 2$, its \emph{Teichm\"uller space} $ \mathcal{T}(S_g)$ can be seen as a connected component of the character variety
\begin{equation*}
\chi \bp{ \pi_1(S_g) , \SO_0(2,1)} = \text{Hom}\bp{\pi_1(S_g),\SO_0(2,1)} \big/ \SO_0(2,1) \, , 
\end{equation*}
where the quotient is by conjugation. More precisely, $ \mathcal{T}(S_g)$ is diffeomorphic to $\chi^{fd}( \pi_1(S_g) , \SO_0(2,1))$ one of the two connected components of the character variety consisting of conjugacy classes of faithful and discrete representations \cite{Goldman_thesis}, i.e. representations $\rho$ such that $\rho(\pi_1(S_g))$ is a uniform lattice of $\SO_0(2,1)$.

Before describing the other setting generalising uniform lattices of $\SO_0(d,1)$, let us remark that more recent works \cite{Fillastre_Veronelli_16,Bonsante_Fillastre_17, Fillastre_Seppi_19, Barbot_Fillastre_20} have shown that tools from \emph{convex geometry} inside the \emph{affine Minkowski space $\M^{d,1}$} (on which flat Lorentzian spacetimes are locally modelled) were very useful in the study of those flat MGHC Lorentzian spacetimes with hyperbolic Cauchy surface. In these works, the use of the Lorentzian metric of the affine Minkowski space $\M^{d,1}$ is not much more relevant than its \emph{causal} and \emph{affine structure}. 

\subsection*{Divisible convex sets}

Another generalisation of uniform lattices of $\SO_0(d,1)$ is given by \emph{divisible convex sets} in the sense of Benoist \cite{Benoist_04}. An open proper (i.e. non-empty and with closure not containing any entire line) convex cone $\C$ in the vector space $\R^d$ is \emph{divisible} by a
discrete torsion-free subgroup $\Gamma$ of $\SL(\R^{d+1})$ if $\Gamma$ preserves $\C$ and the projective action of $\Gamma$ on $\P(\C)$, the projectivisation of the cone, is properly discontinuous and cocompact. Equivalently, we also say that the convex projective domain $\P(\C)$ is divisible by the subgroup $\Gamma_{\P} < \PGL(\R^{d+1})$ induced by $\Gamma$. In particular, every uniform lattice of $\SO_0(d,1)$ divides the cone $\C^{d,1}$ of non-trivial future timelike vectors in $\R^{d,1}$. 

While torsion-free uniform lattices of $\SO_0(d,1)$ are holonomy groups of closed hyperbolic manifolds, subgroups of $\SL(\R^{d+1})$ dividing cones  are holonomy groups of closed manifolds admitting a \emph{convex projective structure} \cite{Benoist_08}. 

Benoist \cite{Benoist_05} has highlighted that divisibility gives a geometric interpretation to some connected components of higher character varieties (higher in the sense that they have images in higher rank Lie groups, such as $\SL(\R^{d+1})$, but also as they concern representations of fundamental groups of manifolds of dimension $d \geq 2$). Letting $\Gamma_0$ be a finitely generated group, he considered $\mathrm{Hom}_{\mathrm{div}} (\Gamma_0, \SL(\R^{d+1}))$ the space of faithful group representations $\rho \in \mathrm{Hom}(\Gamma_0, \SL(\R^{d+1}))$ such that there is an open proper convex cone $\C_\rho \subset \R^{d+1}$ divisible by $\rho(\Gamma_0)$. He proved that if $\mathrm{Hom}_{\mathrm{div}} (\Gamma_0, \SL(\R^{d+1}))$ admits a strongly irreducible representation then it consists of a union of connected components of $\mathrm{Hom}(\Gamma_0, \SL(\R^{d+1}))$.

In the particular case where $d=2$, divisibility is related with projective structure on surfaces. Let $S_g$ be a closed oriented and connected surface of genus $g \geq 2$. By the work of Choi and Goldman \cite{Choi_Goldman_93}, the \emph{moduli space $\Pg$ of convex projective structures on $S_g$} is identified with the \emph{Hitchin component} \cite{Hitchin_92} of the representation space of $\pi_1(S_g)$ into $\SL(\R^3)$, and it can be described as 
\begin{equation*}
\Pg = \mathrm{Hom}_{\mathrm{div}} \bp{\pi_1(S_g), \SL(\R^3)} \big/\SL(\R^3) \, ,
\end{equation*}
where the quotient is by conjugation. 

Labourie \cite{Labourie_07} and Loftin \cite{Loftin_01} have both independently used the Choi--Goldman Theorem to get a holomorphic parametrisation of $\Pg$. It can be recovered from the theory of \emph{affine differential geometry}, using the \emph{affine invariants} of the unique \emph{Cheng--Yau affine sphere} asymptotic to a divisible cone.

\subsection*{An affine generalisation of the Mess--Barbot--Bonsante correspondence}

In this article, we shall consider \emph{affine deformations of groups dividing a cone}. That is, subgroups of $\SL(\R^{d+1}) \ltimes \R^{d+1}$ obtained by adding a translation part to a discrete torsion-free subgroup of $\SL(\R^{d+1})$ dividing some open proper convex cone in $\R^{d+1}$. Our goal is to characterise the geometric structures associated to those subgroups, extending the Mess--Barbot--Bonsante correspondence.

We introduce and study \emph{affine spacetimes}, which are a natural generalisation of flat Lorentzian spacetimes. Indeed, we shall define a $(d+1)$-dimensional affine spacetime $(M,\C)$ as a manifold $M$ endowed with both an \emph{equiaffine structure}, that is a $(G,X)$-structure where $G = \SL (\R^{d+1}) \ltimes \R^{d+1}$ is the group of orientation and volume preserving affine transformations acting on the real affine space $X = \R^{d+1}$, and a \emph{parallel distribution of open proper convex cones} $\C$ in its tangent bundle.

The most basic example of an affine spacetime is $(\R^{d+1}, \C)$, the real affine space $\R^{d+1}$ endowed with the parallel cone distribution given at every point $p \in \R^{d+1}$ by the same proper convex cone $\C \subset \R^{d+1} =\mathrm{T}_p\R^{d+1}$. By definition, every affine spacetime is actually locally modelled on $(\R^{d+1}, \C)$ for some proper convex cone $\C$, and the linear part of a transition map between two affine charts of its equiaffine structure is an element of $\SL(\R^{d+1})$ preserving $\C$.

In particular, flat Lorentzian spacetimes are naturally affine spacetimes: they are locally modelled on the Minkowski space $\M^{d,1}$, so they admit a $(\SO_0(d,1) \ltimes \R^{d,1}, \M^{d,1})$-structure and a natural parallel cone distribution given by the interiors of the future cones in the tangent spaces $\mathrm{T}_p\M^{d,1}=\R^{d,1}$. 

All of the definitions and notions arising purely from the causal structure of Lorentzian spacetimes naturally extend to that new setting (see subsection~\ref{subsec causality}). If $(M,\C)$ is an affine spacetime, a vector $v \in \mathrm{T}_pM$ is called \emph{$\C$-timelike} if $v \in \C_p \cup (-\C_p)$, \emph{$\C$-lightlike} if $v \in \partial \C_p \cup (-\partial\C_p)$, and \emph{$\C$-spacelike} if $v \notin \overline{\C_p} \cup(- \overline{\C_p})$. Then, a $\C$-causal curve is a curve having only $\C$-timelike or $\C$-spacelike velocity vector. Hence, one can again consider \emph{global hyperbolic} affine spacetimes: an affine spacetime $(M,\C)$ is \emph{globally hyperbolic} if it admits a $C^1$ \emph{Cauchy surface}, that is an embedded $C^1$ hypersurface intersecting every inextensible future oriented $\C$-causal curve at exactly one point. 

Note that we can also naturally define \emph{$\C$-spacelike} and \emph{$\C$-null} hyperplanes inside the affine spacetime $(\R^{d+1},\C)$, which is an affine space. Thus, we proceed to generalise notions from \emph{convex geometry inside the Minkowski affine space} \cite{Fillastre_13,Fillastre_Veronelli_16,Bonsante_Fillastre_17}. A \emph{$\C$-convex domain} $K \subset \R^{d+1}$ (see Definition~\ref{def C-convex} and Remark~\ref{rem asymptotic cone}) is an open convex set having \emph{asymptotic cone} equal to $\overline{\C}$, the closure of $\C$, i.e. such that
\begin{equation*}
 \left\lbrace V \in \R^{d+1} \st \forall P \in K, \ P+ \R_+ V \subseteq K \right\rbrace = \overline{\C} \, .
\end{equation*}

We then prove the following theorem generalising the work of Mess \cite{Mess_07,Mess_notes}, Barbot \cite{Barbot_05}, and Bonsante \cite{Bonsante_05}, and providing many examples of maximal globally hyperbolic affine spacetime.

\begin{maintheorem}
\label{Theorem Intro domain}
Let $\Gamma_\tau < \SL (\R^{d+1}) \ltimes \R^{d+1}$ be a subgroup obtained by adding translations part to a discrete torsion-free subgroup $\Gamma < \SL(\R^{d+1})$ dividing an open proper convex cone $\C \subset \R^{d+1}$. Then, there exists a unique maximal $\Gamma_\tau$-invariant $\C$-convex domain $D_\tau \subset \R^{d+1}$. It satisfies the following properties:
\begin{enumerate}[label=(\arabic*)]
\item \label{theo11} The action of $\Gamma_\tau$ on $D_\tau$ is free and properly discontinuous, and there is a homeomorphism
\begin{equation*}
D_\tau /\Gamma_\tau \simeq \P(\C)/\Gamma \times \R \, .
\end{equation*}
\item \label{theo12} The affine spacetime $(D_\tau / \Gamma_\tau,\C)$ is a maximal globally hyperbolic affine spacetime admitting a $C^2$ locally uniformly future-convex and compact Cauchy surface (denoted as MGHCC affine spacetime).
\end{enumerate}
\end{maintheorem}

\begin{remark}
In Theorem~\ref{Theorem Intro domain}, a Cauchy surface $S$ of an affine spacetime $(M,C)$ is said to be $C^2$ locally uniformly future-convex if in every affine chart of the equiaffine manifold $M$, the hypersurface $S$ is locally the graph of a convex function with positive definite Hessian, and every $\C$-future vectors on $S$ points towards its convex side. 
\end{remark}

\begin{remark}
The existence and uniqueness of the maximal $\Gamma_\tau$-invariant $\C$-convex domain $D_\tau$, on which $\Gamma_\tau$ acts properly discontinuously was already proved in the work of Choi \cite[Theorems 4.3.1 and 4.3.7]{Choi_25} in a more general setting and with different language and approach. In dimension $2+1$, that, as well as the topology of the quotient $D_\tau /\Gamma_\tau$, was also independently shown by Nie and Seppi \cite[Theorem~A]{Nie_Seppi_23}. The topology of the quotient in higher dimension, the notion of affine spacetime, the whole point \ref{theo12}, as well as Theorems~\ref{Theorem Intro dev} and \ref{Theorem Intro cosmological time} below are completely new to the best of our knowledge. 
\end{remark}

Note that Theorem~\ref{Theorem Intro domain} can also be applied using the cone $-\C$, which is also divisible by $\Gamma$, in order to produce another domain and MGHCC affine spacetime.

In order to complete our generalisation of the correspondence of Mess, Barbot, and Bonsante, we also prove the following theorem, stating that these examples of MGHCC affine spacetimes given by Theorem~\ref{Theorem Intro domain} happen to be the only ones.

\begin{maintheorem}
\label{Theorem Intro dev}
Let $(M,\C^M)$ be MGHCC affine spacetime. Any pair $(\dev,\rho)$ of a developing map
\begin{equation*}
\dev: \widetilde{M} \to \R^{d+1} \, ,
\end{equation*}
and a holonomy map 
\begin{equation*}
\rho: \pi_1M \to \SL (\R^{d+1}) \ltimes \R^{d+1} \, ,
\end{equation*}
for the equiaffine structure of $M$ satisfies the following:
\begin{itemize}
\item $\rho$ is injective, i.e. it is a group isomorphism onto its image $ \rho(\pi_1M)$, 
\item $\rho(\pi_1M)$ is an affine deformation $\Gamma_\tau$ of its linear part $\Gamma < \SL (\R^{d+1})$ which divides the open proper convex cone $\C \coloneqq \dif \, \dev(\C^M)$,
\item $\dev$ is an isomorphism of affine spacetimes onto its image which is $(D_\tau,\C)$, where $D_\tau $, the domain associated with $\Gamma_\tau \coloneqq \rho(\pi_1M)$ and $\C$, is given by Theorem~\ref{Theorem Intro domain}.
\end{itemize}
\end{maintheorem}

\begin{remark}
In Theorem~\ref{Theorem Intro dev}, an isomorphism of affine spacetimes $\Psi:(M,\C^M) \to (N,\C^N)$ is an isomorphism of equiaffine structures such that $\dif \Psi(\C^M)=\C^N$. 
\end{remark}

Theorems~\ref{Theorem Intro domain} and \ref{Theorem Intro dev} thus give a complete classification of MGHCC affine spacetimes via affine deformations of groups dividing cones.

A group $\Gamma < \SL(\R^{d+1})$ dividing a convex cone does not always admit a non-trivial affine deformation, that is, an affine deformation which is not a conjugation of $\Gamma$ by a translation. When $d=2$, Nie and Seppi \cite{Nie_Seppi_23} have proved that if $\Gamma$ is isomorphic to the fundamental group of the closed connected surface of genus $g \geq 2$, it always does. Concerning higher dimensions, we shall give examples of non-trivial affine deformations of subgroups dividing a cone and which are not conjugated to a uniform lattice of $\SO_0(d,1)$ (see Appendix~\ref{sec Examples of affine deformations}).

Adopting a projective approach , we shall also highlight that the discrete subgroups $\Gamma_\tau$ considered in Theorem~\ref{Theorem Intro domain} have a \emph{convex cocompact action} on an open subset of $\P (\R^{d+2})$ (see Appendix~\ref{sec Projective approach}). When $\Gamma$ is \emph{word hyperbolic}, those are moreover \emph{strongly convex cocompact} in the sense of Danciger, Gu\'eritaud, and Kassel \cite{DGK_17}, and thus, following from their work, are also \emph{projective Anosov}.

\subsection*{Affine sphere cosmological time}

A central tool in the study of MGHCC affine spacetimes and especially in the proof of point~\ref{theo11} in Theorem~\ref{Theorem Intro domain} is the \emph{cosmological time function} on $\C$-convex domains.

Affine spacetimes significantly differ from flat Lorentzian spacetimes as their cone distribution does not come from the distribution of future cones of a Lorentzian metric. Nevertheless, we shall introduce a way to define length on the family of causal curves, satisfying, among other properties, a \emph{time inequality}. That approach is similar to recent works concerning metric generalisations of Lorentzian geometry, based on the seminal work of Busemann \cite{Busemann_67} on the axiomatisation of \emph{timelike spaces}. We refer, for example, to the work of Kunziger and S\"amann \cite{Kunziger_Samann_18} concerning \emph{Lorentzian length spaces}, or the work of Papadopoulos and Yamada \cite{Papadopoulos_Yamada_19} and Buro \cite{Buro_23}, concerning \emph{pseudo-Finsler geometry}. 

Then, using that length function, a \emph{cosmological time function} (see Section~\ref{sec Cosmological time on C-convex domains}) can be defined on $\C$-convex domains of the model affine spacetime $(\R^{d+1},\C)$. Its definition can be taken to be very similar to the usual one for Lorentzian spacetimes \cite{Andersson_Galloway_Howard_98} (see Proposition~\ref{prop cosmological time and length causal curve}). We then prove the following fact. 

\begin{maintheorem}
\label{Theorem Intro cosmological time}
Let $\C$ be an open proper convex cone in the vector space $\R^{d+1}$. Every $\C$-convex domain $K$ in the affine spacetime $(\R^{d+1},\C)$ admits a concave $C^1$ cosmological time function $\mathcal{T}_K: K \to \R^*_+$ whose level sets are $C^1$ future-convex Cauchy surfaces of $(K,\C)$ foliating $K$. 
Moreover, $\mathcal{T}_K$ is invariant by any transformation of $\SA(\R^{d+1}) = \SL(\R^{d+1}) \ltimes \R^{d+1}$ preserving $K$.
\end{maintheorem}

The cosmological time on a domain $D_\tau$ from Theorem~\ref{Theorem Intro domain} is thus $\Gamma_\tau$-invariant and descends to the quotient affine spacetime $(D_\tau / \Gamma_\tau,\C)$. That MGHCC affine spacetime is thus foliated by the level hypersurfaces of its cosmological time, which are $C^1$ embedded future-convex Cauchy surfaces. Using Theorems~\ref{Theorem Intro domain},~\ref{Theorem Intro dev}, and~\ref{Theorem Intro cosmological time}, that gives the following.

\begin{maincorollary}
\label{Corollary Intro MGHCC cosmological time}
Any MGHCC affine spacetime $(M,\C)$ is naturally endowed with a $C^1$ locally concave cosmological time $\mathcal{T}:M\to \R^*_+$. The level hypersurfaces of this cosmological time are $C^1$ $\C$-spacelike future-convex Cauchy surfaces foliating $M$.
\end{maincorollary}

\begin{remark}
In Corollary~\ref{Corollary Intro MGHCC cosmological time}, locally concave means that in every affine chart of the equiaffine manifold $M$, the restriction of $\mathcal{T}$ is a concave function. 
\end{remark}

We actually use the unique \emph{Cheng--Yau affine sphere} associated with the proper convex cone $\C$ \cite{Cheng_Yau_86} in order to normalise $\C$-timelike vectors, measure lengths of $\C$-timelike curves, and produce the cosmological time function from Theorem~\ref{Theorem Intro cosmological time}. Here are some motivation for that choice.

First, let us note that it is truly a generalisation of the Lorentzian case: in any tangent space of a Lorentzian spacetime, timelike vectors are normalised by the hyperboloid which is the Cheng--Yau affine sphere associated with the isotropic cone of the Lorentzian metric. In the general case, the work of Minguzzi \cite{Minguzzi_17} provides a physical interpretation to that affine sphere renormalisation.

Moreover, the unique Cheng--Yau affine sphere associated with a proper convex cone $\C$ is invariant under the subgroup $\mathrm{Aut}_\SL(\C)$ of elements of $\SL(\R^{d+1})$ preserving $\C$. That provides the invariance property of the cosmological time function in Theorem~\ref{Theorem Intro cosmological time}. Those affine spheres also behave well with the duality and convexity tools that were used in the Lorentzian case \cite{Fillastre_Veronelli_16,Bonsante_Fillastre_17}. Indeed, fundamental tools from convex analysis, such as \emph{duality, subdifferentials} or the \emph{Legendre--Fenchel transform}, were already used by Cheng and Yau \cite{Cheng_Yau_77,Cheng_Yau_86} in order to prove their fundamental theorem about existence and uniqueness of a unit complete hyperbolic affine sphere asymptotic to any proper convex cone. Actually, substituting those hyperbolic affine spheres for the sphere or the hyperboloid in, respectively, Euclidean or Minkowski convex geometry \cite{Schneider_93,Fillastre_13}, one can also naturally define and study an affine version of the Minkowski problem, generalising results in the Minkowski space \cite{Bonsante_Fillastre_17}. In that setting, the choice of affine sphere will allow to recover some results about surfaces with \emph{constant affine Gaussian curvature} \cite{Labourie_07,Nie_Seppi_22}. We address this problem in  \cite{Ablondi_prep}.

Finally, as the Labourie--Loftin parametrisation of the Hitchin component has a geometric interpretation through the use of Cheng--Yau affine sphere, we can also hope, at least when $d=2$, for some higher Teichm\"uller theory translation of the tools introduced in this article.

Note that one could use any other locally convex hypersurfaces invariant by special linear automorphisms of the cone and with similar properties (strict convexity, smoothness, asymptotic behaviour) in order to produce a cosmological time for $\C$-convex domains. Another example of such an   hypersurface is given by the level sets of the \emph{Vinberg functional} \cite{Vinberg_63,Goldman_22}. 

\subsection*{Relation with higher Teichm\"uller theory}

We conclude this introduction with some remarks on the case where $d=2$.

Let $S_g$ be a closed oriented and connected surface of genus $g \geq 2$. The \emph{moduli space $\Pgaff$ of affine deformations of convex projective structures on $S_g$}, can be described as 
\begin{equation*}
\Pgaff = \mathrm{Hom}_{\mathrm{div}} \bp{\pi_1(S_g) , \SL(\R^3) \ltimes \R^3} \big/ \bp{\SL(\R^3) \ltimes \R^3 } \, ,
\end{equation*}
where the quotient is by conjugation and $\mathrm{Hom}_{\mathrm{div}} (\pi_1(S_g) , \SL(\R^3) \ltimes \R^3)$ is the set of group representations $\rho_\tau \in \mathrm{Hom}(\pi_1(S_g) , \SL(\R^3) \ltimes \R^3)$ with linear part $\rho \in \mathrm{Hom}_{\mathrm{div}}(\pi_1(S_g) , \SL(\R^3))$.

Recently, Nie and Seppi \cite{Nie_Seppi_22} have highlighted how $\Pgaff$ is naturally a topological vector bundle of rank $6g-6$ on the $(16g-16)$-dimensional Hitchin component given by the natural projection 
\begin{equation*}
\namelessfunction{\mathrm{Hom}_{\mathrm{div}} \bp{\pi_1(S_g) , \SL(\R^3) \ltimes \R^3}}{\mathrm{Hom}_{\mathrm{div}} \bp{\pi_1(S_g) , \SL(\R^3)}}{\rho_\tau}{\rho} \, .
\end{equation*}
Through the special identification between the Lie algebra $\mathfrak{so}(2,1)$ and the Minkowski space $\R^{2,1}$, affine deformations of Fuchsian subgroups of $\SO_0(2,1)$ can be seen as first order derivatives of variation of hyperbolic structures on $S_g$. Then the part of $\Pgaff$ over the Teichm\"uller space $\mathcal{T}_g \subset \Pg$, studied by Mess \cite{Mess_07}, can be identified with the tangent bundle of $\mathcal{T}_g$. Theorem~\ref{Theorem Intro domain} and Theorem~\ref{Theorem Intro dev} thus give a correspondence between each element of $\Pgaff$ and two (maybe equal) isomorphisms classes of $(2+1)$-dimensional MGHCC affine spacetimes with Cauchy surfaces homeomorphic to $S_g$. In the Minkowski case \cite{Mess_07,Bonsante_05}, one makes this correspondence one to one using the time orientation of the Minkowski space. In the general case, as the divisible cones vary, one cannot make such a simple choice.

In \cite[Section 7]{Labourie_07}, Labourie explicitly gives two different parametrisations of the moduli space $\Pgaff$, as well as some symmetries on it. Using the correspondence presented in this article, and a geometric study of the $(2+1)$-dimensional MGHCC affine spacetimes associated with elements of $\Pgaff$, we can hope for new geometric counterparts and explanations to Labourie's work (similarly to previous works on the Teichm\"uller case \cite{Bonsante_Seppi_16,Barbot_Fillastre_20}). 

It is also compelling to try to link our approach to the recent work of Bobb and Farre \cite{Bobb_Farre_24} concerning the generalisation of the part of the Mess correspondence concerning measured geodesic laminations. It could be possible to make those works meet using convex geometry and area measures associated with $\C$-convex domain, similarly to what was done in the Minkowski case \cite{Fillastre_Veronelli_16,Bonsante_Fillastre_17,Barbot_Fillastre_20}.

The cosmological time on MGHCC affine spacetimes, given by Theorems~\ref{Theorem Intro cosmological time} and Corollary~\ref{Corollary Intro MGHCC cosmological time} is a natural generalisation of the classical one on MGHCC Lorentzian spacetimes which happens to be related to \emph{hyperbolic grafting}. Indeed, by the work of Benedetti and Bonsante \cite{Benedetti_Bonsante_09}, if $\rho$ is a Fuchsian representation of $\pi_1(S_g)$ and $\Gamma_\tau$ is an affine deformation of $\Gamma = \rho(\pi_1S_g) < \SO_0(2,1)$, the level set $\{\mathcal{T} = 1\}$ of the cosmological time on $D_\tau/\Gamma_\tau$ is the $C^1$ Riemannian surface obtained by the grafting of the hyperbolic surface $\H^{2}/\Gamma$ along a measured geodesic lamination associated with the affine deformation $\Gamma_\tau$. Thus, we can hope for a generalised grafting process for convex projective surfaces. As we define cosmological time through affine spheres, we could hope that such a generalised grafting process would have some nice behaviour in the Labourie--Loftin parametrisation of the $\SL(\R^3)$ Hitchin component.

\subsection*{Organisation of the article} 

We first introduce the definition of an affine spacetime in Section~\ref{Sec Affine spacetimes}. We then give a few useful results and definitions from affine differential geometry and convex analysis in respectively Section~\ref{sec Background ADG} and Section~\ref{sec Background CG}. Section~\ref{sec About C-convex domains} is dedicated to the study of the special case of the model affine spacetime $(\R^{d+1}, \C)$: the real affine space $\R^{d+1}$ endowed with a parallel cone structure given by an open proper convex cone $\C$. In Section~\ref{sec Affine deformations of divisible convex cones}, we focus on the case where $\C$ is divisible by a subgroup $\Gamma$ of $\SL(\R^{d+1})$. In that setting, we build for any affine deformation of $\Gamma$ the maximal invariant domain $D_\tau$ from Theorem~\ref{Theorem Intro domain}. Within Section~\ref{sec Cosmological time on C-convex domains}, we introduce and study a cosmological time function on $\C$-convex domains, and give its geometric interpretation in the context of pseudo-Finsler geometry. That allows us, in Section~\ref{sec Quotient of the maximal domain}, to study the spacetime structure of the quotients $D_\tau / \Gamma_\tau$ and prove point~\ref{theo11} and the GHCC part of point~\ref{theo12} of Theorem~\ref{Theorem Intro domain}, as well as Theorem~\ref{Theorem Intro cosmological time}. Finally, Section~\ref{sec Correspondence between affine deformations of dividing groups and affine spacetimes} consists of the proofs of the maximality part of point~\ref{theo12} in Theorem~\ref{Theorem Intro domain} and of Theorem~\ref{Theorem Intro dev}, thus extending the Mess correspondence. 

In addition, Appendix~\ref{sec Examples of affine deformations} gives examples in any dimension of groups dividing an open proper convex domains not projectively equivalent to an ellipse (giving the Klein ball model of the hyperbolic space), and admitting non-trivial affine deformations, while Appendix~\ref{sec Projective approach} links the work of this article to the notion of convex cocompact representation through a more projective approach to affine deformations.

\subsection*{Acknowledgements}
I am deeply grateful to my two doctoral advisors Fran\c{c}ois Fillastre and Andrea Seppi for introducing me to this subject, helping me in my mathematical work and guiding me through writing this article. I would like to thank Thierry Barbot for interesting discussions about affine spacetimes, Martin Bobb and James Farre for introducing me to the projective approach described in Appendix~\ref{sec Projective approach}, and Colin Davalo for his helpful indications concerning Lemma~\ref{lem inclusion and divisibility}. I also thank Suhyoung Choi for noticing and explaining me how some of the results presented in this article are related to his work.

\section{Affine spacetimes}
\label{Sec Affine spacetimes}

In this section we first present the notion of \emph{equiaffine manifold } and the causality theory given by a \emph{cone structure} on a manifold. We then use those in order to introduce the concept of \emph{affine spacetime}, generalising flat Lorentzian spacetimes.

\subsection{Equiaffine manifolds}
\label{subsec Equiaffine manifolds}

In this whole article the setting will be the following. 

Let $d \geq 2$ be an integer. Let $\V^{d+1}$ be the $(d+1)$-dimensional real vector space endowed with an orientation and a linear volume form $vol$. Let $\A^{d+1}$ be the $(d+1)$-dimensional real affine space with $\V^{d+1}$ as its vector space of translations. It is endowed with an orientation, a flat affine connection $D$ and a parallel differential volume form $\dif\vol$, induced by a linear volume form $vol$ on $\V^{d+1}$.

A \emph{real equiaffine structure}, or simply \emph{equiaffine structure}, on a smooth $(d+1)$-dimensional manifold $M$ consists of an atlas of charts with values in the affine space $\A^{d+1}$, such that the transition maps are restrictions of orientation and volume preserving affine transformations of $\A^{d+1}$. Using the language of $(G, X)$-structures, an equiaffine structure on $M$ is a $(\SA(\A^{d+1}),\A^{d+1})$-structure, where $\SA(\A^{d+1})$ is the group of special affine transformations of $\A^{d+1}$, that is the set of volume preserving affine automorphism of $(\A^{d+1}, \dif \vol)$.

A manifold endowed with an equiaffine structure is called an \emph{equiaffine manifold}. It is naturally endowed with an orientation, a flat torsion-free affine connection $\nabla$ (induced by the flat affine connection $D$ on $\A^{d+1}$) and a volume form $\omega$ (induced by the volume form $\dif \vol$ on $\A^{d+1}$) parallel for the connection $\nabla$, i.e. satisfying $\nabla \omega = 0$.

Equivalently, without the language of $(G, X)$-structures, an equiaffine manifold is a triplet $(M, \nabla, \omega)$ where $M$ is an orientable manifold endowed with a flat torsion-free affine connection $\nabla$ and a volume form $\omega$ such that $\nabla \omega = 0$.

If $M$ is an equiaffine manifold with universal cover $\widetilde{M}$, the theory of $(G, X)$-structures tells us that one can fix a pair of \emph{developing map} 
\begin{equation*}
\dev: \widetilde{M} \to \A^{d+1}
\end{equation*}
and \emph{holonomy representation}
\begin{equation*}
 \rho: \pi_1 M \to \SA(\A^{d+1}) \, ,
\end{equation*}
satisfying the following:
\begin{itemize}
\item the developing map $\dev$ is $\rho$-equivariant, in the sense that for all $p \in \widetilde{M}$ and $\gamma \in \pi_1M$,
\begin{equation*}
 \dev(\gamma \cdot p) = \rho(\gamma) \, \dev(p) \, ,
\end{equation*}
\item the developing map $\dev$ is a local diffeomorphism preserving the equiaffine structure, i.e. on each chart of the equiaffine structure of $\widetilde{M}$, $\dev$ is induced by an element of $\SA(\A^{d+1})$.
\end{itemize}
Moreover, given two such pairs $(\dev,\rho)$ and $(\dev',\rho')$ there exists a unique $g \in \SA(\A^{d+1})$ such that
\begin{equation*}
 \dev' = g \, \dev \quad \text{and} \quad \rho' = g \, \rho \, g^{-1} \, .
\end{equation*}

\subsection{Cone structures and causality}
\label{subsec causality}

\begin{definition}[Proper convex cone]
A \emph{cone} in the real vector space $\V^{d+1}$ is a subset $\C \subseteq \V^{d+1}$ such that for all $v \in \C$ and $\lambda>0$, one has $\lambda v \in \C$.

A cone $\C \subset \V^{d+1}$ is \emph{convex} if for all $v,w \in \C$ and $t \in [0,1]$ , one has $ (1-t)v+tw \in \C$. A convex cone $\C \subset \V^{d+1}$ is said to be \emph{proper} if it is non-empty and its closure $\overline{\C}$ does not contain any entire line. 
\end{definition}

\begin{definition}[Cone structure]
A \emph{cone structure $\C$} on a smooth manifold $M$ is a continuous field of cones $( \C_p )_{p \in M}$ such that for all $p \in M$, $\C_p$ is an open proper convex cone in the real vector space $\mathrm{T}_pM$.
\end{definition}

The following notions are classical in Lorentzian geometry. Noticing that they only rely on the existence of a light cone distribution in the tangent bundle and not on the metric, we can naturally extend those to our setting.

\begin{definition}[$\C$-spacelike, $\C$-null and $\C$-timelike vector]
Let $M$ be a smooth manifold endowed with a cone structure $\C$. A non-zero vector $v \in \mathrm{T}_p M$ is said to be:
\begin{itemize}
\item \emph{$\C$-spacelike} if $v \notin \overline{\C_p} \cup(- \overline{\C_p})$,
\item \emph{$\C$-null}, or \emph{$\C$-lightlike}, if $v \in \partial \C_p \cup (-\partial\C_p)$,
\item \emph{$\C$-timelike} if $v \in \C_p \cup (-\C_p)$.
\end{itemize}
A $\C$-null or $\C$-timelike vector $v \in \mathrm{T}_p M$ is said to be \emph{$\C$-future directed} if $v \in \overline{\C_p}$ and \emph{past directed} if $v \in (-\overline{\C_p})$.
\end{definition}

\begin{definition}[$\C$-chronological and $\C$-causal curves]
Let $M$ be a smooth manifold endowed with a cone structure $\C$. A piecewise $C^1$ curve in $M$ is \emph{$\C$-chronological} (respectively \emph{$\C$-causal}) if its velocity vectors are all $\C$-timelike (respectively non-$\C$-spacelike). In both cases, we say that it is \emph{future directed} (respectively \emph{past directed}) if all its velocity vectors are future directed (respectively past directed). 
\end{definition}

\begin{definition}[$\C$-future, $\C$-past and $\C$-diamond]
Let $p$ be a point of a smooth manifold manifold $M$ endowed with a cone structure $\C$. The \emph{$\C$-future of $p$} (respectively the \emph{$\C$-past of $p$}) is the subset $I^+(p)$ (respectively $I^-(p)$) of $M$ formed by the points which are end-points of future directed (respectively past directed) $\C$-chronological curves starting at $p$ .

Replacing the $\C$-chronological curves by $\C$-causal curves, one obtains the definition of the \emph{$\C$-causal future} $J^+(p)$ (respectively \emph{$\C$-causal past} $J^-(p)$) of $p$. 

For two points $p,q \in M$ the (possibly empty) set $I^+(p) \cap I^-(q)$ is called a \emph{$\C$-diamond}, and the set $J^+(p) \cap J^-(q)$ is called a \emph{$\C$-causal diamond}.
\end{definition}

\begin{definition}[Global hyperbolicity and Cauchy surface]
\label{def Cauchy surface}
A cone structure $\C$ on a connected smooth manifold $M$ is \emph{globally hyperbolic} if there exists a $C^1$ embedded hypersurface $S \subset M$ which is:
\begin{itemize}
\item \emph{$\C$-spacelike}, that is every tangent hyperplane $\mathrm{T}_p S \subset \mathrm{T}_p M$ is $\C$-spacelike,
\item a \emph{Cauchy surface}, namely every inextensible future directed $\C$-causal curve in $M$ intersects $S$ at exactly one point.
\end{itemize}
\end{definition}

The following results for Lorentzian spacetimes are due to Geroch in \cite{Geroch_70}. As their proofs only rely on topological and differential arguments, and not on the Lorentzian metric or shape of the cones in the tangent space, these results still hold for smooth manifolds with a proper cone structure \cite{Minguzzi_19}.

\begin{proposition}
\label{prop Geroch Cauchy surfaces}
Let $M$ be a connected smooth manifold admitting a globally hyperbolic cone structure $\C$ with a Cauchy surface $S$. All the Cauchy surfaces of $M$ for the cone structure $\C$ are homeomorphic and $M$ is homeomorphic to $S \times \R$.
\end{proposition}
\begin{proof}
In order to be able to replicate the proof in {\cite[Property 7]{Geroch_70}}, we just need to produce a non-vanishing $\C$-causal vector field on $M$. For that, fix any Riemannian metric $g$ on $M$ and take $X(p) \in \mathrm{T}_pM$ to be the barycentre with respect to $\vol_{g_p}$ of the intersection $\C_p \cap B_{g_p}$, where $ B_{g_p}$ is the unit ball of $(\mathrm{T}_pM,g_p)$. The vector field $X$ is $\C$-causal because of the convexity of the cone distribution $C$, and non-vanishing because of its properness.
Let $\phi^t_X$ be the flow associated with $X$. Up to rescaling $X$, we can assume that it is always defined for all $t \in \R$. Every curve $(\phi^t_X(p))_{t \in \R}$ is an inextensible $\C$-causal curve and thus intersects $S$. Hence,  the wanted homeomorphism is provided by 
\begin{equation*}
 \function{\Phi}{S \times \R}{M}{(p,t)}{\phi^t_X(p)} \, . \qedhere
\end{equation*} 
\end{proof}

\begin{theorem}[{\cite[Theorem~2.45]{Minguzzi_19}}]
A cone structure $\C$ on a connected smooth manifold $M$ is globally hyperbolic if and only if it admits no closed $\C$-causal curve and every $\C$-causal diamond is compact.
\end{theorem}

\subsection{Affine spacetimes}
\label{subsec affine spacetimes}

Let $\C \subseteq \V^{d+1}$ be an open proper convex cone. The cone $\C$ induces a \emph{parallel cone structure} on the affine space $\A^{d+1}$ (endowed with the flat connection $D$). Indeed, the trivial parallelisation of $\A^{d+1}$ gives a natural identification of every tangent space $\mathrm{T}_p\A^{d+1}$ with $\V^{d+1}$. Then, the cone distribution defined by $\C_p=\C$ for all $p \in \A^{d+1}$ is parallel (i.e. it is preserved by $D$-parallel transport). 

We shall introduce affine spacetimes as manifolds with a cone structure which are locally modelled on $(\A^{d+1},\C)$, for some open proper convex cone $\C \subseteq \V^{d+1}$.

\begin{definition}[Affine spacetime]
An \emph{affine spacetime} $(M,\C)$ is a connected equiaffine manifold $M$ endowed with a cone structure $\C$ parallel for the flat affine connection $\nabla$ on $M$ induced by its equiaffine structure (i.e. the field $(\C_p)_{p\in M}$ is preserved by $\nabla$-parallel transport). 
\end{definition}

\begin{remark}
Requiring a cone structure to be parallel is a natural choice generalising the Lorentzian case. Indeed, the distribution of isotropic cones in the tangent bundle of a Lorentzian manifold is parallel for the Levi-Civita connexion of its metric.
\end{remark}

\begin{definition}[Morphism of affine spacetimes]
\label{def isomorphism of affine spacetimes}
A \emph{morphism of affine spacetimes} between two $(d+1)$-dimensional affine spacetimes $(M,\C)$ and $(M',\C')$ is a smooth map $\varphi: M \to M'$ such that
\begin{itemize}
\item $\varphi$ preserves the equiaffine structure, i.e. in each chart of the equiaffine structure of $M$ the corestriction of $\varphi$ to a chart of the equiaffine structure of $M'$ is induced by an element of $\SA(\A^{d+1})$,
\item $\varphi$ sends the cone distribution $\C$ onto $\C'$: for all $x$ in $M$,
\begin{equation*}
\dif_x \varphi ( \C_x) = \C'_{\varphi(x)} \, .
\end{equation*}
\end{itemize}
A \emph{isomorphism of affine spacetimes} is a morphism of affine spacetimes which is a diffeomorphism (then one easily checks that its inverse is also a morphism of affine spacetimes).
\end{definition}

\begin{definition}[Maximal globally hyperbolic affine spacetime]
\label{def Maximal}
A globally hyperbolic affine spacetime $(M,\C)$ is said to be 
\begin{itemize}
\item \emph{Cauchy compact globally hyperbolic} if it is globally hyperbolic with a compact Cauchy surface,
\item \emph{maximal} if, for every globally hyperbolic affine spacetime $(M',\C')$ and affine spacetime morphism embedding $i:M \lhook\joinrel\to M'$ mapping any Cauchy surface of $(M,\C)$ to a Cauchy surface of $(M',\C')$, $i$ is necessarily surjective (and is thus an affine spacetime isomorphism).
\end{itemize}
\end{definition}

\begin{definition}[Convexity of a Cauchy surface of an affine spacetime]
\label{def Convex Cauchy surface}
Let $S$ be a Cauchy surface $S$ of a globally hyperbolic affine spacetime $(M,\C)$.

Then, $S$ is said to be a \emph{$\C$-future-convex Cauchy surface of M}, or equivalently a \emph{future-convex Cauchy surface of $(M,\C)$}, if in any local affine chart of the equiaffine manifold $M$, the $\C$-spacelike hypersurface $S$ can locally be parametrised as the graph of convex function, and any future $\C$-causal vector on $S$ points toward the convex side of $S$. It is said to be \emph{$C^2$ and locally uniformly future-convex} if, moreover, those local graphs are graphs of $C^2$ functions with positive definite Hessian. 
\end{definition}

\section{Background on affine differential geometry and affine spheres}
\label{sec Background ADG}

This section is composed of a few results and definitions concerning \emph{affine differential geometry} and \emph{affine spheres}. For a more complete and deeper overview on those notions we refer the reader to \cite{Nomizu_Sassaki_95,Loftin_10,LSZH_15}.

Affine differential geometry studies smooth hypersurfaces $\Sigma$ embedded in $(\A^{d+1},\dif \vol)$ through the following method. The datum of a transverse vector field $N: \Sigma \to \V^{d+1}$ decomposes the tangent space at each point $p \in \Sigma$ as the direct sum 
\begin{equation*}
\mathrm{T}_p\A^{d+1} = \mathrm{T}_p\Sigma \oplus \bigl\langle N(p)\bigr\rangle \, .
\end{equation*}
Using that decomposition, one can define \emph{affine invariants} of $(\Sigma,N)$:
\begin{itemize}
\item the \emph{induced affine connection} $\nabla$ and the \emph{induced affine fundamental form} or \emph{Blaschke fundamental form} $h \in \Gamma(\mathrm{T}^*\Sigma \otimes \mathrm{T}^*\Sigma)$, determined by
\begin{equation*}
D_X Y =\nabla_X Y + h(X,Y) N \, ,
\end{equation*}
\item the \emph{affine shape operator} $S \in \Gamma(\mathrm{T}^*\Sigma \otimes \mathrm{T}\Sigma)$ and the \emph{transversal connection form} $\tau \in \Gamma(\mathrm{T}^*\Sigma)$, determined by 
\begin{equation*}
D_X N =S(X) + \tau(X) N \, ,
\end{equation*}
\item the \emph{affine Gauss-Kronecker curvature} $\phi \in C^\infty(\Sigma,\R)$, defined by
\begin{equation*}
\phi = \det S \, ,
\end{equation*}
\item the \emph{induced volume form $\nu \in \Gamma(\bigwedge^{d+1}\mathrm{T}^*\Sigma)$}, defined by 
\begin{equation*}
\nu(X_1, \dots, X_d) = \dif \vol(X_1, \dots, X_d,N) \, ,
\end{equation*}
\item in the case where the Blaschke fundamental form $h$ is non-degenerate, the \emph{volume form $\dif\vol_h \in \Gamma(\bigwedge^{d+1}\mathrm{T}^*\Sigma)$ induced by} $h$ on $\Sigma$.
\end{itemize}

Note that given an open set $U \subseteq \R^d$, a $C^2$ immersion $f: U \to \A^{d+1}$ and a $C^1$ map $N: U \to \V^{d+1}$ transversal to $f$, the invariants above can still be defined on $U$ by pullback. In order to lighten them, all the following definitions and propositions will be stated for smooth hypersurfaces, but once again they can be extended to the case of immersed $C^2$ hypersurfaces with a $C^1$ transversal vector field.

\begin{definition}[Equiaffine and normal affine fields]
A transverse vector field $N$ on $\Sigma$ a smooth hypersurface of $\A^{d+1}$ is said to be \emph{equiaffine} if the transversal connection form $\tau$ vanishes. An equiaffine vector field is said to be an \emph{affine normal field} if moreover $\nu = \dif\vol_h$. 
\end{definition}

 Here are some fundamental facts from the theory of affine differential geometry.

\begin{proposition}
\label{prop facts about aff diff geo}
Let $\Sigma$ be a smooth hypersurface in $\A^{d+1}$.
\begin{itemize}
\item The rank and signature of the affine fundamental form $h$ only depend on $\Sigma$ and not on the choice of the transverse vector field $N$. Thus, $\Sigma$ is said to be \emph{non-degenerate} if $h$ is, in that case $h$ is a genuine pseudo-Riemannian metric. 
\item If $\Sigma$ is locally convex, it is non-degenerate if and only if it is \emph{locally uniformly convex}, meaning that $\Sigma$ can locally be parametrized as the graph of a map with positive definite Hessian matrix. Then $h$ is a Riemannian metric if and only if the transverse vector field $N$ is pointing towards the convex side.
\item Any non-degenerate hypersurface admits a unique (up to sign) affine normal field.
\end{itemize}
\end{proposition}

\begin{definition}[Blaschke metric]
\label{def Blaschke metric}
Let $\Sigma$ be a smooth locally uniformly convex hypersurface in $\A^{d+1}$. The Riemannian metric $h$ given by the unique affine normal pointing towards the convex side of $\Sigma$ is called the \emph{affine metric} or \emph{Blaschke metric}.
\end{definition}

\begin{definition}[Affine sphere]
A non-degenerate hypersurface $\Sigma$ is called an \emph{affine sphere} if its shape operator with respect to an affine normal field is of the form $S = \lambda \, \Id$, where $\lambda$ is a real constant. It is called \emph{proper} if $\lambda \neq 0$ and \emph{improper} if $\lambda = 0$. If $\Sigma$ is locally convex, then it is called \emph{hyperbolic} if for the affine normal pointing towards the convex side $\lambda > 0$ and \emph{elliptic} if for that normal $\lambda < 0$.
\end{definition}
It can be shown that being a proper affine sphere with shape operator $\lambda \, \Id$ is a condition equivalent to the existence of a point $o \in \A^{d+1}$, called the \emph{centre of the affine sphere}, such that $N(p) = \lambda \, \overrightarrow{op}$ is an affine normal field.

\begin{theorem}[Cheng--Yau \cite{Cheng_Yau_77,Gigena_81,Cheng_Yau_86}]
\label{theo Cheng--Yau V1}
For any open proper convex cone $\C \subseteq \A^{d+1}$, there exists
a unique complete hyperbolic affine sphere $\S \subset \C$ centred at the vertex of $\C$, asymptotic to $\partial \C$ with affine shape operator the identity.
\end{theorem}
Here, \emph{proper} again means that the cone is non-empty and with closure containing no entire line. The hypersurface is \emph{complete} in the sense that the Riemannian metric on it induced by the ambient Euclidean metric is complete\footnote{In the case of an hyperbolic affine sphere, it is equivalent to completeness of the Blaschke metric \cite[Theorem~3.54]{LSZH_15}, but in a more general setting it is important to keep in mind the differences between the distinct kinds of completeness. A whole discussion on that subject can be found in \cite[Section 3.4]{LSZH_15}.}, and it is \emph{asymptotic to $\partial \C$} in the sense that the distance from $X \in \Sigma$ to $\partial \C$ with respect to an ambient Euclidean metric tends to $0$ as $X$ goes to infinity in $\A^{d+1}$. 

\begin{definition}[Affine conormal field]
\label{def conormal}
Let $\Sigma$ be a smooth hypersurface in $\A^{d+1}$ with an affine normal field $N$. The \emph{affine conormal field} on $\Sigma$ is the field $N^*:\Sigma
 \to (\V^{d+1})^*$ taking value in the vector space $(\V^{d+1})^*$ dual to $\V^{d+1}$, such that for all $p\in \Sigma$, $N^*(p)$ is the unique linear form $\varphi$ satisfying 
\begin{equation*}
\varphi\bp{N(p)} = 1 \quad \text{and} \quad \ker \varphi = \mathrm{T}_p \Sigma \, . 
\end{equation*}
\end{definition}

Note that the vector space $(\V^{d+1})^*$ dual to $\V^{d+1}$ is naturally endowed with a linear volume form $vol^*$ defined in the following way: the volume of a basis $\varphi_1, \dots,\varphi_{d+1}$ of $(\V^{d+1})^*$ is 
\begin{equation*}
 vol^*(\varphi_1, \dots,\varphi_{d+1}) = \frac{1}{vol({e}_{d+1},\dots,{e}_1)} \, ,
\end{equation*}
where ${e}_1,\dots,{e}_{d+1}$ is the basis of $\V^{d+1}$ pre-dual to $\varphi_1, \dots,\varphi_{d+1}$.

\begin{proposition}[{\cite[Proposition~1]{Gigena_81}}]
\label{prop dual of an affine sphere is affine sphere} Let $\Sigma \subset \R^{d+1}$ be a proper affine sphere. The image of the conormal map $N^*$ on $\Sigma$ is a proper affine sphere in $(\V^{d+1})^*$ (seen as an affine space endowed with a parallel volume form induced by $vol^*$) centred at the origin. We shall call it the \emph{affine sphere dual to $\Sigma$}. 

If $\Sigma$ is locally convex, then $N^*(\Sigma)$ is of the same type (elliptic or hyperbolic) as $\Sigma$. 
\end{proposition}

\section{Background on convex analysis and Legendre--Fenchel transform}
\label{sec Background CG}

This section presents notions and results from \emph{convex analysis}. We refer the reader to the classic literature \cite{Rockafellar_97} for a richer view of the theory.
 
\begin{definition}[Legendre--Fenchel transform]
\label{def Legendre--Fenchel transform}
Let $s: \R^d \to \R \cup \{+\infty\}$ be a lower semi-continuous function not constant and equal to $+\infty$. The \emph{Legendre--Fenchel transform} or \emph{Legendre--Fenchel conjugate} of $s$ is the function $s^*:\R^d \to \R \cup \{+\infty\}$ defined by 
\begin{equation*}
s^*(y) = \sup_{x \in \R^d}\bp{x \cdot y - s(x)} \, ,
\end{equation*}
where $\cdot$ is the usual scalar product on $\R^d$.
\end{definition}

The Legendre--Fenchel transform is a standard and fundamental object in analysis. One easily proves that:
\begin{itemize}
\item it is order reversing, that is, if $s \leq f$ then $f^* \leq s^*$,
\item if $c$ is a constant function $(s + c)^* = s^*-c$ ,
\item the Legendre--Fenchel transform of a function is lower semi-continuous and convex,
\item the \emph{biconjugate} $s^{**}$ of a function is its convex hull, that is the pointwise largest lower semi-continuous convex function bounded above by $s$. Hence,  if $s$ is convex and lower semi-continuous then $s^{**}=s$, and for any function $s$ (not necessarily convex) there are no new higher conjugates after the biconjugate as $s^{***}=s^*$.
\end{itemize}

The \emph{domain} of a convex lower semi-continuous function $s: \R^d \to \R \cup \{+\infty\}$ is the convex domain 
\begin{equation*}
\dom(s) \coloneqq \left\lbrace x \in \R^d \st s(x) < +\infty \right\rbrace .
\end{equation*}

\begin{proposition}[{\cite[Corollary~13.3.3]{Rockafellar_97}}]
\label{prop conjugate are Lipschitz}
If $s$ is a lower semi-continuous convex function on $\R^d$ with non-empty domain $\mathrm{dom}(s)$ included in the Euclidean ball $B_{\mathrm{Euc}}(0,r)$, then $s^*$ has finite values and is $r$-Lipschitz on the whole $\R^d$. 
\end{proposition}

\begin{remark}
\label{rem extension of convex function}
Let $U$ is a bounded open convex subset of $\R^d$ and $s:U \to \R$ be a convex function. The function $\bar s:\R^d \to \R \cup \{+\infty\}$ defined by for all $x \in \R^{d+1}$,
\begin{equation*}
 \bar{s}(x) = \begin{cases}
 s(x) & \text{if} \ x \in U , \\
 \liminf\limits_{{x' \in U \to x}}s(x') & \text{if} \ x \in \partial U , \\
 + \infty & \text{if} \ x \notin \overline U , \\
 \end{cases}
\end{equation*}
is the only convex lower semi-continuous function restricting to $s$ on $U$ and taking value $+\infty$ on $\R^d \setminus\bar{U}$ (its values on $\partial U$ are a priori in $\R \cup \{+\infty\}$).
\end{remark}

If $U$ is a bounded open convex subset of $\R^d$, we shall define the Legendre--Fenchel transform of a convex function $s:U \to \R$ as the transform of its extension $\bar s$ given by Remark~\ref{rem extension of convex function}. Then, $s^*$ is defined by, for all $y \in \R^{d}$, 
\begin{equation*}
s^*(y) \coloneqq \bar{s}^*(y)= \sup_{x \in \overline U}\bp{x \cdot y - \bar s(x)} = \sup_{x \in U}\bp{x \cdot y - s(x)}\, ,
\end{equation*}
and it has finite values on the all $\R^d$ by Proposition~\ref{prop conjugate are Lipschitz}. Moreover, we have $(s^{**})_{\vert U} = \bar{s}_{\vert U} = s$.

\begin{proposition}
\label{prop uniform convergence of conjugates}
Let $U$ be a bounded open convex subset of $\R^d$. A sequence $s_n$ of convex functions $U \to \R$ uniformly converges to some convex function $s: U \to \R$ if and only if $f_n=s_n^*: \R^{d} \to \R$ uniformly converges to some convex function $f:\R^{d} \to \R$ . 
\end{proposition}
\begin{proof}
For $s$ a finite valued convex function on $U$ with conjugate $f$, as $(s^{**})_{\vert U}=s$ and $(s_n^{**})_{\vert U}=s_n$, using the order reversing property of the Legendre--Fenchel transform we get the equivalence
\begin{equation*}
s_n-\varepsilon \leq s \leq s_n+ \varepsilon \iff f_n-\varepsilon \leq f \leq f_n+ \varepsilon \, . \qedhere
\end{equation*}
\end{proof}

\begin{definition}[Subgradient and subdifferential]
\label{def subgrad subdif}
Let $s: \R^d \to \R \cup \{+\infty\}$ be a lower semi-continuous convex function and $x_0 \in \R^d$, an element $y \in \R^d$ is a \emph{subgradient} of $s$ at $x_0$ if for all $x\in \R^d$
\begin{equation*}
s(x_0) + (x-x_0) \cdot y \leq s(x).
\end{equation*}
The set of all subgradients of $s$ at $x_0$ is a convex set of $\R^d$ called the \emph{subdifferential} of $s$ at $x_0$ and is denoted by $\partial s(x_0)$. If $x_0$ is in the interior of $\dom(s)$, its subdifferential is non-empty.
\end{definition}

\begin{theorem}[{Fenchel Inequality Theorem \cite[Theorem~23.5]{Rockafellar_97}}]
\label{theo equality for subgradients}
Let $s$ be a lower semi-continuous convex function $s$ on $\R^d$ and $x,y \in \R^d$. The following inequality, known as the \emph{Fenchel Inequality}, holds:
\begin{equation}
\label{eq Fenchel Inequality}
x \cdot y \leq s(x) + s^*(y) \, .
\end{equation} 
Moreover the following conditions are equivalent:
\begin{enumerate}[label=(\roman*)]
\item the equality is achieved,
\item $y\in \partial s(x)$,
\item $x\in \partial s^*(y)$. 
\end{enumerate}
\end{theorem}

One easily notices that if $s_1,s_2:\R^{d} \to \R\cup\{+\infty\}$ are two lower semi-continuous convex functions,
\begin{equation*}
\partial s_1(x) +\partial s_2(x) \subseteq \partial (s_1+s_2)(x) \, .
\end{equation*}
In some cases, that inclusion is actually an equality.

\begin{theorem}[{\cite[Theorem~23.8]{Rockafellar_97}}]
\label{theo equality of sum of differentials}
Let $s_1,s_2: \R^d \to \R \cup \{+\infty\}$ be two lower semi-continuous convex functions. If $\dom(s_1) \cap \dom(s_2)$ has non-empty interior, then for all $x \in \R^d$,
\begin{equation*}
\partial s_1(x) +\partial s_2(x) = \partial (s_1+s_2)(x) \, .
\end{equation*}
\end{theorem}

\begin{proposition}[{\cite[Theorems 25.1 and 25.5]{Rockafellar_97}}]
\label{prop C1 convex function}
Let $s: \R^d \to \R\cup\{+\infty\}$ be a lower semi-continuous convex function. It is differentiable at a point $x \in \dom (s)$ if and only if its subdifferential $\partial s(x)$ is a singleton. In that case $\partial s(x) = \left\lbrace \grad s(x) \right\rbrace$ where $\grad$ is the gradient for the usual scalar product $\cdot$ on $\R^d$.

A lower semi-continuous convex function $s: \R^d \to \R \cup \{+\infty\}$ is $C^1$ on a non-empty open set $U \subseteq \dom(s)$ if and only if at every point of $U$ its subdifferential is a singleton. 
\end{proposition}

\begin{proposition}
\label{prop strictly convex gives C1 conjugate}
Let $U$ be a bounded open subset of $\R^d$ and $s: U \to \R$ a convex function. Its Legendre--Fenchel transform $f = s^*: \R^n \to \R$ is a $C^1$ function if and only if $s$ is strictly convex.
\end{proposition}
\begin{proof}
\noindent\textbullet\ Assume that $s$ is strictly convex. Let $y \in \R^d$, for all $x \in \partial f(y)$, using Proposition~\ref{theo equality for subgradients} we have 
\begin{equation*}
s(x) = x \cdot y-f(y) \, .
\end{equation*}
Hence, if $x, x' \in \partial f(y)$, for all $t \in [0,1]$ we have $(1-t) x+t x' \in \partial f(y)$ and 
\begin{equation*}
s\bp{(1 - t) x + t x'}=(1 - t) x \cdot y + t x' \cdot y- f(y) = (1 - t) s(x)+t s(x') \, ,
\end{equation*}
but as $s$ is strictly convex that means that $x = x'$ and hence $\partial f(y) = \{x\}$. Thus, $f$ is $C^1$ by Proposition~\ref{prop C1 convex function}.

\noindent\textbullet\ Assume that $s$ is not strictly convex. Let $x,x'$ be two distinct point of $U$ such that for all $t \in [0,1]$
\begin{equation*}
s\bp{(1 - t) x + t x'} = (1 - t) s(x)+t s(x') \, .
\end{equation*}
Let $y$ be a subgradient of $s$ at $(x+x')/2$. We claim that $y$ is then also a subgradient of $s$ at $x$ and $x'$, meaning, by Proposition~\ref{theo equality for subgradients}, that $\left\lbrace x,x'\right\rbrace \subset \partial s^*(y)$ and thus that $ s^*(y)$ is not a singleton, so that $f=s^*$ is not $C^1$ by Proposition~\ref{prop C1 convex function}. 

Let us now prove our claim, by definition of a subgradient, $y$ satisfies 
\begin{equation*}
s \vp{\frac{x+x'}{2}} + \frac{x-x'}{2} \cdot y \leq s(x) \quad \text{and} \quad s \vp{\frac{x'+x}{2}} + \frac{x'-x}{2} \cdot y \leq s(x') \, .
\end{equation*}
The sum of two inequalities is
\begin{equation*}
 2 s \vp{\frac{x+x'}{2}} \leq s(x)+s(x') \, .
\end{equation*}
which happens to be an equality, by assumption. Hence, both inequalities must have been equalities:
\begin{equation*}
s \vp{\frac{x+x'}{2}} + \frac{x-x'}{2} \cdot y = s(x) \quad \text{and} \quad s \vp{\frac{x'+x}{2}} + \frac{x'-x}{2} \cdot y = s(x') \, .
\end{equation*}
Then, as $y$ is a subgradient of $s$ at $(x+x')/2$, by the Fenchel Inequality Theorem (Theorem~\ref{theo equality for subgradients}), we have
\begin{equation*}
s(x)= s\vp{\frac{x+x'}{2}} + \frac{x-x'}{2} \cdot y = s^*(y) + \frac{x+x'}{2} \cdot y + \frac{x-x'}{2} \cdot y = s^*(y) + x \cdot y \, .
\end{equation*}
By the Fenchel Inequality Theorem (Theorem~\ref{theo equality for subgradients}), that means that $y \in \partial s(x)$. As $x$ and $x'$ have symmetric roles, we also get that $y \in \partial s(x')$, proving our claim.
\end{proof}

\section{\texorpdfstring{$\C$-convex domains in the affine spacetime $(\A^{d+1},\C)$}{C-convex domains in the affine spacetime (A,C) }}
\label{sec About C-convex domains}

This section is devoted to generalisations of notions and tools from convex geometry in the Minkowski space to the model affine spacetimes $(\A^{d+1},\C)$ (introduced in subsection~\ref{subsec affine spacetimes}).

In this whole section, we let $\C \subset \V^{d+1}$ be an open proper convex cone with vertex at the origin. We arbitrarily choose a base point $O \in \A^{d+1}$, and a direct unit-volume basis $(e_1, \dots, e_{d+1})$ of $(\V^{d+1},vol)$ such that $e_{d+1} \in \C$. That gives a parametrisation of both $\A^{d+1}$ and $\V^{d+1}$ as $\R^{d+1}$ such that the linear volume form $vol$ on $\V^{d+1} \simeq \R^{d+1}$ is the determinant map and the cone $\C \subset \V^{d+1}\simeq \R^{d+1}$ is of the form
 \begin{equation}
 \label{eq expression of C}
 \C = \left\lbrace t \, (x,1) \st x \in \Omega, \ t>0 \right\rbrace ,
 \end{equation}
where $\Omega$ is a bounded open convex domain of $\R^d$ containing $0$.

We are thus working inside the affine spacetime $(\R^{d+1},\C)$, where the cone distribution is given at any point $X \in \R^{d+1}$ by $\C_X = \C \subset \R^{d+1} = \mathrm{T}_X\R^{d+1}$.

In an attempt to keep a coherent and consistent notation that allows the reader to easily identify in which spaces each element lies, we shall use the following notations:
\begin{itemize}
\item an element of $\R^d$ will be denoted by $x$ and an element of $\R^{d+1}$ by $X$,
\item we shall use usual scalar product $\cdot$ on $\R^d$ and $\R^{d+1}$ for duality identifications, and dual elements will be denoted by $y \in \R^d = (\R^d)^*$ and $Y \in \R^{d+1} = (\R^{d+1})^*$.
\end{itemize}

\subsection{\texorpdfstring{$\C$-convex and $\C$-regular domains}{C-convex and C-regular domains}}

Following the work of Nie and Seppi \cite{Nie_Seppi_22}, we generalise standard notions from Lorentzian Minkowski geometry to the affine spacetime $(\R^{d+1},\C)$.

\begin{definition}[Supporting hyperplane and Gauss map]
\label{def sup hyperplane gauss map}
Let $K$ be a convex domain of $\R^{d+1}$, that is a non-empty open convex subset of $\R^{d+1}$. 

An affine hyperplane $H$ of $\R^{d+1}$ is said to be a \emph{supporting hyperplane of $K$ at $X$} if $H \cap K = \emptyset$ and $X \in H \cap \partial K$.

The \emph{Gauss map} of $K$ is the set-valued map on $\partial K$ mapping a point $X$ to the set of linear hyperplanes directing a supporting hyperplane of $K$ at $X$. If $K$ has a $\C^1$ boundary, it is a true map (i.e. it has only singleton values).
\end{definition}

\begin{definition}[$\C$-spacelike and $\C$-null hyperplanes]
\label{def C-sp C-n hyperplanes}
A linear hyperplane $H$ of $\R^{d+1}$ is said to be \emph{$\C$-spacelike} if $H \cap \overline{\C} =\{0\}$, and \emph{$\C$-null} if $H \cap \overline{\C}$ is a half-line (see Figure~\ref{Hyperplanes figure}). An affine hyperplane of $\R^{d+1}$ is said to be respectively \emph{$\C$-spacelike} or \emph{$\C$-null} if it is directed respectively by a $\C$-spacelike or $\C$-null linear hyperplane. 

The \emph{future side} of a linear $\C$-spacelike or $\C$-null hyperplane is the open half-space it bounds which contains $\C$. Similarly, the \emph{future side} of an affine $\C$-null or $\C$-spacelike affine hyperplane $H$ is the open half-space it bounds which contains the future cone $p+\C$ of any point $p \in H$.

We denote respectively by $\mathcal{H}_{sp}(\C)$ and $\mathcal{H}_n(\C)$ the sets of $\C$-spacelike and $\C$-null linear hyperplanes of $\V^{d+1}$.
\end{definition}

\begin{figure}[ht]
\centering
\includegraphics{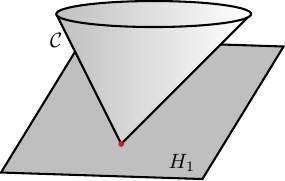}
\quad
\includegraphics{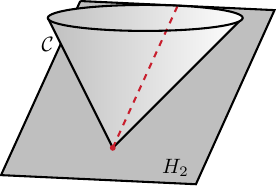}
\quad
\includegraphics{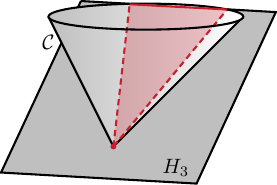}
\caption{Illustrations of a $\C$-spacelike hyperplane $H_1$, a $\C$-null hyperplane $H_2$, and a hyperplane $H_3$ which is neither $\C$-spacelike nor $\C$-null.}
\label{Hyperplanes figure}
\end{figure}

\begin{definition}[$\C$-convex domain]
\label{def C-convex}
A \emph{$\C$-convex} domain is a convex domain $K \subset \A^{d+1}$ such that $K+\C = K$ and the set $\mathcal{H}(K)$ of directions of supporting hyperplanes of $K$ satisfies 
 \begin{equation*}
 \mathcal{H}_{sp}(\C) \subseteq \mathcal{H}(K) \subseteq \mathcal{H}_{sp}(\C) \cup\mathcal{H}_n(\C)\, ,
 \end{equation*}
where $\mathcal{H}_{sp}(\C)$ and $\mathcal{H}_n(\C)$ are respectively the set of directions of $\C$-spacelike and $\C$-null affine hyperplanes.

The \emph{$\C$-spacelike boundary} of a $\C$-convex domain $K$, is the subset of $\partial K$ defined by
\begin{equation*}
 \partial_{sp}K \coloneqq\left \lbrace X \in \partial K \st \text{$K$ admits a $\C$-spacelike supporting hyperplane at $X$} \right\rbrace .
\end{equation*}
A $\C$-convex domain $K$ is called \emph{$\C$-spacelike} if all its supporting hyperplanes are $\C$-spacelike, i.e. if it satisfies $\partial K = \partial_{sp}K$. 
\end{definition}

\begin{figure}[ht]
 \centering
 \includegraphics{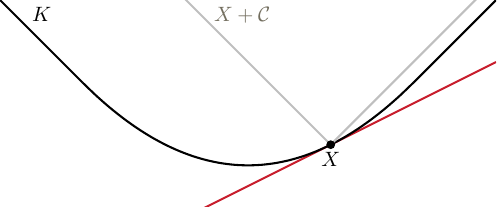}
 \caption{A $\C$-convex domain $K$. }
 \label{fig C-convex}
\end{figure}

It is well-known that any open convex subset of $\R^{d+1}$ is the interior of the intersections of all half-spaces that contain it. Hence, a $\C$-convex domain is the interior of the intersection of the future sides of all its supporting hyperplanes (which are either $\C$-null or $\C$-spacelike). Then, let us introduce the following definition.

\begin{definition}[$\C$-regular domain]
\label{def C-regular}
A \emph{$\C$-regular domain} is a $\C$-convex domain which can be expressed as the interior of the intersection of the future sides of only $\C$-null hyperplanes.
\end{definition}

\begin{remark}
A trivial example of a $\C$-regular domain is any translation $\ X+\C \subset \R^{d+1}$ of the cone $\C$. It is the intersection of the $\C$-future sides of all the $\C$-null affine hyperplanes containing the point $X\in \R^{d+1}$. Its $\C$-spacelike boundary is its vertex: $\partial_{sp} (X+\C) = \left\lbrace X \right\rbrace$.
\end{remark}

\begin{definition}[$C^2_+$ $\C$-convex domain]
\label{def C2+}
A $\C$-convex domain $K$ is called $C^2_+$ if its $\C$-spacelike boundary $\partial_{sp}K$ is a $C^2$ hypersurface and the restriction of its Gauss map to $\partial_{sp} K$ is a $C^1$ diffeomorphism onto $\mathcal{H}_{sp}(\C)$ (for the differentiable structure of the Grassmannian $\mathrm{Gr}_d(\R^{d+1})$, which is the space of linear hyperplanes of $\R^{d+1}$).
\end{definition}

\subsection{Support functions and duality}
\label{subsec support functions}
We shall now use the convex geometry tools developed in \cite{Fillastre_Veronelli_16, Bonsante_Fillastre_17, Barbot_Fillastre_20, Nie_Seppi_22} in this more general setting. We recall that, using our parametrisation of $\V^{d+1}$ and $\A^{d+1}$ as $\R^{d+1}$, we identify the dual vector space $(\R^{d+1})^*$ with $\R^{d+1}$ through the usual scalar product $\cdot$ on $\R^{d+1}$.

The open cone $\C^* \subset \R^{d+1}$ \emph{polar} to $\C$ is the open proper convex cone defined by
\begin{equation*}
\C^* \coloneqq \mathrm{int} \left\lbrace Y \in \R^{d+1} \st X \cdot Y \leq 0, \forall X \in \C\right\rbrace ,
\end{equation*}
where $\mathrm{int}$ denotes the interior of a set. Notice that $\C^*$ is equivalently defined by
\begin{equation*}
 \C^* = \left\lbrace Y \in \R^{d+1} \st X \cdot Y<0, \forall X \in \overline{\C} \setminus \{0\} \right\rbrace .
\end{equation*}
The identity \eqref{eq expression of C} allows to see $\C^*$ as
\begin{equation*}
 \C^* = \left\lbrace t(y,-1) \st y \in \Omega^*, t>0 \right\rbrace , 
\end{equation*} 
where $\Omega^*$ is the bounded open convex domain of $\R^d$ containing $0$ defined by
\begin{equation*}
\Omega^* \coloneqq \left\lbrace y \in \R^d \st x \cdot y < 1 , \forall x \in \overline{\Omega} \right\rbrace .
\end{equation*}

\begin{remark}
\label{Rem correspondance C* support plane}
Let $Y \in \R^{d+1}$. Notice that $Y \in \C^*$ (respectively $Y \in \partial \C^*$) if and only if the linear hyperplane $\{ X \in \R^{d+1} \ \vert \ X \cdot Y =0 \}$ is $\C$-spacelike (resp. $\C$-null) and its $\C$-future is $\{ X \in \R^{d+1} \ \vert \ X \cdot Y \leq 0 \}$. Thus, the future side any $\C$-spacelike (resp. $\C$-null) $H$ can always be expressed as $I^+(H) =\{ X \in \R^{d+1} \ \vert \ X \cdot Y \leq a \}$ where $a \in \R$ and $Y \in \C^*$ (resp. $Y \in \partial \C^*$) directs $H$.
\end{remark} 

\begin{remark}
\label{rem asymptotic cone}
Using Remark~\ref{Rem correspondance C* support plane} and a classical fact from convex analysis \cite[Corollary~14.2.1]{Rockafellar_97}, $\C$-convex domains can be be characterised in the following way. A $\C$-convex domain is a convex domain $K \subset \R^{d+1}$ having \emph{asymptotic cone} (or \emph{recession cone}) equal to $\overline{\C}$, the closure of $\C$, i.e. such that
\begin{equation*}
 \left\lbrace V \in \R^{d+1} \st \forall X \in K, \ X+ \R_+ V \subseteq K \right\rbrace = \overline{\C} \, .
\end{equation*}
\end{remark}

The \emph{total support function} of a $\C$-convex domain $K$ is the function $\tilde{s}_K$ on $\C^*$ defined by
\begin{equation*}
\function{\widetilde{s}_K}{\C^*}{\R}{Y}{\sup \left\lbrace X\cdot Y \st X \in K\right\rbrace} \, .
\end{equation*}
Using Remark~\ref{Rem correspondance C* support plane} and the definition of a $\C$-convex domain (Definition~\ref{def C-convex}), one easily sees that it is well-defined (i.e. the supremum is always finite). It is also clearly sublinear and $1$-homogeneous, hence convex. 

The \emph{$\Omega^*$-support function} of a $\C$-convex domain $K$, is the function $s_K:\Omega^* \to \R$ defined by $s_K (y)=\tilde{s}_K(y,-1)$. As total support function of $K$ is $1$-homogeneous, for all $Y = (y,-\mu)$ in $\C^*$, one has
\begin{equation*}
\tilde{s}_K(Y)=\tilde{s}_K(y,-\mu)=\mu \, s_K\vp{\frac{y}{\mu}} \, .
\end{equation*}
Hence, one can can fully recover the total support function of a $\C$-convex from its $\Omega^*$-support function. We shall then mostly use $\Omega^*$-support functions. In order to lighten notations, we shall simply write that $s_K$ is the \emph{support function of $K$}.

\begin{proposition}
\label{prop K is epigraph}
For every $\C$-convex domain $K$ of $\R^{d+1}$, the hypersurface $\partial K \subset \R^{d+1}=\R^d \times \R$ is the graph of $f=s^*$, the Legendre--Fenchel transform of $s$, the support function of $K$. 
\begin{equation*}
\partial K = \graph(f) \coloneqq \left\lbrace \bp{x,f(x)} \st x \in \R^{d} \right\rbrace.
\end{equation*}
The domain $K$ is then its epigraph: 
\begin{equation*}
K = \epi(f) \coloneqq \left\lbrace (x,\lambda)\in \R^d \times \R \st \lambda > f(x) \right\rbrace .
\end{equation*}

Conversely, for every convex function $s:\Omega^* \to \R$, the convex domain $K(s) = \epi(s^*)$ is a $\C$-convex domain of $\R^{d+1}$.
\end{proposition}
\begin{proof}
Let $K$ be a $\C$-convex domain of $\R^{d+1}$ with support function $s$. For all $y\in \Omega^*$,
\begin{equation*}
s(y)=\tilde{s}(y,-1)= \sup_{(x,\lambda)\in K} (x,\lambda)\cdot(y,-1)= \sup_{(x,\lambda)\in K} (x\cdot y - \lambda) \, .
\end{equation*}
Since the supporting half-spaces of $K$ are $\C$-spacelike and hence not vertical in our coordinates, we can see $\partial K$ as a graph: $\partial K = \graph(f)$ of some convex function $f:\R^d \to \R$. Hence, the previous equation becomes
\begin{equation*}
s(y)= \sup_{x\in \R^d} \bp{x\cdot y - f(x)} = f^*(y) \, ,
\end{equation*}
and as $f$ is convex that means that $f=f^{**}=s^*$, $\partial K = \graph(f) = \graph(s^*)$, and $K=\epi(f)$.

For the converse, let $s:\Omega^* \to \R$ be a convex function. By Proposition~\ref{prop conjugate are Lipschitz}, $s^*$ is finite over the whole $\R^d$. One easily checks that $K(s) = \epi(s^*)$ satisfies $K(s)+\C =K(s) $, and thus that any supporting hyperplane of $K(s)$ is $\C$-spacelike or $\C$-null. Moreover, by the Fenchel Inequality Theorem (Theorem~\ref{theo equality for subgradients}), for all $y \in \Omega^*$, the hyperplane 
\begin{equation*}
 H_{y} = \left\lbrace (x,\lambda) \in \R^{d+1} \st x\cdot y_0 -\lambda = s(y_0)\right\rbrace
\end{equation*}
is a support plane of $K(s) = \epi(s^*)$ at $X=(x,s^*(x))$ where $x \in \partial s(y)$. Using Remark~\ref{Rem correspondance C* support plane}, that implies that $K(s)$ is a $\C$-convex domain.
\end{proof}

The \emph{tube domain} $\Omega^*\times \R$, in which the graphs of support function live, generalises the setting of \emph{co-Minkowski geometry} \cite{Fillastre_Seppi_19,Barbot_Fillastre_20}, also known as \emph{‘‘half-pipe'' geometry} \cite{Danciger_13}. The duality with that tube domain has already been used and studied in the work of Choi \cite{Choi_17,Choi_25}, Nie and Seppi \cite{Nie_Seppi_22,Nie_Seppi_23}, or Bobb and Farre \cite{Bobb_Farre_24}. Figure~\ref{Duality figure} illustrates that duality between $\C$-convex domains and convex graphs in the tube domain $\Omega^*\times \R$. It shows points 
\begin{equation*}
 X_0 = \bp{x_0,s^*(x_0)} \in \R^{d+1} \quad \text{and} \quad Y_0 = \bp{y_0,s(y_0)} \in \Omega^*\times \R \, ,
\end{equation*}
with $x_0 \in \partial s(y_0)$. By the Fenchel Inequality Theorem (Theorem~\ref{theo equality for subgradients}), the affine hyperplanes 
\begin{equation*}
 H_{Y_0} = \left\lbrace (x,\lambda) \in \R^{d+1} \st x\cdot y_0 -\lambda = s(y_0)\right\rbrace \quad \text{and} \quad H_{X_0} = \left\lbrace (y,\nu) \in \Omega^*\times \R \st x_0 \cdot y - \nu = s^*(x_0) \right\rbrace ,
\end{equation*}
are supporting hyperplanes of respectively $\graph(s^*)$ and $\graph(s)$.

\begin{figure}[ht]
\centering
\includegraphics{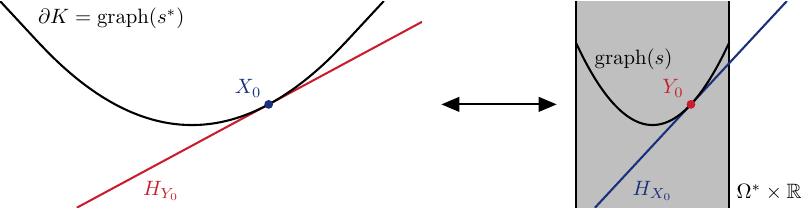}
\caption{Convex duality for $\C$-convex domains.}
\label{Duality figure}
\end{figure}

The following proposition, concerning sum of support functions, is a classical convex geometry fact \cite[Theorem~1.7.5]{Schneider_93}. 

\begin{proposition}
\label{prop support function and sum}
Let $K$ and $K'$ be two $\C$-convex domains of $\R^{d+1}$, with respective support function $s$ and $s'$. Then, their \emph{Minkowski sum}
\begin{equation*}
 K+K'\coloneqq \left\lbrace X+X' \st (X,X')\in K\times K' \right\rbrace
\end{equation*}
is a $\C$-convex domain with support function $s+s'$.
\end{proposition}
\begin{proof}
It is clear that $K+K'$ is a convex domain and that $K+K'+\C = K+K'$. Thus, to prove that it is a $\C$-convex domain we have to show that $\sup \{ X_0\cdot Y \, \vert \, X_0 \in K+K' \}$ is finite for all $Y $ in $\C^*$, and infinite for all $Y \in \R^{d+1} \setminus \overline{\C^*}$. That is a direct consequence from the fact that for all $Y \in \R^{d+1}$,
\begin{align}
\nonumber \sup \left\lbrace X_0\cdot Y \st X_0 \in K+K'\right\rbrace & = \sup \left\lbrace (X+X')\cdot Y \st X_0 \in K, \ X' \in K'\right\rbrace \\
\label{eq sum sup function} & = \sup \left\lbrace X\cdot Y \st X \in K\right\rbrace + \sup \left\lbrace X'\cdot Y \st X \in K'\right\rbrace ,
\end{align}
and the fact that both $K$ and $K'$ are $\C$-convex.

Moreover, equation~\eqref{eq sum sup function} also implies that the support function of $K+K'$ is the sum of the support functions of $K$ and $K'$.
\end{proof}

\subsection{Gauss map}

Let $K$ be a $\C$-convex domain. The \emph{$\C$-spacelike Gauss map} of $K$, denoted by $\G_K$, is the restriction of it Gauss map (Definition~\ref{def sup hyperplane gauss map}) to its $\C$-spacelike boundary (Definiton~\ref{def C-convex}). Notice that the duality described in Remark~\ref{Rem correspondance C* support plane} gives a smooth diffeomorphism between $\Omega^*$ and $\mathcal{H}_{sp}(\C)$, and a homeomorphism between $\partial \Omega^*$ and $\mathcal{H}_n(\C)$; by $y \mapsto \{ X \in \R^{d+1} \ \vert \ X \cdot (y,-1) =0 \}$. Using those identifications, Proposition~\ref{prop K is epigraph}, and the notion of subdifferential (Definition~\ref{def subgrad subdif}), $\G_K$ is explicitly expressed as the set-valued map 
\begin{equation}
\label{eq gauss map def}
\function{\G_{K}}{\partial_{sp} K}{\mathcal{P}(\Omega^*)}{X}{\partial s^*\bp{\pi(X)}} \, ,
\end{equation}
where $\mathcal{P}(\Omega^*)$ is the power set of $\Omega^*$, $s$ is the support function of $K$, and $\pi$ is the projection 
\begin{equation*}
 \function{\pi}{\R^{d+1}}{\R^{d}}{X=(x,\lambda)}{x} \, .
\end{equation*}

If $s^*$ is $C^1$ (or equivalently if $s$ is strictly convex, by Proposition~\ref{prop strictly convex gives C1 conjugate}), $\G_K$ is the true map
\begin{equation*}
\label{eq expression of the Gauss map}
\function{\G_{K}}{\partial_{sp} K}{\Omega^*}{X}{\grad s^*\bp{\pi(X)}} \, .
\end{equation*}
Then, by the Fenchel Inequality Theorem (Theorem~\ref{theo equality for subgradients}), $\G_K$ is invertible if and only if $s$ has a gradient at each point of $\Omega^*$. In that case, $s$ is also $C^1$ (Proposition~\ref{prop C1 convex function}) and the inverse of the Gauss map is expressed as 
\begin{equation}
\label{eq expression of the inverse of the Gauss map}
\function{\G_{K}^{-1}}{\Omega^*}{\partial_{sp} K}{y}{\bp{\grad s(y) , \grad s(y) \cdot y -s (y)}} \, ,
\end{equation}
giving a parametrisation of $\partial_{sp}K$ by $\Omega^*$. Then, the Global Inversion Theorem implies the following.

\begin{proposition}
\label{prop C2+ support function has positive Hessian}
A $\C$-convex domain is $C^2_+$ (Definition~\ref{def C2+}) if and only if its support function is $C^2$ with positive definite Hessian everywhere. We shall call such functions \emph{$C^2_+$ support functions}. 
\end{proposition}

\subsection{Cheng--Yau affine sphere}

In our setting, using our parametrisation of $\A^{d+1}$ as $\R^{d+1}$, Theorem~\ref{theo Cheng--Yau V1} can now be stated more precisely.

\begin{theorem}[Cheng--Yau \cite{Cheng_Yau_77,Gigena_81,Cheng_Yau_86}]
\label{theo Cheng--Yau V2}
Let $\C \subset \R^{d+1}$ be an open proper convex cone. Then, there is a unique complete hyperbolic affine sphere $\S \subset \C \subset \R^{d+1}$ asymptotic to $\partial \C$ and with shape operator the identity. In a parametrisation of $\R^{d+1}$ such that $\C = \{ t (x,1) \ \vert \ x \in \Omega, t>0 \}$ with $\Omega$ a bounded open convex domain of $\R^d$ containing $0$, it is the graph of the Legendre--Fenchel transform of the unique convex function $\omega_\C$ satisfying the Monge--Amp\`ere equation 
\begin{equation*}
\begin{cases}
\det \Hess \omega= (-\omega)^{-d-2} & \text{in} \ \Omega^* \, ,\\
\omega_{\vert \partial\Omega^*} =0 \, .
\end {cases}
\end{equation*}
Moreover, $\omega_\C$ satisfies $\lim_{y \to y_0} \left\vert \grad \omega_\C(y) \right\vert = +\infty$ for all $y_0 \in \partial \Omega^*$.

Moreover, the affine sphere $\S^*$ polar to $\S$ (the opposite of the dual affine sphere described in Proposition~\ref{prop dual of an affine sphere is affine sphere}) is the unique complete hyperbolic affine sphere with shape operator the identity and asymptotic to $\partial \C^*$. It is given by the image of $-N^*_{\S}$, the opposite of the conormal map on the affine sphere (Definition~\ref{def conormal}): 
\begin{equation*}
\Sigma_{(\C^*)} = \S^* = \left\lbrace -\frac{1}{\omega_\C(y)} (y,-1) \st y \in \Omega^*\right\rbrace .
\end{equation*} 
\end{theorem}

\begin{figure}[ht]
\centering
\includegraphics{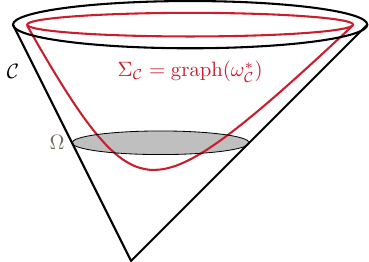}
\caption{The affine sphere given by the theorem of Cheng and Yau.}
\label{Affine sphere figure}
\end{figure}

\begin{corollary}
\label{cor affine sphere invariance}
The affine sphere $\S=\graph(\omega_\C^*)$ from Theorem~\ref{theo Cheng--Yau V2} is preserved by $\mathrm{Aut}_{\SL}(\C)$, the subgroup of elements of $\SL(\R^{d+1})$ preserving the cone $\C$.
\end{corollary}
\begin{proof}
The action of any element of $\mathrm{Aut}_{\SL}(\C)$ on $\R^{d+1}$ preserves the cone $\C$ and the volume, so it necessarily preserves the unique affine sphere $\S$ given by Theorem~\ref{theo Cheng--Yau V2}.
\end{proof}

\begin{remark}[The Minkowski case]
\label{rem the Minkowski case}
In the case where $\Omega$ is the Euclidean unit ball $\B^d$ in $\R^d$, then $\C$ is the Minkowski future light cone given by the inequalities $x_1^2 + \dots + x_d^2 - x_{d+1}^2 < 0$ and $x_{d+1}>0$, and $\Omega^* = \B^d$ is the unit ball given by $y_1^2 + \dots + y_d^2 < 1$. The support function of the affine sphere is $\omega_\C (y) = -\sqrt{1 - y_1^2 - \dots - y_d^2}$ and the affine sphere $\S$ is the hyperboloid $\H^d \subset \R^{d+1}$ given by the equation $x_1^2 + \dots + x_d^2 - x_{d+1}^2 = -1$ and $x_{d+1}>0$. The polar affine sphere $\S^*$ is also a hyperboloid given by the equation $y_1^2 + \dots + y_d^2 - y_{d+1}^2 = -1$ and $y_{d+1}<0$.
\end{remark}

We define $\hat{s}:\S^* \to \R$, the $\S^*$-support function of a $\C$-convex domain $K$, by $\hat{s} = \tilde{s}_{\vert \S}$. 
Let us introduce the radial identification on the dual affine spheres:
\begin{equation*}
\function{R^*}{\S^*}{\Omega^*}{(y,\mu)}{-\frac{y}{\mu}} \quad \text{with inverse} \ \function{(R^*)^{-1}}{\Omega^*}{\S^*}{y}{ -\frac{1}{\omega_\C(y)} (y,-1)} \, .
\end{equation*}
Notice that because $\tilde{s}$ is 1-homogeneous and because of the expression of $\S^*$ given by Theorem~\ref{theo Cheng--Yau V2}, it also satisfies $ \hat{s}=-\frac{s}{\omega_\C}\circ R^*$. Thus, the affine sphere $\S$ has the constant function $-1$ for $\S^*$-support function.

\begin{remark}
\label{rem smooth diffeos}
Notice that the Gauss map of the affine sphere gives a favourite smooth diffeomorphism
\begin{equation*}
\G_{\S}: \S \longrightarrow \Omega^* \, .
\end{equation*}
Moreover the radial identifications in $\R^{d+1}$ and $(\R^{d+1})^*$ give two favourite smooth diffeomorphisms
\begin{equation*}
R: \S\longrightarrow\Omega \quad \text{and} \quad R^*: \S^*\longrightarrow\Omega^* \, .
\end{equation*}
Thus, the followings are all smoothly diffeomorphic 
\begin{equation*}
\Omega \approx \Omega^* \approx \S \approx \S^* \, .
\end{equation*}
Using $N^*_{\S}$ the affine conormal map on the affine sphere (Definition~\ref{def conormal}), everything can be summarised through the following commutative diagram.

\begin{equation*}
\begin{tikzcd}[column sep = large, row sep = large]
& \Omega^* \\
\S \arrow[ur, "\G_{\S}"] \arrow[rr, "- N^*_{\S}"] \arrow[dr, swap, "R"] && \S^*\arrow[ul, swap, "R^*"] \arrow[dl, "\G_{\S^*}"] \\
& \Omega
\end{tikzcd}
\end{equation*}

\end{remark}

\subsection{\texorpdfstring{Boundary functions, $g$-convex domains and Cauchy developments of $\C$-spacelike surfaces}{Boundary functions, g-convex domains and Cauchy developments of C-spacelike surfaces}}
\label{subsec Boundary functions and g-convex domains}

Recall Remark~\ref{rem extension of convex function}: every convex function $s: \Omega^* \to \R$, i.e. every support function of a $\C$-convex domain, extends uniquely to a convex lower semi-continuous function $\bar s:\overline{\Omega^*} \to \R \cup \{+\infty\}$, let us introduce the following objects.

\begin{definition}[$\Omega^*$-boundary function and $g$-convex domain]
\label{def boundary function and g-convex domain}
A function $g:\partial \Omega^* \to \R \cup \{+\infty\}$ is an \emph{$\Omega^*$-boundary function} if there exist a convex function $s: \Omega^* \to \R$ extending lower semi-continuously to $g$ on the boundary $\partial \Omega^*$. 

Given a $\Omega^*$-boundary function $g$, a \emph{$g$-convex domain} of $\R^{d+1}$ is a convex domain $K$ which can be expressed as the epigraph $\epi(s^*)$ of the Legendre--Fenchel transform of a lower semi-continuous convex function $s_K : \Omega^* \to \R\cup\{+\infty\}$ extending lower semi-continuously to $g$ on $\partial \Omega^*$. In that case, $s_K$ is called the \emph{support function} of $K$.
\end{definition}

\begin{remark}
A $g$-convex domain is not necessarily $\C$-convex, as its support function might take infinite values inside the domain $\Omega^*$. Nevertheless, if $g$ is finite, then every $g$-convex domain is a $\C$-convex domain.
\end{remark}

\begin{definition}[Convex envelope of an $\Omega^*$-boundary function]
\label{def convex envelope}
Let $g:\partial \Omega^* \to \R \cup \{+\infty\}$ be an $\Omega^*$-boundary function. The \emph{convex envelope} of $g$ is the pointwise largest convex lower semi-continuous function $\bar s_g : \overline{\Omega^*} \to \R \cup \{+\infty\}$ with boundary values on $\partial \Omega^*$ given by $g$. It can be expressed in the following way: for all $y \in \overline{\Omega^*}$,
\begin{equation*}
\bar s_g (y) = \mathrm{Env(g)} \coloneqq \sup \left\lbrace a(y) \st a \ \text{is affine and} \ a\leq g \ \text{on} \ \partial \Omega^*\right\rbrace .
\end{equation*}
In order to keep consistent notation with respect to Remark~\ref{rem extension of convex function}, we shall denote by $s_g$ the restriction of $\bar s_g (y) = \mathrm{Env(g)}$ to $\Omega^*$. 
\end{definition}

\begin{remark}
\label{rem convex env}
Let $a: \R^{d} \to \R$ be an affine function. Using the compactness of $\overline{\Omega^*}$ and the lower semi-continuity of $g$, one easily shows that, $a < \bar s_g = \mathrm{Env(g)}$ on $\overline{\Omega^*}$ if and only if $a < g$ on $\partial \Omega^*$.
\end{remark}

Because of the order reversing property of the Legendre--Fenchel transform, the fact that the convex envelope of an $\Omega^*$-boundary function is the pointwise largest support function of a $g$-convex domain clearly implies the following.

\begin{proposition}
\label{prop Dg is maximal}
Let $g:\partial \Omega^* \to \R \cup \{+\infty\}$ be an $\Omega^*$-boundary function and consider its convex envelope $\bar s_g : \overline{\Omega^*} \to \R \cup \{+\infty\}$. 
\begin{itemize}
 \item If $g$ (and thus $s_g$) is not constant and equal to $+\infty$, by Proposition~\ref{prop conjugate are Lipschitz}, the Legendre--Fenchel transform $s_g^*:\R^{d} \to \R$ is finite. Then, the convex domain $D_g$ with support function $s_g$, i.e. the epigraph of $s_g^*$, is the maximal $g$-convex domain for inclusion, in the sense that every $g$-convex domain is included in $D_g$.
 \item If $g \equiv +\infty$, then the $g$-convex domain $D_g\coloneqq\R^{d+1}$ is the maximal $g$-convex domain for inclusion.
\end{itemize}

\begin{remark}
Again, the maximal $g$-convex domain $D_g$ is not necessarily $\C$-convex. Nevertheless, if $g$ is finitely valued, then $D_g$ is actually a $\C$-convex domain.
\end{remark}

\end{proposition} 

\begin{proposition}
\label{prop Dg as intersection}
Let $g:\partial \Omega^* \to \R~\cup~\{+\infty\}$ be an $\Omega^*$-boundary function. Then, the maximal $g$-convex domain $D_g$ given by Proposition~\ref{prop Dg is maximal} is the following intersection of $\C$-future sides of $\C$-null affine hyperplanes
\begin{equation*}
D_g = \bigcap_{y \in  \overline{\Omega^*}} \left\lbrace X \in \R^{d+1} \st X \cdot (y,-1) < \bar s_g(y) \right\rbrace =\bigcap_{y \in \partial \Omega^*} \left\lbrace X \in \R^{d+1} \st X \cdot (y,-1) < g(y) \right\rbrace .
\end{equation*}
\end{proposition}
\begin{proof}
Note that
\begin{equation*}
 D_g = \epi (s_g^*) = \left\lbrace (x,\lambda) \in \R^{d+1} \st \lambda > s_g^*(x)\right\rbrace = \left\lbrace (x,\lambda) \in \R^{d+1} \st \lambda > \sup_{y \in \overline{\Omega^*}}\bp{x\cdot y -\bar s_g(y)}\right\rbrace \,. 
\end{equation*}
As $\overline{\Omega^*}$ is compact and $\bar s_g:\overline{\Omega^*} \to \R$ is lower semi-continous, for all $x \in \R^{d}$, the supremum $\sup_{y \in \overline{\Omega^*}}(x\cdot y -\bar s_g(y)) $ is attained at a point of $\overline{\Omega^*}$. Hence, the last identity is equivalent to 
\begin{equation*}
 D_g = \left\lbrace (x,\lambda) \in \R^{d+1} \st x\cdot y - \lambda <\bar s_g(y), \ \forall y \in \overline{\Omega^*}\right\rbrace \, = \bigcap_{y \in  \overline{\Omega^*}} \left\lbrace X \in \R^{d+1} \st X \cdot (y,-1) < \bar s_g(y) \right\rbrace .
\end{equation*}
For all $X \in \R^{d+1}$, $y \to X \cdot (y,-1)$ is an affine function on $\R^{d+1}$. Thus, using Remark~\ref{rem convex env}, the previous identity becomes
\begin{equation*}
 D_g  = \bigcap_{y \in \partial \Omega^*} \left\lbrace X \in \R^{d+1} \st X \cdot (y,-1) < g(y) \right\rbrace . \qedhere
\end{equation*}
\end{proof}

Let us now highlight how maximal $g$-convex domains given by Proposition~\ref{prop Dg is maximal}, which at first seem to be objects purely related to convex geometry, actually have nice causality properties in the affine spacetime $(\R^{d+1},\C)$.

\begin{definition}[Cauchy development]
\label{def Cauchy dev}
The \emph{Cauchy development} of a $\C$-spacelike $C^1$ hypersurface $S$ in $(\R^{d+1}, \C)$ is the set of points $X$ of $\R^{d+1}$ such that every inextensible future directed $\C$-causal curves through $X$ meet $S$. We shall denote it by $D(S)$.
\end{definition}

\begin{proposition}
\label{prop maximal g domain = Cauchy development}
Let $K$ be a $\C$-spacelike $\C$-convex domain (Definition~\ref{def C-convex}) with $C^1$ boundary and support function $s: \Omega \to \R$ extending lower semi-continuously to an $\Omega^*$-boundary function $g:\partial \Omega^* \to \R\cup\{+\infty\}$. Its boundary $\partial K = \partial_{sp} K$ has Cauchy development $D(\partial K)$ equal to the maximal $g$-convex domain $D_g$.
\end{proposition}

\begin{proof}
By hypothesis, the boundary $\partial K=\partial_{sp} K$ is the entire graph of the $C^1$ convex function $s^*:\R^{d} \to \R$.

\noindent\textbullet\ First let us show that $D(\partial K)\subseteq D_g$. Let $X=(x,\lambda) \in D(\partial K)$, its $\C$-future $I^+(X) =X +\C$ is the $\C$-regular domain having the affine function $a_X(y) = x \cdot y - \lambda$ as its support function. As $X$ is in the Cauchy development of $\partial K$, every $\C$-null line through $X$ intersects $\partial K$. Thus, every $\C$-null hyperplane through $X$ intersects $\partial K$. 

As $K$ is $g$-convex, it is contained in the maximal $g$-convex domain $D_g$, and thus every $\C$-null hyperplane through $X$ also intersects $D_g$. By Proposition~\ref{prop Dg as intersection}, $D_g$ is the following intersection of $\C$-future sides of $\C$-null hyperplane,
\begin{equation*}
D_g = \bigcap_{y \in \partial \Omega^*} \left\lbrace X' \in \R^{d+1} \st X' \cdot (y,-1) < g(y) \right\rbrace ,
\end{equation*}
Thus, every $\C$-null hyperplane through $X$ directed by a $y \in \partial \Omega^*$ is in the future of the following hyperplane:
\begin{equation*}
 \left\lbrace X' \in \R^{d+1} \st X' \cdot (y,-1) = g(y) \right\rbrace .
\end{equation*}
That means that for all $y \in \partial \Omega^*$, $X\cdot(y,-1)<g(y)$ (see Figure~\ref{proof Cauchy dev figure}), i.e. $X \in D_g$ by Proposition~\ref{prop Dg as intersection}.

\noindent\textbullet\ Now, let us show the reverse inclusion $D_g \subseteq D(\partial K)$. Let $X \in D_g$, its support function $a_X$ satisfies that for all $y \in \partial \Omega^*$, $a_X(y)<g(y)$. As the support function $s$ of $K$ continuously extends to $g$ on $\partial \Omega^*$, we claim that every $\C$-causal line through $X$ intersects $\partial K$ (see Figure~\ref{proof Cauchy dev figure}). Then, the graph $\partial K = \graph(s^*)$ separates $J^+(X) \cup J^-(X) = (X + \overline{\C}) \cup (X - \overline{\C})$ in two distinct connected components, ensuring that every inextensible future directed $\C$-causal curves through $X$ meet $\partial K$. That is precisely $X \in D(\partial K)$.

Now, let us prove our claim. Let $L$ be a $\C$-causal line through $X$, it is directed by a vector $V=(x',1) \in \overline{\Omega}\times\left\lbrace1\right\rbrace$ and expressed as
\begin{equation*}
 L = \left\lbrace X_t \coloneqq X + tV \st t \in \R \right\rbrace .
\end{equation*}
The $\C$-convex domain $K$ has a horizontal supporting hyperplane, thus for $t$ small enough we have that $X_t \notin K$. Hence, in order to prove that $L$ intersects $\partial K$, we just have to show that for $t$ large enough we have $X_t \in K$. As the convex $K$ domain can be expressed as 
\begin{align*}
 K & = \epi(s^*) = \left\lbrace (x,\lambda) \in \R^{d+1} \st \sup_{y \in \Omega^*} \bp{x \cdot y- s(y) } <\lambda \right\rbrace \\
 & = \left\lbrace X \in \R^{d+1} \st \sup_{y \in \Omega^*} \bp{X \cdot (y,-1) - s(y) } <0 \right\rbrace ,
\end{align*}
we shall prove that for $t$ large enough, we have that for all $y\in \Omega^*$,
\begin{equation*}
 X_t\cdot (y,-1) = a_X(y) + t(x'\cdot y-1) < s(y) -1 \, . 
\end{equation*}
Notice that because $s$ extends lower semi-continuously to $g >a_X$ on $\partial \Omega^*$, the subset $A \coloneqq \{ y \in \Omega^* \ \vert \ a_X(y)\geq s(y) -1\}$ is closed and bounded, thus compact. If $A$ is empty we have $a_X < s$ on $\Omega^*$ and thus $X=X_0 \in K$. Otherwise, as $A$ is compact and non-empty, we can set $m \coloneqq \min \{ x' \cdot y -1 \ \vert \ y \in A \subset \Omega^* \} < 0$ (by definition of $\Omega^*)$ and $M \coloneqq \min \{ s(y) - a_X(y) \ \vert \ y \in A \} \in \R$. For $t >0$ and such that $M+tm < -1$, we have that for all $y \in \Omega^*$, if $y \in A$, then 
\begin{equation*}
X_t\cdot (y,-1) = a_X(y) + t(x'\cdot y-1)\leq s(y)+ M+ tm < s(y)-1 \, ,
\end{equation*}
otherwise $y \notin A$, and then
\begin{equation*}
X_t\cdot (y,-1) = a_X(y) + t(x'\cdot y-1)\leq a_X(y) < s(y) -1 \, ,
\end{equation*}
concluding the proof.
\end{proof}

\begin{figure}[ht]
\centering
\includegraphics{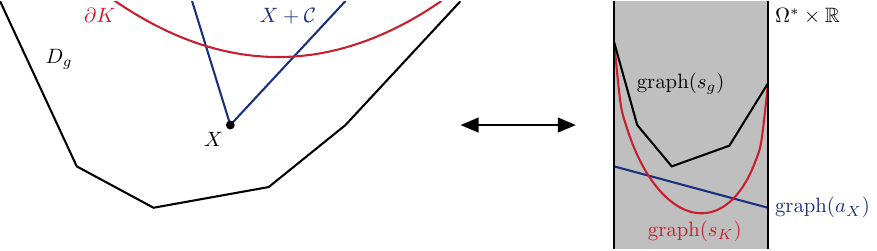}
\caption{Section view of situation in the proof of Proposition~\ref{prop maximal domain = Cauchy development}.}
\label{proof Cauchy dev figure}
\end{figure}

\section{Affine deformations of divisible convex cones}
\label{sec Affine deformations of divisible convex cones}

In the following section, we properly introduce and describe the groups $\Gamma_\tau$ from Theorem~\ref{Theorem Intro domain}. The existence of $D_\tau$, the unique maximal $\Gamma_\tau$-invariant $C$-convex domain from that same theorem, follows from the work of Choi \cite{Choi_25} (see Remark~\ref{rem link with Choi}). Still, we proceed to construct that domain $D_\tau$ following the same approach as Nie and Seppi \cite{Nie_Seppi_23}, as details of this construction will later be useful in our study of the affine spacetime $(D_\tau / \Gamma_\tau,\C)$.

\subsection{Convex projective structures on manifolds and divisible convex sets}
\label{subsec Projective structures on manifolds and divisible convex sets}

A \emph{real projective structure}, or simply \emph{projective structure}, on a smooth manifold $M$ of dimension $d$, consists of an atlas of charts with values in the $d$-dimensional projective space $\P(\R^{d+1})$, such that the transition maps are restrictions of projective transformations. In the language of $(G, X)$-structures, a projective structure is a $(\PGL (\R^{d+1}), \P(\R^{d+1}))$-structure. A manifold is said to admit a \emph{convex projective structure} if it can be realised as the quotient $U/G$ of a domain $U$, included and convex in some affine chart of $\P(\R^{d+1})$, by the properly discontinuous action of a discrete subgroup $G$ of $\PGL (\R^{d+1})$. If $U$ is a \emph{proper convex domain}, that is if it is bounded in some affine chart, the projective structure is called \emph{properly convex}.

Following the nomenclature introduced by Benoist \cite{Benoist_04,Benoist_08}, an open proper convex domain $U$ of $\P(\R^{d+1})$ is said to be \emph{divisible} by a discrete torsion free subgroup $G < \PGL (\R^{d+1})$ if $G$ acts properly discontinuously and cocompactly on $U$.

Throughout all this section, we let $\C$ be an open proper convex cone of $\R^{d+1}$ \emph{divisible} by a discrete torsion-free subgroup $\Gamma < \SL (\R^{d+1})$. That is, $\Gamma$ preserves $\C$ and $\P(\C)$ is divisible by $G = \varphi(\Gamma)$ which is the image of $\Gamma$ by the natural projection morphism
\begin{equation*}
\varphi: \SL (\R^{d+1}) \longrightarrow \PGL (\R^{d+1}) \, .
\end{equation*}
Notice that $\Gamma$ and $G = \varphi(\Gamma)$ are isomorphic. Indeed if $d$ is even, it is clear, because $\varphi$ is an isomorphism. When $d$ is odd, $\varphi$ is a $2:1$ surjective morphism with kernel $\ker(\varphi) = \left\lbrace \pm \Id \right\rbrace$, but if $\gamma \in \Gamma$ preserves $\C$ then $-\gamma$ does not, so $-\gamma \notin \Gamma$, meaning that the restriction $\varphi_{\vert \Gamma}$ is an isomorphism onto its image.

Like in Section~\ref{sec About C-convex domains}, we choose a parametrisation such that $\C = \{ t (x,1) \ \vert \ x \in \Omega, t>0 \}$ where $\Omega$ is a bounded open convex domain of $\R^d$ containing $0$. The subgroup $\Gamma$ acts on different spaces by different manners:
\begin{itemize}
\item $\Gamma$ acts on elements $X \in \R^{d+1}$ by $\gamma X$, that action preserves the cone $\C$ and the unique affine sphere $\S$ given by Theorem~\ref{theo Cheng--Yau V2} (Corollary~\ref{cor affine sphere invariance}),

\item $\Gamma$ acts projectively on $\Omega$, i.e. by identifying both $\Gamma$ with $\varphi(\Gamma) < \PGL(\R^{d+1})$ and $\Omega$ with the projective domain $\P(\C)$. Using the affine parametrisation, the affine sphere and the radial identification, that action can be written as $\gamma x = R(\gamma R^{-1}(x))$, where $R$ is the radial identification between $\S$ and $\Omega$.

\item $\Gamma$ acts dually on elements $Y \in \R^{d+1}$ by $\gamma \star Y \coloneqq \gamma^{-\top} Y$, where $\gamma^{-\top}$ is the inverse of the transpose of $\gamma$. That action preserves the cone $\C^*$, and the unique affine sphere $\S^*$, moreover the conormal map (Definition~\ref{def conormal}) is equivariant in the sense that for all $X \in \S$ and $\gamma \in \Gamma$,
\begin{equation*}
N^*_{\S}(\gamma X) = \gamma \star N^*_{\S}(X) \, .
\end{equation*}

\item $\Gamma$ acts dually projectively on $\Omega^*$, that action being expressed as $\gamma*y \coloneqq R^*(\gamma \star (R^*)^{-1}(y))$, where $R^*$ is the radial identification between $\S^*$ and $\Omega^*$.
\end{itemize}

The divisibility of $\C$ by $\Gamma$ tells us that $\Omega / \Gamma$ is a compact manifold and admitting a properly convex projective structure, as it is naturally smoothly diffeomorphic to $\P(\C) / \Gamma$. The smooth diffeomorphisms introduced in Remark~\ref{rem smooth diffeos} are all $\Gamma$-equivariant by definition of the actions. Thus, we get the following.

\begin{proposition}
\label{prop dual convex is divisible}
Let $\C = \left\lbrace t (x,1) \st x \in \Omega, t>0 \right\rbrace$ be an open proper convex cone, where $\Omega$ is a bounded open convex domain of $\R^d$ containing $0$. If $\C$ is divisible by a discrete torsion-free subgroup $\Gamma < \SL(\R^{d+1})$ then its polar cone $\C^*$ is also divisible by $\Gamma$ (acting dually). Moreover $\Omega / \Gamma$, $\Omega^* / \Gamma$, $\S / \Gamma$ and $\S^* / \Gamma$ are all smoothly diffeomorphic and compact.
\end{proposition}

Finally, here is a useful lemma concerning divisible convex sets.

\begin{lemma}
\label{lem inclusion and divisibility}
Let $A$ and $B$ be two open proper convex subsets of $\P(\R^{d+1})$ such that $A \subseteq B$. If there exists a discrete subgroup $\Gamma< \PGL(\R^{d+1})$ preserving both $A$ and $B$ and such that either $A$ or $B$ is divisible by $\Gamma$, then $A=B$.
\end{lemma}
\begin{proof}
First, notice that we can assume $B$, the smallest convex subset, to be the one divisible by $\Gamma$. Otherwise, if $A$ is divisible by $\Gamma$, the dual convex subsets verify the reverse inclusion $B^*\subseteq A^*$, are both preserved by the dual action of $\Gamma$, and, by Proposition~\ref{prop dual convex is divisible}, the bigest convex subset $A^*$ is divisible by $\Gamma$.

Then, as $B$ is divisible by $\Gamma$, by \cite[Proposition~3]{Vey_70}, for all $p \in A \subseteq B$, the convex hull of the orbit $\Gamma\cdot p$ is the whole $B$. Thus, as $\Gamma$ preserves the convex set $A$, we have $B\subseteq A$.
\end{proof}

\subsection{Affine deformation and cocycle}
\label{subsec Affine deformation and cocycle}
We recall that $\Gamma < \SL(\R^{d+1})$ is a torsion-free subgroup dividing the proper convex cone $\C \subset\R^{d+1}$.

An \emph{affine deformation} of $\Gamma$ is a subgroup $\Gamma_\tau \subset \SL (\R^{d+1}) \ltimes \R^{d+1}$ obtained by adding translation parts to elements of $\Gamma$. Thus, an affine deformation of $\Gamma$ is determined by a translation function $\tau: \Gamma \to \R^{d+1}$. The condition that $\Gamma_\tau$ is a subgroup forces $\tau$ to satisfy the following cocycle relation: for all $\alpha,\beta \in \Gamma$
\begin{equation}
\label{eq cocycle condition}
\tau(\alpha \beta) = \tau(\alpha) + \alpha \tau(\beta) \, .
\end{equation}
Two affine deformations $\Gamma_\tau$ and $\Gamma_\tau'$ are said to be equivalent if there exists $X \in \R^{d+1}$ such that for all $\alpha \in \Gamma$, $(\alpha,\tau(\alpha))$ is conjugate to $(\alpha,\tau'(\alpha))$ through the translation by $X$. That is, if there exists $X \in \R^{d+1}$ such that $\tau - \tau'$ satisfies that for all $\alpha \in \Gamma$,
\begin{equation}
\label{eq coboundary condition}
(\tau - \tau')(\alpha)= \tau(\alpha) - \tau'(\alpha) = (I - \alpha)X \, ,
\end{equation}
where $I$ is the identity in $\SL (\R^{d+1})$.

Equation \eqref{eq cocycle condition} characterises the vector space of \emph{cocycles of $\Gamma$ in $\R^{d+1}$}
\begin{equation*}
\mathrm{Z}^1(\Gamma, \R^{d+1}) \coloneqq \left\lbrace \tau: \Gamma \to \R^{d+1} \st \forall \alpha,\beta \in \Gamma, \, \tau(\alpha \beta) = \tau(\alpha) + \alpha \tau(\beta)\right\rbrace ,
\end{equation*}
while equation \eqref{eq coboundary condition} characterises the vector space of \emph{coboundaries of $\Gamma$ in $\R^{d+1}$}
\begin{equation*}
\mathrm{B}^1(\Gamma, \R^{d+1}) \coloneqq \left\lbrace \tau_X: \Gamma \to \R^{d+1} \st X \in \R^{d+1} \right\rbrace ,
\end{equation*}
where $\tau_X(\alpha) \coloneqq(I - \alpha)X$.

Therefore the space of equivalence classes of affine deformations of $\Gamma$ is described by the first group of the \emph{$\R^{d+1}$-valued group cohomology of $\Gamma$}, that is the vector space
\begin{equation*}
\mathrm{H}^1(\Gamma, \R^{d+1}) \coloneqq \mathrm{Z}^1(\Gamma, \R^{d+1}) / \mathrm{B}^1(\Gamma, \R^{d+1}) \, .
\end{equation*}

\begin{remark}
The existence of non-trivial affine deformations of a group $\Gamma < \SL(\R^{d+1})$ dividing an open proper convex cone is not obvious. When $d=2$, Nie and Seppi \cite{Nie_Seppi_23} have proved that if $\Gamma$ is isomorphic to the fundamental group of the closed connected surface of genus $g \geq 2$, then $\dim \mathrm{H}^1(\Gamma, \R^{d+1}) = 6g-6$. In higher dimension, it is already well-known \cite[Theorem~3.4]{Barbot_Fillastre_20} that closed hyperbolic manifolds admitting a totally geodesic hypersurface have holonomy group $\Gamma < \SO_0(d,1)$ with a non-trivial deformation group $\mathrm{H}^1(\Gamma, \R^{d+1})$. In Appendix~\ref{sec Examples of affine deformations}, we extend those examples and produce, in any dimension, groups dividing an open proper convex domain, which is not projectively equivalent to an ellipse (i.e. giving the Klein ball model of the hyperbolic space), and admitting non-trivial affine deformations.
\end{remark}

\begin{remark}
\label{rem necessary condition for affine deformation}
Notice that a subgroup $G < \SL(\R^{d+1}) \ltimes \R^{d+1}$ is not always an affine deformation of its linear part. A necessary and sufficient condition for that is that any two distinct elements of $G$ have distinct linear parts. Equivalently, that is the case if and only if $G$ does not contain any non-trivial translation, i.e. element of the form $(I,\tau)$, where $I$ is the identity in $\SL (\R^{d+1})$ and $\tau \neq 0$.
\end{remark}
\subsection{Invariant convex domains}
\label{subsec Invariant convex domains}

Throughout the rest of this section we let $\tau \in \mathrm{Z}^1(\Gamma, \R^{d+1})$ be a cocycle of the torsion-free subgroup $\Gamma$ dividing the cone $\C$, and $\Gamma_\tau$ the associated affine deformation. From now on, in order to lighten the notation, we shall denote $\tau_\gamma \coloneqq \tau(\gamma)$.

\begin{definition}[$\tau$-convex domain and $\tau$-support functions]
\label{def tau-convex domain and tau-support functions}
A \emph{$\tau$-convex domain} is a $\C$-convex domain $K \subset \R^{d+1}$ invariant under the action of $\Gamma_\tau$. The \emph{$\tau$-support function} of a $\tau$-convex domain is its $\Omega^*$-support function as a $\C$-convex domain.
\end{definition}

\begin{proposition}
Let $K$ be a $\C$-convex domain with total support function $\tilde{s}$. If $\sigma \in \SL (\R^{d+1}) \ltimes \R^{d+1}$ has linear part $\gamma$ and translation part $\tau$, the total support function of $\sigma K$, that we denote by $\sigma \cdot \tilde{s}$, satisfies 
\begin{equation*}
(\sigma \cdot \tilde{s})(Y) = \tilde{s}(\gamma^{-1} \star Y) + \tau \cdot Y.
\end{equation*}
Thus, $K$ is $\tau$-convex if and only if $\tilde{s}$ is $\tau$-equivariant, in the sense that for all $Y \in \C^*$ and $\gamma \in \Gamma$, it satisfies 
\begin{equation}
\label{eq action on C* support function}
\tilde{s}(Y) = \tilde{s}(\gamma^{-1} \star Y) + \tau_\gamma \cdot Y \, .
\end{equation}
\end{proposition}

\begin{proof}
For all $Y \in \C^*$ and $\sigma =(\gamma, \tau) \in \SL (\R^{d+1}) \ltimes \R^{d+1}$,
\begin{align*}
(\sigma \cdot \tilde{s})(Y) & = \sup_{X\in \sigma K}X \cdot Y = \sup_{X\in K}\sigma X \cdot Y = \sup_{X\in K}\gamma X \cdot Y + \tau \cdot Y  \\
& = \sup_{X\in K} X \cdot \gamma^\top Y + \tau \cdot Y = \tilde{s}(\gamma^{-1} \star Y) + \tau \cdot Y \, . \qedhere
\end{align*}
\end{proof}

Using the radial identification $R^*$ (Remark~\ref{rem smooth diffeos}) and the $\Gamma$-invariance of the affine sphere $\S^*$, we get the following.
\begin{corollary}
Let $K$ be a $\C$-convex domain with support function $s$. If $\sigma \in \SL (\R^{d+1}) \ltimes \R^{d+1}$ has linear part $\gamma$ and translation part $\tau$, the $\Omega^*$-support function of $\sigma K$, that we denote by $\sigma \cdot s$, satisfies 
\begin{equation*}
(\sigma \cdot s)(y) = \frac{\omega_\C(y)}{\omega_\C(\gamma^{-1}*y)}s(\gamma^{-1} * y)+ \tau \cdot (y,-1) \, .
\end{equation*}
Thus, $K$ is $\tau$-convex if and only if its support function $s$ is $\tau$-equivariant, in the sense that for all $y \in \Omega^*$ and $\gamma \in \Gamma$, it satisfies
\begin{equation}
\label{eq action on Omega* support function}
s(y) = \frac{\omega_\C(y)}{\omega_\C(\gamma^{-1}*y)}s(\gamma^{-1} * y)+ \tau_\gamma \cdot (y,-1) \, .
\end{equation}
\end{corollary}

\begin{proposition}
\label{prop equivariance of the Gauss map}
Let $K$ be a $\tau$-convex domain. Its Gauss map is an equivariant set valued map: with respect to the action of $\Gamma_\tau$ on $\partial K$ and the action of $\Gamma$ on $\Omega^*$, i.e. for all $X \in \R^{d+1}$ and $\gamma \in \Gamma$, we have the set equality
\begin{equation*}
\G_K(\gamma X + \tau_\gamma) = \gamma * \G_K(X) \coloneqq \left\lbrace \gamma * y \st y \in \G_K(X)\right\rbrace .
\end{equation*}
If moreover $K$ has a strictly convex support function, its Gauss map is a true equivariant map: with respect to the action of $\Gamma_\tau$ on $\partial K$ and the action of $\Gamma$ on $\Omega^*$, i.e. for all $X \in \R^{d+1}$ and $\gamma \in \Gamma$,
\begin{equation*}
\G_K(\gamma X + \tau_\gamma) = \gamma * \G_K(X) \, .
\end{equation*}
\end{proposition}

\begin{proposition}
\label{prop free and properly discontinuous action on C1 tau-invariant hypersurfaces}
Let $K$ be a $\tau$-convex domain with a strictly convex support function $s$. Then, its boundary $\partial K$ is $C^1$ and the action of $\Gamma_\tau$ on its $\C$-spacelike boundary $\partial_{sp} K$ is free and properly discontinuous. If moreover $s$ is $C^1$, then $\partial_{sp} K / \Gamma_\tau$ is homeomorphic to $\Omega / \Gamma$.
\end{proposition}
\begin{proof}
By Proposition~\ref{prop strictly convex gives C1 conjugate}, the strict convexity of $s$ implies that its Legendre--Fenchel transform $s^*$ is $C^1$, hence $\partial K = \graph(s^*)$ is $C^1$. As the Gauss map $\G_K:\partial_{sp} K \to \Omega^*$ is continuous, surjective by definition of a $\C$-convex domain (Definition~\ref{def C-convex}), and equivariant by Proposition~\ref{prop equivariance of the Gauss map}, and the action of $\Gamma$ on $\Omega^*$ is free and properly discontinuous, then the action of $\Gamma_\tau$ on $\partial_{sp} K$ is necessarily free and properly discontinuous. If moreover $s$ is $C^1$, the Gauss map is an equivariant homeomorphism with inverse map given by expression \eqref{eq expression of the inverse of the Gauss map}. The Gauss map thus descends to a homeomorphism from $\partial_{sp} K / \Gamma_\tau $ to $\Omega^* / \Gamma \simeq \Omega / \Gamma$ (by Proposition~\ref{prop dual convex is divisible}).
\end{proof}

\subsection{\texorpdfstring{Extension of $\tau$-support functions to the boundary}{Extension of tau-support functions to the boundary}}

The following propositions can be found in \cite{Nie_Seppi_23}, for the reader's convenience we do not only state these propositions but also explicitly present their proofs, in any dimension and with notation coherent with the rest of this article.

\begin{proposition}[{\cite[Lemma 5.1]{Nie_Seppi_23}}]
\label{prop exists smooth tau-equivariant function}
There exists a smooth $\tau$-equivariant function on $\Omega^*$, i.e. a smooth function satisfying \eqref{eq action on Omega* support function} on $\Omega^*$.
\end{proposition}
\begin{proof}
As $\Omega^*/\Gamma$ is compact and $\pi_\Gamma:\Omega^* \to \Omega^*/\Gamma$ is its universal cover, we can set $(U_i)_{0 \leq i \leq n}$ to be a finite open cover of $\Omega^*/\Gamma$ such that for all $i$, 
\begin{equation*}
V_i:=\pi_\Gamma^{-1}(U_i) =\bigsqcup_{\gamma \in \Gamma} U_{i,\gamma} \, ,
\end{equation*}
where the union is disjoint, each $\pi_{\Gamma \vert U_{i,\gamma}}:U_{i,\gamma} \to U_i$ is a homeomorphism, and for all $\gamma,\gamma' \in \Gamma$, $U_{i,\gamma \gamma'} = \gamma * U_{i,\gamma'}$. We set $(\psi_i)_{0 \leq i \leq n}$ to be a partition of unity on $\Omega^*/\Gamma$ subordinate to $(U_i)_ {i}$. It lifts to a partition of unity $(\Psi_i)_i$ subordinate to $(V_i)_i$ defined by $\Psi_i = \psi_i \circ \pi_\Gamma$. Let $u \in C^\infty(\Omega^*/\Gamma)$. We define functions $v_i \in C^\infty(V_i)$ by
\begin{equation*}
v_{i \vert U_{i,\gamma}}(y) = \frac{\omega_\C(y)}{\omega_\C(\gamma^{-1}*y)}u \bp{\pi_\Gamma(y)} + \tau_\gamma \cdot (y,-1) \, .
\end{equation*}
Thus, $v_i: V_i \to \R$ satisfies for all $y$ in $V_i$ and $\gamma \in \Gamma$,
\begin{equation*}
v_i(y) = \frac{\omega_\C(y)}{\omega_\C(\gamma^{-1}*y)}v_i(\gamma*y) + \tau_\gamma \cdot (y,-1) \, ,
\end{equation*}
which is precisely the equivariance condition \eqref{eq action on Omega* support function}. That implies that the function $v:=\sum_i v_i\Psi_i$ is $\tau$-equivariant.
\end{proof}

\begin{proposition}[{\cite[Proposition~5.2]{Nie_Seppi_23}}]
\label{prop exists tau-support function extending to boundary}
There exists a smooth $\tau$-support function (Definition~\ref{def tau-convex domain and tau-support functions}) on $\Omega^*$ which extends to a continuous function on $\overline{\Omega^*}$.
\end{proposition}
\begin{proof}
Let $v$ be a smooth $\tau$-equivariant function on $\Omega^*$ given by Proposition~\ref{prop exists smooth tau-equivariant function}, let $K$ be a precompact open subset of $\Omega^*$ sufficiently large so that $\Omega^* = \bigcup_{\gamma\in\Gamma}\gamma*K$. As $\omega_\C$ is smooth and strictly convex we can set $T>0$ such that 
\begin{equation*}
s_+:= v - T\omega_\C \quad \text{and} \quad s_-:= v + T\omega_\C
\end{equation*}
are respectively locally uniformly concave (i.e. with everywhere negative definite Hessian) and locally uniformly convex (i.e. with everywhere positive definite Hessian) on $K$. As $\omega_\C$ is $\Gamma$-equivariant, they are also $\tau$-equivariant, and we claim that implies that they are respectively concave and convex on the whole $\Omega^*$. Then that means that $s_+$ extends to $g_+$, an upper semi-continuous function with values in $\R\cup\{-\infty\}$ on $\partial \Omega^*$, and that $s_-$ extends to $g_-$, a lower semi-continuous function with values in $\R\cup\{+\infty\}$ on the boundary $\partial \Omega^*$. As $\omega_\C$ extends to the zero function on $\partial \Omega^*$, we have that $g_-=g_+$ is a continuous function with finite values on $\partial \Omega^*$.

Now, let us prove our claim. We have that $\tilde{s}_-$, the 1-homogeneous extension of $s_-$ to $\C^*$, has a semi-definite positive Hessian on $\widetilde{K}:= \{ t(y,-1) \ \vert \ y\in K, t>0\}$. As for all $Y \in \gamma\star\widetilde{K}$ with $\gamma \in \Gamma$,
\begin{equation*}
\tilde{s}_-(Y) = \tilde{s}(\gamma^{-1} \star Y) +\tau_\gamma \cdot Y = \tilde{s}(\gamma^T Y) + \tau_\gamma \cdot Y \, ,
\end{equation*}
we get that $\tilde{s}_-$ has a semi-definite positive Hessian on $\gamma \star \widetilde{K}$. Hence, as $\C^* = \bigcup_{\gamma\in\Gamma}\gamma \star \widetilde{K}$, we see that $\tilde{s}_-$ is convex, and so is $s_-$. The exact same can be done to show the concavity of $s_+$.
\end{proof}

\begin{proposition}
\label{prop existence of boundary function gtau}
There exists a unique continuous function, denoted by $g_\tau$, on the boundary $\partial \Omega^*$ such that any $\tau$-support function (or more generally any $\tau$-equivariant continuous function on $\Omega^*$) continuously extends on $\overline{\Omega^*}$ to a function coinciding with $g_\tau$ on $\partial \Omega^*$, i.e. every $\tau$-convex domain is $g_\tau$-convex (Definition~\ref{def boundary function and g-convex domain}).
\end{proposition}
\begin{proof}
Let $s$ be a $\tau$-support function extending to $g$ on $\partial \Omega^*$ given by Proposition~\ref{prop exists tau-support function extending to boundary}. Let $s'$ be any other $\tau$-support function. Then $(s-s')/{\omega_\C}$ is $\Gamma$-invariant, it is the lift of a continuous function on the compact quotient $\Omega^*/\Gamma$. Thus, it reaches a minimum $m$ and a maximum $M$, and we have
\begin{equation*}
m \omega_\C \leq s-s' \leq M \omega_\C \, .
\end{equation*}
As $\omega_\C$ tends to zero on the boundary $\partial \Omega^*$ (Theorem~\ref{theo Cheng--Yau V2}), we get that $s'$ continuously extends on $\overline{\Omega^*}$ to a function coinciding with $g$ on $\partial \Omega^*$.
\end{proof}

\begin{remark}
We have shown that every $\tau$-convex domain is $g_\tau$-convex, but the converse is not true, a $g_\tau$-convex domain has a priori no reason to be invariant under the action of $\Gamma_\tau$.
\end{remark}

\subsection{\texorpdfstring{Maximal $\tau$-convex domain and Cauchy developments of $\Gamma_\tau$-invariant $\C$-spacelike surfaces}{Maximal tau-convex domain and Cauchy developments of Gamma tau-invariant C-spacelike surfaces}}
\label{subsec Maximal domains}

In this subsection, we prove the existence of the maximal invariant domain from Theorem~\ref{Theorem Intro domain} and show that it is the Cauchy development of the boundary of any $\C$-spacelike $\tau$-convex domains. 

Using the boundary map $g_\tau$ given by Proposition~\ref{prop existence of boundary function gtau}, we can use the results from subsection~\ref{subsec Boundary functions and g-convex domains}. Let $s_\tau$ be the restriction to $\Omega^*$ of the \emph{convex envelope} of $g_\tau$ (Definition~\ref{def convex envelope}). It can be expressed in the following way: for all $y \in \Omega^*$,
\begin{equation*}
s_\tau (y)= \sup \left\lbrace a(y) \st a \ \text{is affine and} \ a\leq g_\tau \ \text{on} \ \partial \Omega^*\right\rbrace .
\end{equation*}
Then, one can easily check that $s_\tau$ is $\tau$-equivariant, i.e. satisfies \eqref{eq action on C* support function}, and thus that $D_\tau \coloneqq K(s_\tau)= \epi(s_\tau^*)$ is a $\tau$-convex domain. Proposition~\ref{prop Dg is maximal} and \ref{prop Dg as intersection} give the following.

\begin{proposition}
\label{prop Dtau is maximal}
The $\C$-convex domain $D_\tau$ with support function $s_\tau$ is the \emph{maximal $\tau$-convex domain} for inclusion. That is, every $\tau$-convex domain is included in $D_\tau$.
\end{proposition} 

\begin{figure}[ht]
\centering
\includegraphics{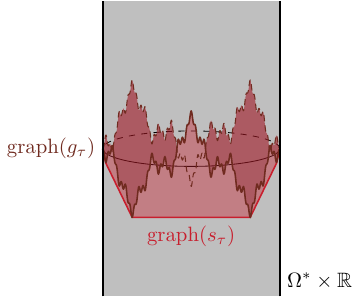}
\caption{The unique $\tau$-equivariant boundary function $g_\tau$, and its convex envelope $s_\tau.$}
\label{boundary function figure}
\end{figure}

\begin{proposition}
\label{prop Dtau as intersection}
The maximal $\tau$-convex domain $D_\tau$ is $\C$-regular. It is the intersection of the $\C$-future sides of the $\C$-null affine hyperplanes given by the boundary function $g_\tau: \partial \Omega^* \to \R$ associated with $\tau$:
\begin{equation*}
D_\tau = \bigcap_{y \in \partial \Omega^*} \left\lbrace X \in \R^{d+1} \st X \cdot (y,-1) < g_\tau(y) \right\rbrace .
\end{equation*}
\end{proposition}

\begin{remark}
\label{rem link with Choi}
In the work of Choi \cite{Choi_25}, the existence of $\tau$-convex domains is given by Theorem~4.3.1. Indeed, we are in a setting where the affine subgroup $\Gamma_\tau < \mathrm{GL}(\R^{d+1}) \ltimes\R^{d+1}$ has linear part in $\SL(\R^{d+1})$ and thus automatically verifies Choi's \emph{‘‘uniform middle eigenvalue condition''}. Letting $\Omega_{\Sph^{d+1}}\subset \Sph^{d+1}$ be the extension at infinity of $\C$ in $\Sph^{d+1}$, the double cover of $\P(\R^{d+1})$, Theorem~4.3.1. states the existence of a convex domain $K \subset \R^{d+1}$ such that $(\Gamma_\tau,K,\Omega_{\Sph^{d+1}})$ is a \emph{‘‘properly convex triple''} and $\Gamma_\tau$ acts \emph{‘‘asymptotically nicely''} on $K$. The definitions of a ‘‘properly convex triple'' and an ‘‘asymptotically nice'' affine action \cite[Section 4.1]{Choi_25} imply that the convex domain $K$ verifies $K+\C = K$, that it has supporting hyperplanes directed by every elements of $\Omega^*$ and that it is $\Gamma_\tau$-invariant, i.e. that $K$ is $\tau$-convex.

The existence and uniqueness of the maximal $\tau$-convex domain $D_\tau$ then comes from Theorem~4.3.7 in \cite{Choi_25}. It states that there is a unique family family $(H_y)_{y\in \partial\Omega^*}$ of affine $\C$-null hyperplane (given by the boundary function $g_\tau$ in our approach) such that for every properly convex affine triple $(\Gamma_\tau,K,\Omega_{\Sph^{d+1}})$, the domain $K$ has all $(H_y)_{y\in \partial\Omega^*}$ as supporting hyperplanes ‘‘at infinity''. Then, the intersection of the $\C$-future of that family of $\C$-null affine hyperplanes gives a unique maximal $\tau$-convex domain $D_\tau$ (Proposition~\ref{prop Dtau as intersection}).
\end{remark}

Using Proposition~\ref{prop maximal g domain = Cauchy development}, we also have the following results about the Cauchy development (Definition~\ref{def Cauchy dev}) of the boundary of $\C$-spacelike (Definition~\ref{def C-convex}) $\tau$-convex domains . 

\begin{proposition}
\label{prop maximal domain = Cauchy development}
Let $K$ be a $\C$-spacelike $\tau$-convex domain with $C^1$ boundary. Its boundary $\partial K = \partial_{sp} K$ has Cauchy development $D(\partial K)$ equal to the maximal $\tau$-convex domain $D_\tau$.
\end{proposition}

\begin{corollary}
\label{cor C2+ implies Cauchy surface}
Let $K$ be a $\C$-spacelike $C^2_+$ $\tau$-convex domain. Its boundary $\partial K = \partial_{sp} K$ is a $\Gamma_\tau$-invariant $C^2$ locally uniformly future-convex Cauchy surface of the affine spacetime $(D_\tau,\C)$. 
\end{corollary}

\subsection{The two maximal domains}
\label{subsec The two maximal domains}

Notice that, up to now, both the group $\Gamma$ and the cone $\C$ it divides were fixed, but if $\Gamma$ divides the cone $\C$ it also divides its opposite $-\C$. The two affine spacetime structures $(\R^{d+1},\C)$ and $(\R^{d+1},-\C)$ are related: the notions of $\C$-null and $\C$-spacelike hyperplane are the same as the notions of $(-\C)$-null and $(-\C)$-spacelike hyperplane, but time is ``reversed'' as the notion of $\C$-future coincides with the notion of $(-\C)$-past. Thus, to $\Gamma_\tau$ an affine deformation of $\Gamma$ we can associate two domains $D^{\C}_{\tau}$ and $D^{-\C}_{\tau}$.

Notice that using the $\C$-duality, $(-\C)$-convex domains are given by concave support functions on $\Omega^*$. In Proposition~\ref{prop existence of boundary function gtau} we have seen that both convex and concave $\tau$-equivariant functions on $\Omega^*$ have the same continuous extension $g_\tau$ on $\partial \Omega^*$. Thus, the two domains come from considering:
\begin{itemize}
\item$s_\tau^-: \Omega^* \to \R$ (that we previously only denoted as $s_\tau$) the pointwise largest convex function on $\Omega^*$ with graph under the graph of $g_\tau$ on $\partial \Omega^*$, i.e.
\begin{equation*}
s_\tau^- (y)= \sup \left\lbrace a(y) \st a \ \text{is affine and} \ a\leq g_\tau \ \text{on} \ \partial \Omega^*\right\rbrace ,
\end{equation*}
\item$s_\tau^+: \Omega^* \to \R$ the pointwise smallest concave function on $\Omega^*$ with graph over the graph of $g_\tau$ on $\partial \Omega^*$, i.e.
\begin{equation*}
s_\tau^+ (y)= \inf \left\lbrace a(y) \st a \ \text{is affine and} \ a\geq g_\tau \ \text{on 
} \ \partial \Omega^*\right\rbrace .
\end{equation*}
\end{itemize}
Then, the graphs of $s_\tau^-$ and $s_\tau^+$ bound $\mathrm{CH}^*(\tau)$, the convex hull of $\graph(g_\tau)$ inside the dual space $(\R^{d+1})^*$:
\begin{equation*}
\mathrm{CH}^*(\tau) = \left\lbrace (y,\mu) \in \R^{d+1} \st y \in \Omega^*, s_\tau^-(y) \leq \mu \leq s_\tau^+(y)\right\rbrace .
\end{equation*}

Consequently, by Proposition~\ref{prop Dtau as intersection}, the two domains $D^{\C}_{\tau}$ and $D^{-\C}_{\tau}$ can be expressed as, respectively, the intersection of the $\C$-future, and the intersection of the $(-\C)$-future (i.e. the $\C$-past), of the same family of $\C$-null affine hyperplanes given by the boundary function $g_\tau: \partial \Omega^* \to \R$ associated with $\tau$. That implies that $D^{\C}_{\tau}$ and $D^{-\C}_{\tau}$ are disjoint. One can also easily check that $D^{\C}_\tau = - D^{-\C}_{-\tau}$.

\section{\texorpdfstring{Cosmological time on $\C$-convex domains}{Cosmological time on C-convex domains}}
\label{sec Cosmological time on C-convex domains}

In this section, we introduce the \emph{cosmological time} function on $\C$-convex domains and prove that it satisfies the regularity and concavity conditions from Theorem~\ref{Theorem Intro cosmological time}.  That cosmological time will be a central tool for proving point~\ref{theo11} of Theorem~\ref{Theorem Intro domain} in Section~\ref{sec Quotient of the maximal domain}.

All the tools from this section can be used in an affine spacetime $(\A^{d+1},\C)$, where $\C$ is an open proper convex cone (not necessarily divisible). Thus, we work in the same setting as in Section~\ref{sec About C-convex domains}: the ambient affine spacetime is $(\R^{d+1},\C)$ where $\C = \{ t (x,1) \ \vert \ x \in \Omega, t>0 \}$ with $\Omega$ a bounded open convex domain of $\R^d$ containing $0$.

\begin{definition}[Gauge support function]
\label{def gauge support function}
A \emph{gauge support function} is a convex function $\omega: \Omega^* \to \R$ satisfying the following conditions:
\begin{enumerate}[label=(GS\arabic*)]
 \item \label{cdt C2+} $\omega$ is $C^2_+$, i.e. it is $C^2$ with positive definite Hessian on the whole $\Omega^*$,
 \item \label{cdt 0 on dOmega} $\omega$ extends continuously to $\overline{\Omega^*}$ and $\omega_{\vert \partial \Omega^*} = 0$,
 \item \label{cdt empty subdif on dOmega} for all $y \in \partial \Omega^*$, we have $\lim_{y' \to y} \vert \grad \omega_\C(y') \vert = +\infty$ and thus $\partial\omega(y)=\emptyset$.
\end{enumerate}
\end{definition}

For the rest of this section, we let $\omega: \Omega^* \to \R$ be a gauge support function. Let $\Sigma = \graph(\omega^*)$ be the boundary of the $\C$-convex domain $K(\omega)$ with support function $\omega$. One can check that conditions~\ref{cdt C2+} and~\ref{cdt 0 on dOmega} imply that $\Sigma$ is included in $\C$ and asymptotic to $\partial \C$. Moreover, condition~\ref{cdt empty subdif on dOmega} tells us that $\Sigma$ is $\C$-spacelike, i.e. all its supporting hyperplanes are $\C$-spacelike. Also note that because of conditions~\ref{cdt C2+} and~\ref{cdt 0 on dOmega}, $\omega<0$ on $\Omega^*$.

A nice candidate for the gauge support function $\omega$ is $\omega_\C$, the support function of $\S$ the Cheng--Yau affine sphere given by Theorem~\ref{theo Cheng--Yau V2}. Because $\S$ is invariant by all the elements of $\SL(\R^{d+1})$ preserving $\C$ (Corollary~\ref{cor affine sphere invariance}), it will give rise to tools behaving well with quotients. Hence, we shall make this choice in subsection~\ref{subsec Invariance of the cosmological time given by the CHeng--Yau affine spheres} and later in Section~\ref{sec Quotient of the maximal domain} for our study of the quotients $D_\tau/\Gamma_\tau$ through their cosmological time.

\subsection{\texorpdfstring{Normal projection, cosmological time and foliation of $\C$-convex domains}{Normal projection, cosmological time and foliation of C-convex domains}}
\label{subsec Normal projection, cosmological time and foliation of C-convex domains}

\begin{lemma}
\label{lem s + t omega is C1}
Let $s : \Omega^*\to \R$ be a convex function. Then, for all $t>0$, the convex function $(s+t\omega)^*$ is $C^1$.
\end{lemma}
\begin{proof}
As $s$ is convex and $\omega$ is strictly convex, for all $t>0$ the function $s+t\omega$ is strictly convex. Hence, by Proposition~\ref{prop strictly convex gives C1 conjugate}, the convex function $(s+t\omega)^*$ is $C^1$.
\end{proof}

\begin{definition}[$\C$-spacelike graph]
The $\C$-\emph{spacelike graph graph} of a convex function $f:\R^{d} \to \R$ is the set
\begin{equation*}
\graph_{sp} (f) \coloneqq \left\lbrace \bp{x,f(x)} \ \vert \ \partial f (x) \subseteq \Omega^*\right\rbrace .
\end{equation*}
\end{definition}

\begin{lemma}
\label{lem DspK =DK}
Let $s : \Omega^*\to \R$ be a convex function. Then, for all $t>0$,
\begin{equation*}
\graph_{sp} (s+t\omega)^* = \graph (s+t\omega)^* \, .
\end{equation*}
\end{lemma}
\begin{proof}
Let $t>0$. It suffices to prove that for all $x \in \R^d$, $\partial(s+t\omega)^*(x) \subseteq \Omega^*$. Let us proceed by contradiction and assume that we can set $x \in \R^d$ and $y \in \partial(s+t\omega)^*(x) \cap \partial \Omega^*$. Then, by the Fenchel Inequality Theorem (Theorem~\ref{theo equality for subgradients}), $x \in \partial (\bar s+t\omega)(y)$ where $\bar s : \overline{\Omega^*}\to\R$ is the convex lower semi-continuous extension of $s$ given by Remark~\ref{rem convex env} and $\omega$ is extended continuously to $0$ on $\partial \Omega^*$.

Condition~\ref{cdt empty subdif on dOmega} from Definition~\ref{def gauge support function} is satisfied by $\omega$, so $ \partial (t\omega)(y) = \emptyset$. Thus, by Theorem~\ref{theo equality of sum of differentials},
\begin{equation*}
\partial (s+t\omega)(y) = \partial s(y)+ \partial (t\omega)(y) = \partial s(y)+\emptyset = \emptyset \, .
\end{equation*}
That is a contradiction as $x \in \partial (s+t\omega)(y)$. 
\end{proof}

Note that the $\C$-convex domain with support function $\omega$ is the Minkowski sum
\begin{equation*}
 K(\omega) = \Sigma + \C \, .
\end{equation*}
Thus, if $K$ is a $\C$-convex domain with support function $s$, by Proposition~\ref{prop support function and sum}, for all $t \geq 0$, the $\C$-convex domain with support function $s+ t\omega$ is 
\begin{equation*}
 K(s+t\omega) = K +t(\Sigma + \C) = K+ \C +t \Sigma = K+ t \Sigma \, .
\end{equation*}
Hence, Proposition~\ref{prop K is epigraph}, Lemma~\ref{lem s + t omega is C1} and Lemma~\ref{lem DspK =DK} imply that for all $t > 0$, $K +t \Sigma$ is a $\C$-spacelike $\C$-convex domain with $C^1$ boundary 
\begin{equation*}
 \partial(K+t\Sigma) = \graph(s+t\omega)^* \, .
\end{equation*}

\begin{lemma}
\label{lem C1 foliation of K}
Let $K$ be a $\C$-convex domain with support function $s$. Then, the $C^1$ $\C$-spacelike hypersurfaces $(\partial(K+t\Sigma))_{t>0}$ foliate $K$.
\end{lemma}
\begin{proof}
\noindent\textbullet\ First, let us prove that the hypersurfaces $(\partial(K+t\Sigma))_{t>0} = (\graph(s+t\omega)^*)_{t>0}$ are disjoint. It suffices to notice that if $x\in\R^n$ and $y = \grad (s+t\omega)^*(x) \in \Omega^*$ (by Lemma~\ref{lem s + t omega is C1} and Lemma~\ref{lem DspK =DK}), then for all $0<t<t'$, as $\omega(y) < 0$, by Fenchel Inequality Theorem (Theorem~\ref{theo equality for subgradients}),
\begin{align*}
(s+t\omega)^*(x) & = x \cdot y-(s+t\omega)(y) \\
& < x \cdot y-(s+t'\omega)(y) \\
& < \sup_{y'\in \Omega^*}\bp{x \cdot y' -(s+t'\omega)(y')} = (s+t'\omega)^*(x) \, .
\end{align*}

\noindent\textbullet\ Now, let us show that those leaves foliate the whole $K$, that is $K = \bigcup_{t>0} \partial (K+t\Sigma)$. As $\Sigma \subset \C$ and $K$ is a $\C$-convex domain, the inclusion $\bigcup_{t>0} \partial (K+t\Sigma) \subseteq K$ is clear. Let us then prove the reverse inclusion. Let $X=(x,\lambda)\in K, t>0$ and $y \in \partial s^*(x)$. By the Fenchel Inequality Theorem (Theorem~\ref{theo equality for subgradients}),
\begin{equation}
\label{eq Fenchel equality s*}
s^*(x) = x \cdot y-s(y) \, .
\end{equation}
Thus, for all $t \geq 0$, using the Fenchel Inequality \eqref{eq Fenchel Inequality} and then equality \eqref{eq Fenchel equality s*}, 
\begin{equation*}
(s+t\omega)^*(x) \geq x \cdot y-s(y)-t\omega(y) = s^*(x)-t\omega(y) \, .
\end{equation*}
Hence, as $\omega(y) < 0$, we have that $\lim_{t \to +\infty}(s+t\omega)^*(x) = +\infty$. Thus, as $(s+t\omega)^*(x)$ is continuous in $t$ (by Proposition~\ref{prop uniform convergence of conjugates}) and $s^*(x) < \lambda$ (because $X = (x,\lambda) \in K$), by the intermediate value theorem there is some $t>0$ such that $\lambda = (s+t \omega)^*(x)$, i.e. such that $X \in \graph(s+t \omega)^* = \partial (K+t \Sigma)$. Thus, $K \subseteq \bigcup_{t>0} \partial (K+t\Sigma)$.
\end{proof}

\begin{lemma}
\label{lem link between graphs and normal}
Let $K$ be a $\C$-convex domain with support function $s$. For all $t>0$,
\begin{equation*}
\partial (K+t\Sigma)=\left\lbrace X + t \G_{\Sigma}^{-1}(y) \st X \in \partial_{sp} K, \, y \in \G_K(X) \right\rbrace,
\end{equation*}
where $\G_K$, the $\C$-spacelike Gauss map of the $\C$-convex domain $K$, is the set-valued map defined by \eqref{eq gauss map def}, and, as $\omega$ is $C^2_+$, the inverse $\G_{\Sigma}^{-1} : \Omega^* \to \partial_{sp}K(\omega)=\Sigma$ of the Gauss map of $\partial_{sp}K(\omega)$, is a $C^1$ diffeomorphism expressed by \eqref{eq expression of the inverse of the Gauss map}.
\end{lemma}

\begin{proof}
First, let us show that 
\begin{equation*}
\partial (K+t\Sigma) \supseteq\left\lbrace X + t \G_{\Sigma}^{-1}(y) \st X \in \partial_{sp} K, y \in \G_K(X) \right\rbrace .
\end{equation*}
Let $X =(x,s(x))\in \partial_{sp} K$, let $y \in \G_K(X)$ and let $t>0$. We have that $y$ directs a supporting hyperplane to $K$ at $X$, thus $y \in \partial s^*(x)$. By the Fenchel Inequality Theorem (Theorem~\ref{theo equality for subgradients}), we also have $x \in \partial s(y)$ and $s^*(x)=x\cdot y -s(y)$. Then, using the expression of $\G_{\Sigma}^{-1}$ given by \eqref{eq expression of the inverse of the Gauss map}, we get
\begin{equation*}
X + t \G_{\Sigma}^{-1}(y) = \bp{x+t \grad \omega(y), (x+t \grad \omega(y)) \cdot y - (s+t \omega )(y)} \, .
\end{equation*}
As $x \in \partial s(y)$, by Theorem~\ref{theo equality of sum of differentials}, $x_t\coloneqq x+t \grad \omega(y) \in \partial(s+t \omega)(y)$. Hence, by Theorem~\ref{theo equality for subgradients}, we have 
$(s+t \omega )^*(x_t)= x_t \cdot y - (s+t \omega )(y)$
and 
\begin{equation*}
X + t \G_{\Sigma}^{-1}(y) = \bp{x_t, x_t \cdot y - (s+t \omega )(y)} \in \graph(s+t\omega)^* = \partial (K+t\Sigma) \, .
\end{equation*}

Now, let us show that
\begin{equation*}
 \partial (K+t\Sigma) \subseteq \left\lbrace X + t \G_{\Sigma}^{-1}(y) \st X \in \partial_{sp} K, y \in \G_K(X) \right\rbrace .
\end{equation*}
Let $X_t =(x_t,(s+t\omega)^*(x_t)) \in \graph(s+t\omega)^* = \partial (K+t\Sigma)$. Set $y = \grad (s+t\omega)^*(x_t) \in \Omega^*$. By the Fenchel Inequality Theorem (Theorem~\ref{theo equality for subgradients}), we also have $x_t \in \partial (s+t\omega)(y)$ and $(s+t\omega)^*(x_t)=x_t\cdot y -(s+t\omega)(y)$. Thus, using the expression of $\G_{\Sigma}^{-1}$ given by \eqref{eq expression of the inverse of the Gauss map}, we get
\begin{equation*}
X \coloneqq X_t - t \G_{\Sigma}^{-1}(y) = \bp{x_t-t \grad \omega(y), (x_t-t \grad \omega(y)) \cdot y - s(y)} \, .
\end{equation*}
As $x_t \in \partial (s + t \omega)(y)$, by Theorem~\ref{theo equality of sum of differentials}, $x \coloneqq x_t - t \grad \omega (y) \in \partial s(y)$. Hence, by Theorem~\ref{theo equality for subgradients}, we have $s^*(x) = x \cdot y - s(y)$ and 
\begin{equation*}
 X = X_t - t \G_{\Sigma}^{-1}(y)= \bp{x, x \cdot y - s(y)} \in \graph_{sp}(s^*)= \partial_{sp}K \, . \qedhere
\end{equation*}
\end{proof}

Lemmas~\ref{lem C1 foliation of K} and~\ref{lem link between graphs and normal} together give the following.

\begin{proposition}
\label{prop decomposition}
Let $K$ be a $\C$-convex domain with support function $s$. Then, any point $X \in K$ is uniquely decomposed as
\begin{equation*}
X = P + t \G_{\Sigma}^{-1}(y) \, ,
\end{equation*}
where $t>0$, $P\in \partial_{sp} K$ and $y \in \G_K(P)\subseteq \Omega^*$ (see Figure~\ref{Normal projection figure}). Moreover, one has 
\begin{equation*}
y =\G_{K+t\Sigma}(X) =\grad (s+t\omega)^*\bp{\pi (X)}.
\end{equation*}
\end{proposition}

\begin{figure}[ht]
\centering
\includegraphics{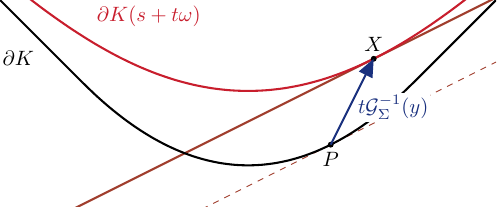}
\caption{Normal projection onto the $\C$-spacelike boundary of a $\C$-convex.}
\label{Normal projection figure}
\end{figure}

\begin{definition}[Normal projection, cosmological time and projecting normal]
\label{def normal projection, cosmological time and projecting normal}
Let $K$ be a $\C$-convex domain. In the decomposition of some element $X \in K$ from Proposition~\ref{prop decomposition},
\begin{itemize}
\item we shall denote $P\in \partial_{sp} K$ by $P_K(X)$ and call it the \emph{normal projection} of $X$ onto $\partial_{sp}K$,
\item we shall denote $t\in \R_+^*$ by $\mathcal{T}_K(X)$ and call it the \emph{cosmological time of $K$} at $X$, 
\item we shall denote $y\in\Omega^*$ by $y_K(X)$ and call it the \emph{projecting normal} from $X$ to $\partial_{sp}K$.
\end{itemize}
\end{definition}

\begin{proposition}
\label{prop TK, PK and yK are continuous}
Let $K$ be a $\C$-convex domain with support function $s$. The cosmological time $\mathcal{T}_K: K \to \R^*_+$, normal projection $P_K: K \to \partial_{sp}K$ and projecting normal $y_K: K \to \Omega^*$ are all continuous.
\end{proposition}
\begin{proof}
First, let us prove that the cosmological time is continuous. The closed epigraph of a continuous function $f: \R^d \to \R$ is the closed set
\begin{equation*}
\overline{\epi}(f) \coloneqq \left\lbrace (x,\lambda) \in \R^d \times \R \st \lambda \geq f(x) \right\rbrace .
\end{equation*}
For all $I = ]a,b[\subset \R_+^*$, as the functions $(s+a\omega)^*$ and $(s+b\omega)^*$ are continuous, we have that
\begin{equation*}
\mathcal{T}_K^{-1} (I) = \bigcup_{t \in I}\partial K(s + t\omega) = \epi(s+a\omega)^* \setminus \overline{\epi} (s+b\omega)^*
\end{equation*}
is open, so $\mathcal{T}_K$ is continuous.

Now, let us prove that $y_K$ is continuous. Let $(X_n)_{n\in\N}$ be a sequence of points of $K$ converging to $X$. For all $n\in\N$, $y_K(X_n) =\grad (s+\mathcal{T}_K(X_n)\omega)^*(\pi (X))$. By continuity of $\mathcal{T}_K$, $\mathcal{T}_K(X_n)$ converges to $\mathcal{T}_K(X)$, and thus, $(s+\mathcal{T}_K(X_n)\omega)^*$ converges pointwise to $(s+\mathcal{T}_K(X)\omega)^*$ by Proposition~\ref{prop uniform convergence of conjugates}. Hence, by \cite[Theorem~24.5]{Rockafellar_97}, $y_K(X_n)$ converges to $y_K(X)=\grad (s+\mathcal{T}_K(X)\omega)^*(\pi (X))$. 

Finally, as $P_K = \Id_K - \mathcal{T}_K \times (\G_{\Sigma}^{-1} \circ y_K)$, the continuities of $\mathcal{T}_K$, $\G_{\Sigma}^{-1}$ and $y_K$ imply the continuity of $P_K$. 
\end{proof}

We shall denote the level sets of the cosmological time by 
\begin{equation*}
S_K^t \coloneqq \mathcal{T}_K^{-1}(t) = \partial( K +t \Sigma)= \graph(s+t\omega)^* \, .
\end{equation*}

\begin{figure}[ht]
\centering
\includegraphics{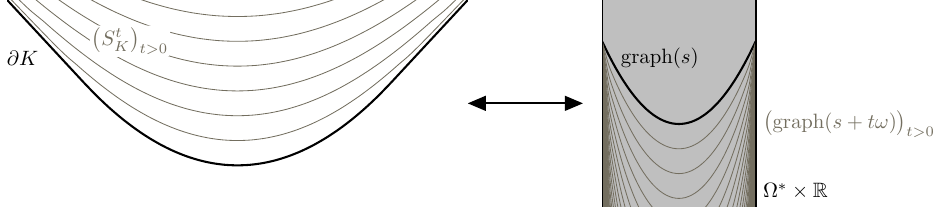}
\caption{Foliation of a $\C$-convex by level hypersurfaces of its cosmological time.}
\label{Cosmological foliation figure}
\end{figure}

\begin{proposition}
\label{prop TK is C1}
Let $K$ be a $\C$-convex domain with support function $s$. The cosmological time $\mathcal{T}_K$ on $K$ is $C^1$ and its gradient at any point $X \in K$ is expressed as 
\begin{equation*}
\grad \mathcal{T}_K (X) = \frac{\bp{y_K(X),-1}}{\omega\bp{y_K(X)}} \, ,
\end{equation*}
where $y_K(X)$ is the projecting normal from $X$ to $\partial_{sp}K$.
\end{proposition}
\begin{proof}
Let $X \in K$. In order to lighten the notation, we denote by $P \coloneqq P_K(X) \in \partial_{sp}K$ the normal projection of $X$ onto $\partial_{sp}K$, by $y \coloneqq y_K(X)$ the projecting normal and $N = \G_{\Sigma}^{-1}(y)$. The cone $I^+(P) = P + \C$, is foliated by its intersections with:
\begin{itemize}
\item[-] the level surfaces of the cosmological time $(S^t_K)_{t>0}$,
\item[-] the hyperplanes directed by $y$, i.e. $(P + t \, H_y)_{t>0}$, where $H_y = \{ V \in \C \ \vert \ V \cdot (y,-1) = \omega(y) \}$,
\item[-] the $\C$-spacelike hypersufaces $(P + t \, \Sigma)_{t>0}$.
\end{itemize}
For any vector $V \in \C$, we set $g(V) \coloneqq \mathcal{T}_K(P+V)$, it is the unique $t>0$ such that $P+V \in S^t_K$. Similarly we define $f(V)$ to be the unique $t>0$ such that $V \in t H_y$, it can explicitly be expressed as
\begin{equation*}
\function{f}{\C}{\R}{V}{\frac{V \cdot (y,-1)}{\omega(y)}} \, ,
\end{equation*}
and we set $h(V)$ to be the unique $t>0$ such that $V \in t \Sigma$. 

Inspired by the proof of Bonsante \cite[Proposition~4.3]{Bonsante_05} in the Lorentzian case, we claim that those satisfy the hypothesis of the following fact: 

 Let $U$ be an non-empty open subset of $\R^{d+1}$ and $f,g,h: U \to \R$ be maps such that:
\begin{itemize}
\item $f \leq g \leq h$,
\item $f$ and $h$ are $C^1$,
\item there exist a point $X_0 \in U$ such that $f(X_0)=g(X_0)=h(X_0)$ and $\dif_{X_0} f = \dif_{X_0} h$.
\end{itemize}
Then, $g$ is differentiable at $X$ and $\dif_{X_0} f = \dif_{X_0} g = \dif_{X_0} h$.

Let us prove our claim.

\begin{itemize}
\item Notice that for all $t>0$, $S^t_K = \graph(s+t \omega)^*$ and $P + t \,\Sigma = \graph(a_p+t \omega)^*$, where $a_P(y) = P\cdot (y,-1)$ is the support function of the future cone $I^+(P) = P + \C$. Notice that because $P \in \partial K(s)$, we have $a_P \leq s$ and thus, using the order reversing property of the Legendre--Fenchel transform, $(s+t \omega)^* \leq (a_p+t \omega)^*$ are convex maps with graphs having the common supporting hyperplane $P+t \, H_y$ at $P + t N$ (see Figure~\ref{fig f<g<h}). Hence, by definition of $f,g$ and $h$, we have $f \leq g \leq h$.
\item Because of the smoothness of hyperplanes and affine spheres, the maps $f$ and $g$ are clearly $1$-homogeneous smooth maps. 
\item Finally for all $t>0$, $f(tN)=g(tN)=h(tN)=t$ and 
\begin{equation}
\label{eq Grad expression}
\grad f (tN) = \grad h(tN) = \frac{(y,-1)}{\omega(y)} \, .
\end{equation}
\end{itemize}

Thus, for all $t>0$, $g$ is differentiable at $tN$ with gradient expressed by equation \eqref{eq Grad expression}. Hence,  as for all $V \in \C$, $\mathcal{T}_K(P+V)=g(V)$ and $X = P + \mathcal{T}_K(X) N$, $\mathcal{T}_K$ is differentiable at $X$ and its gradient at any point $X \in K$ is expressed as 
\begin{equation*}
\grad \mathcal{T}_K (X) = \frac{(y ,-1)}{\omega(y)} = \frac{\bp{y_K(X),-1}}{\omega\bp{y_K(X)}} \, .
\end{equation*}
Then, the continuity of the projecting normal $y_K$ (Proposition~\ref{prop TK, PK and yK are continuous}) implies that the cosmological time $\mathcal{T}_K$ is $C^1$.
\end{proof}

\begin{figure}[ht]
 \centering
 \includegraphics{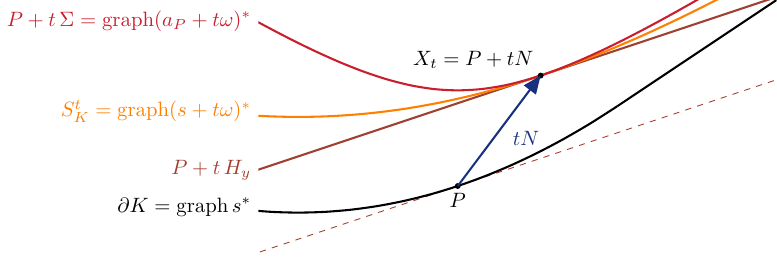}
 \caption{Situation in the proof of Proposition~\ref{prop TK is C1}.}
 \label{fig f<g<h}
\end{figure}

\begin{proposition}
\label{prop level set of cosmo time is Cauchy surf}
Let $K$ be a $\C$-convex domain with support function $s$. Every level set of the cosmological time $\mathcal{T}_K$ on $K$ is a future-convex Cauchy surface of the affine spacetime $(K,\C)$. 
\end{proposition}
\begin{proof}
Let $t>0$. The $C^1$ hypersurface $S^t_K=\left\lbrace \mathcal{T}_K =t \right\rbrace$ is the boundary of the $\C$-spacelike $\C$-convex domain $K+t\Sigma$, having support function $s+t\omega$. Because of condition~\ref{cdt 0 on dOmega} from Definition~\ref{def gauge support function}, the lower semi-continuous convex extensions (see Remark~\ref{rem extension of convex function}) of $s$ and $s+t\omega$ to $\overline{\Omega^*}$ coincide on $\partial \Omega^*$. Thus, by Proposition~\ref{prop maximal g domain = Cauchy development}, $\partial (K+t\Sigma)$ is a Cauchy surface of $(K,\C)$. The $\C$-future-convexity of the Cauchy surface $\partial (K+t\Sigma)$, in the sense of Definition~\ref{def Convex Cauchy surface}, is clear.
\end{proof}

\subsection{\texorpdfstring{Length of $\C$-causal curves in the affine spacetime $(\A^{d+1},\C)$}{Length of C-causal curves in the affine spacetime (A,C) }}
\label{subseclength of causal curves}

This subsection discusses the cosmological time through the framework of pseudo-Finsler geometry described in \cite{Papadopoulos_Yamada_19, Buro_23} and based on the work of Busemann \cite{Busemann_67}. In that context we proceed to give a geometric interpretation of cosmological time, justifying its name \cite{Andersson_Galloway_Howard_98}.

In the model affine spacetime $(\R^{d+1},\C)$, using our fixed gauge support functions $\omega$ satisfying properties~\ref{cdt C2+},~\ref{cdt 0 on dOmega} and~\ref{cdt empty subdif on dOmega}, from Definition~\ref{def gauge support function}, we can define a \emph{pseudo-Finsler} norm of a $\C$-timelike vector as the \emph{gauge} of $\Sigma$: a vector $V \in \C $ has norm $\lambda >0$ if and only if $V \in \lambda \Sigma$. As $\Sigma$ is asymptotic to $\C$, the norm of a $\C$-null vector is set to $0$, so that the norm is continuous on $\overline{\C}$.

\begin{remark}
Here \(K(\omega)\), the $\C$-convex domain bounded by $\Sigma$, is an example of a Valinor convex domain introduced by Buro in order to introduce pseudo-Finsler metrics on manifolds \cite[Section 3.2]{Buro_23}.
\end{remark}

Because $\Sigma = \graph(\omega^*)\subseteq \C$ is the graph of a strictly convex $C^1$ function asymptotic to $\partial\C$, we can parametrise $\Sigma$ radially, i.e. set $w: \Omega \to \R$ to be the convex function such that 
\begin{equation*}
\Sigma = \left\lbrace -\frac{1}{w(x)}(x,1) \st x \in \Omega \right\rbrace .
\end{equation*}
Then, a direct computation shows that the pseudo-Finsler norm of vectors $V = (v,\nu) \in \C \subset \R^d \times \R $ is expressed as the positive 1-homogeneous function
\begin{equation*}
F(V) \coloneqq -\nu \, w\vp{\frac{v}{\nu}} \, .
\end{equation*}

\begin{remark} 
When using $\omega =\omega_\C$ the support function of the Cheng--Yau affine sphere $\S$ given by Theorem~\ref{theo Cheng--Yau V2}, notice that Remark~\ref{rem the Minkowski case} implies that in the Minkowski case it precisely coincides with the Lorentzian norm of a causal vector. In the general case, by Gigena's work \cite{Gigena_81}, the radial parametrisation of an hyperbolic affine sphere is
\begin{equation*}
\S = \left\lbrace -\frac{1}{w_\C(x)}(x,1) \st x \in \Omega \right\rbrace,
\end{equation*}
where $w_\C = \omega_{(\C^*)}$ is the unique solution to the Monge--Amp\`ere equation 
\begin{equation*}
\begin{cases}
\det \Hess w= (-w)^{-d-2} & \text{in} \ \Omega \, ,\\
w_{\vert \partial\Omega} =0 \, .
\end {cases}
\end{equation*} 
\end{remark}

\begin{definition}[Length of a $\C$-causal curve and $\C$-distance]
The \emph{length of a $\C$-causal curve} $\gamma: I \to \R^{d+1}$ is defined as 
\begin{equation*}
l_F(\gamma) \coloneqq \int_I F\bp{\gamma'(t)} \dif t \geq 0 \, .
\end{equation*}
The homogeneity of the function $F$ ensures that it does not depend on the parametrisation of the path $\gamma$. 

Given $X_0 \in \R^{d+1}$ and $X_1 \in J^+(X_0) = X_0 + \C$, the \emph{$\C$-distance between $X_0$ and $X_1$} is 
\begin{equation*}
\rho(X_0,X_1) \coloneqq \sup \left\lbrace l_F(\gamma) \st \gamma \ \text{is a future directed $\C$-causal curve from} \ X_0 \ \text{to} \ X_1 \right\rbrace .
\end{equation*}
\end{definition}

\begin{proposition}
Let $X_0 \in \R^{d+1}$ and $X_1 \in J^+(X_0) = X_0 + \C$. The $\C$-distance between $X_0$ and $X_1$ is attained for the straight line from $X_0$ to $X_1$, i.e. 
\begin{equation*}
\rho(X_0,X_1) = F(X_1 - X_0) \, .
\end{equation*}
\end{proposition}
\begin{proof}
Let $\gamma: \left[0,1\right] \to \R^{d+1}$ be a future directed $\C$-causal curve from $\gamma(0) = X_0$ to $\gamma(1) = X_1$. We decompose it as $\gamma(t)=(x(t),\lambda(t))$, then up to reparametrisation we can assume that $\lambda(t) = (1-t)\lambda_0 + t \lambda_1$. Then 
\begin{equation*}
l_F(\gamma) = \int_{0}^{1} F\bp{\gamma'(t)} \dif t = -\int_{0}^{1} \lambda'(t) w \vp{\frac{x'(t)}{\lambda'(t)}} \dif t = -(\lambda_1 - \lambda_0) \int_{0}^{1} w \vp{\frac{x'(t)}{\lambda_1 - \lambda_0}}\dif t \, .
\end{equation*}
As $w$ is convex, the Jensen inequality gives
\begin{equation*}
\int_{0}^{1} w \vp{\frac{x'(t)}{\lambda_1 - \lambda_0}}\dif t \geq w \vp{\int_{0}^{1} \frac{x'(t)}{\lambda_1 - \lambda_0}\dif t} = w \vp{\frac{x_1 - x_0}{\lambda_1 - \lambda_0}} \, .
\end{equation*}
Because the path is future directed, $\lambda_1 - \lambda_0 \geq 0$, and thus
\begin{equation*}
l_F(\gamma) \leq -(\lambda_1 - \lambda_0) w\vp{\frac{x_1 - x_0}{\lambda_1 - \lambda_0}} = F(X_1 - X_0) \, ,
\end{equation*}
where $F(X_1 - X_0)$ the length of the straight line from $X_0$ to $X_1$. 
\end{proof}

\begin{proposition}[Time inequality]
\label{prop time inequality}
If $X_1,X_2,X_3 \in \R^{d+1}$ are causally ordered, i.e. such that $X_2 \in J^+(X_1)$ and $X_3 \in J^+(X_2) \subseteq J^+(X_3)$, we have the \emph{time inequality}
\begin{equation*}
\rho(X_1,X_2) + \rho(X_2,X_3) \leq \rho(X_1,X_3) \, .
\end{equation*}
\end{proposition}
\begin{proof}
Set $X_i = (x_i, \lambda_i)$, for all $i \in \left\lbrace 1,2,3 \right\rbrace$ As $w$ is convex we have 
\begin{equation*}
\frac{\lambda_2 - \lambda_1}{\lambda_3 - \lambda_1} w\vp{\frac{x_2 - x_1}{\lambda_2 - \lambda_1}} + \frac{\lambda_3 - \lambda_2}{\lambda_3 - \lambda_1} w\vp{\frac{x_3 - x_2}{\lambda_3 - \lambda_2}} \geq w \vp{\frac{x_3 - x_1}{\lambda_3 - \lambda_1}} \, .
\end{equation*}
Multiplying by $-(\lambda_3 - \lambda_1) \leq 0$, we get the wanted result.
\end{proof}

\begin{remark}
This time inequality is one of the axioms describing a \emph{timelike space} in the sense of Busemann \cite{Busemann_67}. One can easily check that a $\C$-convex domains endowed with the causal structure induced by $\C$ and the time function $\rho$ satisfies all the other axioms.
\end{remark}

\begin{proposition}
\label{prop cosmological time and length causal curve}
Let $K$ a $\C$-convex domain. The cosmological time defined in Definition~\ref{def normal projection, cosmological time and projecting normal} satisfies
\begin{align*}
\mathcal{T}_K(X) & = \sup \left\lbrace \rho(X_0,X) \st X_0 \in J^-(X) \cap \partial_{sp}K \right\rbrace \\
& = \sup \left\lbrace l_F(\gamma) \st \gamma \ \text{is a future directed $\C$-causal curve from a point of $\partial_{sp}K$ to $X$} \right\rbrace .
\end{align*}
\end{proposition}
\begin{proof}
First, notice that $T \coloneqq \mathcal{T}_K(X) = F(X - P) = \rho(P,X)$ where $P \coloneqq P_K(X) \in J^-(X) \cap \partial_{sp}K$ is the normal projection of $X$ onto $\partial_{sp}K$. Then notice that $X - T \, \Sigma$ and $\partial_{sp}K$ touch at $P$ and have a common $\C$-spacelike hyperplane directed by $y \coloneqq y_K(X)$ the projecting normal from $X$ to $\partial_{sp}K$. Then, using the foliation of $(X - \C)$ by the $-\C$-convex hypersurfaces $(X - t \, \Sigma)_{t>0}$, as $K$ is $\C$-convex, it is clear that any element $X_0 \in (X - \C) \cap \partial_{sp}K$ belongs to some $X - T_0 \cdot \Sigma$ with $T_0 \leq \mathcal{T}_K(X)$, i.e. $\rho(X_0,X) \leq \mathcal{T}_K(X)$ with equality for $X_0 = P$ (see Figure~\ref{fig proof length}). 
\end{proof}

\begin{figure}[ht]
\centering
\includegraphics{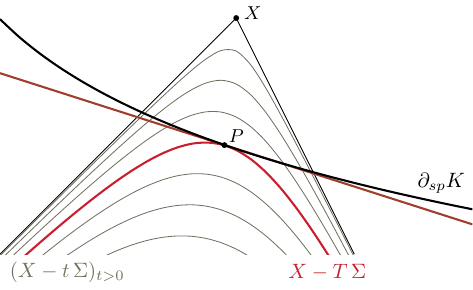}
\caption{Situation in the proof of Proposition~\ref{prop cosmological time and length causal curve}.}
\label{fig proof length}
\end{figure}

\begin{remark}
This new characterisation of cosmological time highlights its similarity with the Lorentzian one \cite{Andersson_Galloway_Howard_98}. When using $\omega =\omega_\C$, the support function of the Cheng--Yau affine sphere given by Theorem~\ref{theo Cheng--Yau V2}, as our gauge support function, Remark~\ref{rem the Minkowski case} implies that in the Minkowski case it precisely coincides with the Lorentzian cosmological time. 
\end{remark}

\begin{proposition}
\label{prop TK is concave}
Let $K$ be a $\C$-convex domain. The cosmological time $\mathcal{T}_K$ on $K$ is concave.
\end{proposition}

\begin{proof}
Let $X_1,X_2 \in K$ be two points, we denote by $P_i \coloneqq P_K(X_i) \in \partial_{sp}K$ their normal projections. Then we have for $i \in \{1,2\}$,
\begin{equation*}
\mathcal{T}_K(X_i) = \rho(P_i, X_i) = F(X_i - P_i) \, .
\end{equation*}
Starting from
\begin{align*}
\frac{\mathcal{T}_K(X_1) + \mathcal{T}_K(X_2)}{2} & = \frac{F(X_1 - P_1) + F(X_2 -P_2)}{2} \\
& = \frac{F(X_1 - P_1) + F\bp{(X_1 + X_2 - P_2)-X_1}}{2} \\
& = \frac{\rho(P_1, X_1) + \rho(X_1,X_1 + X_2 - P_2)}{2} \, ,
\end{align*}
and using the time inequality (Proposition~\ref{prop time inequality}),
\begin{equation*}
\rho(P_1,X_1) + \rho(X_1,X_1 + X_2 - P_2) \leq \rho(P_1,X_1 + X_2 - P_2) = \rho(P_1 + P_2 ,X_1 + X_2) \, ,
\end{equation*}
we get 
\begin{equation*}
\frac{\mathcal{T}_K(X_1) + \mathcal{T}_K(X_2)}{2} \leq \frac{\rho(P_1 + P_2 ,X_1 + X_2)}{2} = \rho\vp{\frac{P_1 + P_2}{2} ,\frac{X_1 + X_2}{2}} \, . 
\end{equation*}
As $K$ is convex and $P_1,P_2 \in \partial_{sp}K$, we have that $(P_1 + P_2)/2 \in \overline{K}$ and thus, by Proposition~\ref{prop cosmological time and length causal curve},
\begin{equation*}
\frac{\mathcal{T}_K(X_1) + \mathcal{T}_K(X_2)}{2} \leq \rho\vp{\frac{P_1 + P_2}{2} ,\frac{X_1 + X_2}{2}} \leq \mathcal{T}_K \vp{\frac{X_1 + X_2}{2}} \, .\qedhere
\end{equation*}
\end{proof}

\subsection{Invariance of the cosmological time given by the Cheng--Yau affine spheres}
\label{subsec Invariance of the cosmological time given by the CHeng--Yau affine spheres}

In this subsection, we shall prove Theorem~\ref{Theorem Intro cosmological time}. We shall consider the cosmological time on $\C$-convex domains induced by the gauge support function $\omega_\C$, the support function of the unique Cheng--Yau affine sphere $\S$ associated with $\C$ (see Theorem~\ref{theo Cheng--Yau V2}). We shall call it the \emph{affine sphere cosmological time}.

Let us recall Corollary~\ref{cor affine sphere invariance}: the affine sphere $\S=\graph(\omega_\C^*)$ is preserved by $\mathrm{Aut}_{\SL}(\C)$, the subgroup of elements of $\SL(\R^{d+1})$ preserving the cone $\C$. Hence, the affine sphere cosmological time satisfies the following invariance property.

\begin{proposition}
\label{prop cosmological time equivariance}
Let $K$ be a $\C$-convex domain and $\mathcal{T}_K$, $P_K$ and $y_K$ be be the cosmological time, normal projection and projecting normal on $K$ (Definition~\ref{def normal projection, cosmological time and projecting normal}) induced by the gauge support function $\omega_\C$ (Theorem~\ref{theo Cheng--Yau V2}). If $K$ is preserved by an affine transformation $(\gamma,\tau) \in \SL(\R^{d+1})\ltimes\R^{d+1}$, then $\gamma \in \mathrm{Aut}_{\SL}(\C)$, the subgroup of elements of $\SL(\R^{d+1})$ preserving $\C$, and for all $X$ in $K$ we have the following
\begin{align*}
& \mathcal{T}_K(\gamma X + \tau) = \mathcal{T}_K(X) \, , \\
& P_K(\gamma X + \tau) = \gamma P_K(X) + \tau \, ,\\
& y_K(\gamma X + \tau) = \gamma *y_K(X) \, ,
\end{align*}
where the projective dual action of an element $\gamma \in \mathrm{Aut}_{\SL}(\C)$ on $\Omega^*$, denoted by $*$, is described at the beginning of Section~\ref{sec Affine deformations of divisible convex cones}.
\end{proposition}
\begin{proof}
Let $(\gamma,\tau) \in \SL(\R^{d+1})\ltimes\R^{d+1}$ be an affine transformation preserving the $\C$-convex domain $K$. As $(\gamma,\tau)$ preserves $K$, its linear part $\gamma$ preserves its set of supporting hyperplane. By definition of a $\C$-convex domain (Definition~\ref{def C-convex}), that implies that the dual action of $\gamma$ preserves the cone $\C^*$ and hence, by duality, that $\gamma$ preserves $\C$, i.e. $\gamma \in \mathrm{Aut}_{\SL}(\C)$.

Let $X \in K$. Acting by $(\gamma,\tau)$ on the equality
\begin{equation*}
X = P_K(X) + \mathcal{T}_K(X) \cdot \G_{\S}^{-1}\bp{y_K(X)}
\end{equation*}
that defines the cosmological time, normal projection and projecting normal of $X$ (Definition~\ref{def normal projection, cosmological time and projecting normal}), we get 
\begin{equation*}
\gamma X + \tau = \bp{\gamma P_K(X) + \tau} + \mathcal{T}_K(X) \cdot \gamma \, \G_{\S}^{-1}\bp{y_K(X)} \, .
\end{equation*}
As the affine sphere $\S$ is invariant under the action of $\mathrm{Aut}_{\SL}(\C)$ (Corollary~\ref{cor affine sphere invariance}), its Gauss map is clearly $\mathrm{Aut}_{\SL}(\C)$-equivariant, and so is its inverse. Thus, 
using that equivariance of $\G_{\S}^{-1}$, the previous equality becomes
\begin{equation}
\label{eq projection equation composed by gamma}
\gamma X + \tau = \bp{\gamma P_K(X) + \tau} + \mathcal{T}_K(X) \cdot \G_{\S}^{-1}\bp{\gamma *y_K(X)} \, .
\end{equation}
As $K$ is preserved by the affine transformation $(\gamma,\tau)$ and $\gamma \in \mathrm{Aut}_{\SL}(\C)$, its $\C$-spacelike boundary $\partial_{sp}K$ is also preserved by $(\gamma,\tau)$. Thus, $\gamma P_K(X) + \tau \in \partial_{sp}K$ and $\gamma * y_K(X)$ directs a supporting hyperplane to $K$ at that point. Hence, equality \eqref{eq projection equation composed by gamma} characterises the cosmological time, normal projection and projecting normal of $\gamma X + \tau$ (Definition~\ref{def normal projection, cosmological time and projecting normal}), giving the desired expressions.
\end{proof}

We can now prove the following theorem, implying Theorem~\ref{Theorem Intro cosmological time}.

\begin{theorem}
Let $K$ be a $\C$-convex domain. The affine sphere cosmological time $\mathcal{T}_K: K \to \R^*_+$ is a concave $C^1$ function whose level sets are $C^1$ $\C$-future-convex Cauchy surfaces of $(K,\C)$ foliating $K$.
Moreover, $\mathcal{T}_K$ is invariant by any transformation of $\SA(\R^{d+1})$ preserving $K$.
\end{theorem}
\begin{proof}
This is a consequence of Lemma~\ref{lem C1 foliation of K} and Propositions~\ref{prop TK is C1}, \ref{prop level set of cosmo time is Cauchy surf}, \ref{prop TK is concave} and \ref{prop cosmological time equivariance}.
\end{proof}

\section{Quotient of the maximal domain}
\label{sec Quotient of the maximal domain}

In this section, we first study the action of $\Gamma_\tau$, an affine deformation of a group $\Gamma$ dividing an open proper convex cone $\C$, on the maximal $\tau$-convex domain $D_\tau$ from Theorem~\ref{Theorem Intro domain} constructed in Section~\ref{sec Affine deformations of divisible convex cones}. We then focus on the affine spacetime $(D_\tau / \Gamma_\tau, \C)$, using the affine sphere cosmological time introduced in Section~\ref{sec Cosmological time on C-convex domains}. That allows us to produce proofs for point~\ref{theo11} and the GHCC part of point~\ref{theo12} of Theorem~\ref{Theorem Intro domain}.

\subsection{\texorpdfstring{Action of $\Gamma_\tau$ on $D_\tau$}{Action of Gamma-tau on D-tau}}
\label{subsec Action on the maximal domain}
In this subsection we shall prove point~\ref{theo11} from Theorem~\ref{Theorem Intro domain}.
We come back to the case where $\C$ happens to be divisible by a torsion-free discrete subgroup $\Gamma < \SL(\R^{d+1})$ and let $\Gamma_\tau$ be an affine deformation of $\Gamma$. We can use the tools introduced in subsection~\ref{subsec Normal projection, cosmological time and foliation of C-convex domains} on the maximal $\tau$-convex domain $D_\tau$ constructed in Section~\ref{sec Affine deformations of divisible convex cones}.

From now on we shall use the support function $\omega_\C$ of the affine sphere $\S$ as a gauge support function. It provides a normal projection onto $\partial_{sp}D_\tau$ denoted by $P_\tau$, an affine sphere cosmological time on $D_\tau$ denoted by $\mathcal{T}_\tau$ with $t$-level hypersurface denoted by $S_\tau^t$, and a projecting normal denoted by $y_\tau$. In order to lighten notations we shall simply write that $\mathcal{T}_\tau$ is the cosmological time on $D_\tau$.

\begin{proposition}
\label{prop tau cosmological time equivariance}
On $D_\tau$ the maximal $\tau$-convex domain:
\begin{itemize}
\item the cosmological time $\mathcal{T}_\tau: D_\tau \to \R_+^*$ is $\Gamma_\tau$-invariant, 
\item the normal projection $P_\tau: D_\tau \to \partial_{sp} D_\tau$ is $\Gamma_\tau$-equivariant,
\item the projecting normal $y_\tau: D_\tau \to \Omega^*$ is equivariant with respect to the action of $\Gamma_\tau$ on $D_\tau$ and the action of $\Gamma$ on $\Omega^*$.
\end{itemize}
\end{proposition}
\begin{proof}
This a direct consequence from Proposition~\ref{prop cosmological time equivariance}.
\end{proof}

\begin{proposition}
\label{prop cosmological foliation of Dtau}
The maximal $\tau$-convex domain $D_\tau$ is foliated by the level hypersurfaces of its cosmological time, which are the $\C$-spacelike $C^1$ embedded $\Gamma_\tau$-invariant hypersurfaces $\{\partial K(s_\tau + t \omega_\C)\}_{t>0}$.
\end{proposition}
\begin{proof}
The foliation is simply a consequence of Lemma~\ref{lem C1 foliation of K}, the $\tau$-invariance comes from the equivariance of $\mathcal{T}_\tau$ (Proposition~\ref{prop tau cosmological time equivariance}), or equivalently from the $\tau$-equivariance of any function $s_\tau + t \omega_\C$ with $t>0$.
\end{proof}

\begin{proposition}
\label{prop projection Ptau}
The projection $\Pi_\tau$ defined by 
\begin{equation*}
\function{\Pi_\tau}{D_\tau}{\partial (D_\tau + \S) = S_\tau^1}{X}{P_{\tau}(X) + \G_{\S}^{-1}\bp{y_{\tau}(X)}}
\end{equation*}
is continuous and $\Gamma_\tau$-equivariant, in the sense that for all $(\gamma, \tau_\gamma) \in \Gamma_\tau$,
\begin{equation*}
\Pi_\tau(\gamma X + \tau_\gamma) =\gamma \Pi_\tau(X) + \tau_\gamma \, .
\end{equation*} 
\end{proposition}
\begin{proof}
The continuity and equivariance are consequences of, respectively, Propositions~\ref{prop TK, PK and yK are continuous} and~\ref{prop tau cosmological time equivariance}.
\end{proof}

\begin{corollary}
\label{cor definition of Theta}
There is a homeomorphism
\begin{equation*}
\function{\Theta}{D_\tau}{S_\tau^1 \times \R_+^*}{X}{\bp{\Pi_\tau(X),\mathcal{T}_{\tau}(X)}}
\end{equation*}
satisfying that for all $(\gamma,\tau_\gamma) \in \Gamma_\tau$
\begin{equation}
\label{prop property of Theta}
\Theta(\gamma X + \tau_\gamma) = \bp{\gamma \Pi_\tau(X) + \tau_\gamma,\mathcal{T}_{\tau}(X)} \, .
\end{equation}
\end{corollary}
\begin{proof}
The continuity of $\Theta$ and condition \eqref{prop property of Theta} are consequences of Proposition~\ref{prop projection Ptau}. To see that $\Theta$ is a homeomorphism just notice that its inverse map is explicitly expressed as
\begin{equation*}
\Theta^{-1}=\namelessfunction{S_\tau^1 \times \R_+^*}{D_\tau}{(X,t)}{ P_{\tau}(X) + t \, \G_{\S}^{-1}\bp{y_{\tau}(X)}} \, ,
\end{equation*}
which is continuous by Proposition~\ref{prop TK, PK and yK are continuous}.
\end{proof}

\begin{proposition}
\label{prop action is properly discontinuous}
The action of $\Gamma_\tau$ on $D_\tau$ is free and properly discontinuous and there is a homeomorphism
\begin{equation*}
D_\tau /\Gamma_\tau \simeq S_\tau^1 /\Gamma_\tau \times \R \, .
\end{equation*}
\end{proposition}
\begin{proof}
Using the homeomorphism $\Theta$ from Corollary~\ref{cor definition of Theta} this is a consequence of the free and properly discontinuous action of $\Gamma_\tau$ on the $C^1$ hypersurface $S_\tau^1 = \partial (D_\tau + \S)$ proved in Proposition~\ref{prop free and properly discontinuous action on C1 tau-invariant hypersurfaces}.
\end{proof}

\begin{lemma}
\label{lem concrete homeomorphism between Cauchy surfaces}
Let $K$ is a $\C$-spacelike $\tau$-convex domain of $(\R^{d+1},\C)$ with $C^1$ boundary. Then $\partial K$ is a Cauchy surface of $(D_\tau,\C)$ and it is $\Gamma_\tau$-equivariantly homeomorphic to $S_\tau^1$.
\end{lemma}
\begin{proof}
The fact that $\partial K$ is a Cauchy surface of $D_\tau$ is a driect consequence of Proposition~\ref{prop maximal domain = Cauchy development}. An explicit $\Gamma_\tau$-equivariant homeomorphism between $\partial K$ and $S_\tau^1$ is given by the restriction $\Pi_{\tau\vert \partial K}$ of the projection from Proposition~\ref{prop projection Ptau}.
\end{proof}

\begin{remark}
In Lemma~\ref{lem concrete homeomorphism between Cauchy surfaces}, the fact that $\partial K$ is $\Gamma_\tau$-equivariantly homeomorphic to $S_\tau^1$ can also be seen as a direct consequence of Proposition~\ref{prop Geroch Cauchy surfaces} for the affine spacetime $(D_\tau / \Gamma_\tau,\C)$. 
\end{remark}

\begin{theorem}[Theorem~\ref{Theorem Intro domain}~\ref{theo11}]
\label{the action is properly discontinuous + homeo}
The action of $\Gamma_\tau$ on $D_\tau$ is free and properly discontinuous and there is a homeomorphism
\begin{equation*}
D_\tau /\Gamma_\tau \simeq \Omega /\Gamma \times \R \, .
\end{equation*}
\end{theorem}
\begin{proof}
By Proposition~\ref{prop exists tau-support function extending to boundary}, there exist a $C^2_+$ $\tau$-convex $K$ with $\partial K = \partial_{sp}K$. Then the Gauss map $\G_K$ on $K$ gives a $C^1$ diffeomorphism $\partial K /\Gamma_\tau \approx \Omega^*/\Gamma$. Propositions~\ref{prop dual convex is divisible} and~\ref{prop action is properly discontinuous}, and Lemma~\ref{lem concrete homeomorphism between Cauchy surfaces} then conclude the proof.
\end{proof}

\subsection{Spacetime structure of the quotient}
\label{subsec Spacetime structure of the quotient}

In this subsection, we prove the GHCC part of point~\ref{theo12} (we shall treat maximality in subsection~\ref{subsec Maximality of the quotient spacetimes}) from Theorem~\ref{Theorem Intro domain}.

As $\Gamma$, the linear part of $\Gamma_\tau$, preserves the cone $\C$, the parallel proper cone structure of $(D_\tau,\C)$ descends to the equiaffine manifold $D_\tau / \Gamma_\tau$, endowing it with an affine spacetime structure. 

Proposition~\ref{prop tau cosmological time equivariance} indicates that the cosmological time, normal projection and projecting normal descends to maps $\mathcal{T}_\tau: D_\tau / \Gamma_\tau \to \R_+^*$, $P_\tau: D_\tau / \Gamma_\tau \to \partial_{sp}D_\tau / \Gamma_\tau$, and $y_\tau: D_\tau / \Gamma_\tau \to \Omega^*/\Gamma$ (that we abusively denoted in the same way).

\begin{theorem} [Theorem~\ref{Theorem Intro domain}, GHCC part of~\ref{theo12}]
The affine spacetime $(D_\tau / \Gamma_\tau,\C)$ is a globally hyperbolic spacetimes admitting a $C^2$ locally uniformly $\C$-future-convex (Definition~\ref{def Convex Cauchy surface}) and compact Cauchy surface.
\end{theorem}
\begin{proof}
In the proof of Proposition~\ref{prop exists tau-support function extending to boundary}, we have shown that there exists a $C^2_+$ $\tau$-convex domain $K$ with $\partial K = \partial_{sp}K$. The $C^2_+$ nature of $K$ implies that $\partial K / \Gamma_\tau$ is a $\C$-spacelike $C^2$ locally uniformly convex hypersurface in $(D_\tau / \Gamma_\tau,\C)$. It is also compact as the Gauss map $\G_K$ descends to a diffeomorphism from $\partial K / \Gamma_\tau$ to $\Omega^* / \Gamma$ which is compact by Proposition~\ref{prop dual convex is divisible}. Finally, Corollary~\ref{cor C2+ implies Cauchy surface} descends to the quotient and implies that $\partial K / \Gamma_\tau$ is a Cauchy surface of $(D_\tau / \Gamma_\tau,\C)$.
\end{proof}

\begin{theorem}
The affine spacetime $(D_\tau / \Gamma_\tau,\C)$ is foliated by the level hypersurfaces of its cosmological time, which are $C^1$ $\C$-future-convex Cauchy surfaces.
\end{theorem}
\begin{proof}
The fact that level hypersurfaces of the cosmological time are a $C^1$ foliation is a quotiented version of Proposition~\ref{prop cosmological foliation of Dtau}. Those are quotients of $\C$-spacelike convex hypersurfaces, and thus, by Proposition~\ref{prop maximal domain = Cauchy development}, are future-convex Cauchy surfaces of $(D_\tau / \Gamma_\tau,\C)$.
\end{proof}

\begin{remark}
\label{rem even dimension}
Remembering the discussion in subsection~\ref{subsec The two maximal domains}, we can consider the two GHCC affine spacetimes $(D_{\tau}^{\C} / \Gamma_\tau , \C)$ and $(D_{\tau}^{-\C} / \Gamma_\tau , -\C)$. In general, they have no reason to be isomorphic affine spacetimes. In the case where $d+1$ is even and $\tau = 0 $, they actually are: indeed $-\Id \in \SL (\R^{d+1})$ and we have (see subsection~\ref{subsec The two maximal domains}) that $D^{\C}_0 = - D^{-\C}_{0}$.

In the case of flat Lorentzian spacetimes \cite{Mess_07,Bonsante_05}, $-\Id$ never belongs to $\SO_0(d,1) \subset \SL(\R^{d+1})$. Moreover, in that case, one can arbitrarily fix a time orientation on the common ambient affine Minkowski space containing all the maximal domains (the cone $\C$ does not vary) in order to have a favourite choice between the two spacetimes given by the two quotients of those maximal domains.
\end{remark}

\section{Correspondence between affine deformations of dividing groups and affine spacetimes}
\label{sec Correspondence between affine deformations of dividing groups and affine spacetimes}

This section is devoted to the proofs of the remaining maximality part of Theorem~\ref{Theorem Intro domain}~\ref{theo12} and of Theorem~\ref{Theorem Intro dev}, through the study of the developing maps of \emph{Globally Hyperbolic affine spacetimes admitting a $C^2$ locally uniformly future-Convex and Compact Cauchy surface} (GHCC affine spacetimes).

\subsection{\texorpdfstring{Restriction of a developing map to a compact $C^2$ locally uniformly convex Cauchy surface}{Restriction of a developing map to a compact C2 locally uniformly convex Cauchy surface}}

Let $(M,\C^M)$ be a GHCC affine spacetime. Using the equiaffine structure of $M$, we choose a pair $(\dev,\rho)$ of developing and holonomy maps (see subsection~\ref{subsec Equiaffine manifolds}). The developing map $ \dev: \widetilde{M} \to \A^{d+1}$, 
is then a local diffeomorphism and it is $\rho$-equivariant, that is for all $p \in \widetilde{M}$ and $\gamma \in \pi_1 M$, 
\begin{equation*}
\dev(\gamma \cdot p)=\rho(\gamma) \cdot \dev(p) \, .
\end{equation*}

The cone structure $\C^M$ on $M$ lifts to a cone structure $\C^{\widetilde{M}}$ on $\widetilde{M}$, which is parallel for the flat connection induced by the equiaffine structure of $\widetilde{M}$. Thus, we set $\C \subset \V^{d+1}$ to be the proper open convex cone such that in each tangent space $\mathrm{T}_{\dev(p)}\A^{d+1} = \V^{d+1},$ we have 
\begin{equation*}
\C = \dif_{p} \dev(\C^{\widetilde{M}}_p) \, .
\end{equation*}
From now on, we shall abusively write that this cone is defined as $\C \coloneqq \dif \, \dev(\C^M)$.

Let $S$ be a compact $C^2$ locally uniformly future-convex Cauchy surface of $(M,\C^M)$. We recall that Proposition~\ref{prop Geroch Cauchy surfaces} implies that $\pi_1 M = \pi_1S$ and that $\widetilde{S}$, the preimage of $S$ in the universal covering $\widetilde{M} \to M$, is a universal cover of $S$.

By definition of a $C^2$ locally uniformly $\C^M$-future-convex Cauchy surface (Definition~\ref{def Convex Cauchy surface}), the developing map $\dev$ induces a $\C$-spacelike locally uniformly convex immersion $\dev_{\vert \widetilde{S}}: \widetilde{S} \to \A^{d+1}$, namely $\dev_{\vert \widetilde{S}}$ is locally a diffeomorphism onto a $\C$-spacelike graph of a uniformly convex $C^2$ map. We can thus consider the Blaschke metric $h$ on $\widetilde{S}$ induced by that immersion (Definition~\ref{def Blaschke metric}). We can also fix any Euclidean metric $g_{euc}$ on $\A^{d+1}$ and consider its pullback by the restriction $\dev_{\vert\widetilde{S}}$, denoted by $\dev_{\vert\widetilde{S}}^*g_{euc}$.

As the immersion $\dev_{\vert \widetilde{S}}: \widetilde{S} \to \A^{d+1}$ is $\rho$-equivariant and $\rho$ takes values in $\SA(\A^{d+1})$, the Blaschke metric $h$ on $\widetilde{S}$ is preserved by $\pi_1S$ and descends to a Riemannian metric $\bar{h}$ on $S = \widetilde{S}/\pi_1 S$. The Riemannian manifold $(S,\bar{h})$ is closed and thus complete, so $(\widetilde{S},h)$ is also complete. We are thus in the setting of the following theorem of Trudinger and Wang \cite[Theorem~A]{Trudinger_Wang_02}.

\begin{theorem}[Trudinger--Wang \cite{Trudinger_Wang_02}]
\label{theo Trudinger--Wang}
Let $i: M \to \A^{d+1}$ be a $C^2$ locally uniformly convex immersed hypersurface, namely the map $i$ is locally a diffeomorphism onto a graph (in some well chosen affine coordinates) of a uniformly convex $C^2$ function. Let $h$ be the Blaschke metric on $M$ induced by that immersion and $g_{euc}$ be a Euclidean metric on $\A^{d+1}$. If the Riemannian manifold $(M,h)$ is complete, then $(M, i^*g_{euc})$ is also complete.
\end{theorem}

We can also parametrise $\A^{d+1}$ as $\R^{d+1}$ so that we have 
\begin{equation*}
\C = \left\lbrace t (x,1) \st x \in \Omega, t>0 \right\rbrace ,
\end{equation*}
where $\Omega$ is a bounded open convex domain of $\R^d$ containing $0$. In that setting, the following lemma holds.

\begin{lemma}
\label{lem projection is bilipschitz}
Let $\C \subset \V^{d+1}$ be an open proper convex cone and $g_{euc}$ a Euclidean metric on $\A^{d+1}$. Parametrise $\A^{d+1}$ as $\R^{d+1}$ such that $\C = \{ t (x,1) \ \vert \ x \in \Omega, t>0 \}$ where $\Omega$ is a bounded open convex domain of $\R^d$ containing $0$, and let $\pi: \R^{d+1} \to \R^d$ be the projection onto the first $d$ coordinates. If $i: M \to \A^{d+1}$ is a $\C$-spacelike immersed hypersurface, then the map $\pi\circ i: (M,i^*g_{euc}) \to (\R^d, g_{euc\vert\R^d})$ is bi-Lipschitz.
\end{lemma}
\begin{proof}
As $\Omega$ is a bounded open convex domain containing $0$, we can set
\begin{equation*}
M \coloneqq \sup\left\lbrace g_{euc\vert\R^d}(x,x) \mid x \in \Omega\right\rbrace >0 \, .
\end{equation*}
For every $\C$-spacelike vector $V \in \R^{d+1}$, we then have 
\begin{equation*}
g_{euc\vert\R^d}\bp{\pi(V),\pi(V)} \leq g_{euc}(V,V) \leq \vp{1+\frac{1}{M}}g_{euc\vert\R^d}\bp{\pi(V),\pi(V)} \, .
\end{equation*}
Applying that inequality to vectors $V$ tangent to the immersed hypersurface $i: M \to \A^{d+1}$, we get the lemma.
\end{proof}

We can now prove the following. 

\begin{proposition}
\label{prop devS is diffeo}
Let $(\dev,\rho)$ be a pair of developing and holonomy maps of a GHCC affine spacetime $(M,\C^M)$ and parametrise $\A^{d+1}$ as $\R^{d+1}$ so that the open proper convex cone $\C \coloneqq \dif \, \dev (\C^M)$ is expressed as $\C = \{ t (x,1) \ \vert \ x \in \Omega, t>0 \}$, where $\Omega$ is a bounded open convex domain of $\R^d$ containing $0$. If $S$ is a compact $C^2$ locally uniformly future-convex Cauchy surface of $(M,\C^M)$, the developing map $\dev$ on $\widetilde{M}$ embeds its universal cover $\widetilde{S} \subset \widetilde{M}$ into $\R^{d+1}$ as the graph of a $\C$-spacelike and $C^2$ locally uniformly convex function $f: \R^d \to \R$.
\end{proposition}
\begin{proof}
By definition of a $C^2$ locally uniformly $\C^M$-future-convex Cauchy surface (Definition~\ref{def Convex Cauchy surface}), in our well chosen parametrisation of $\A^{d+1}$, the image of $\dev_{\vert\widetilde{S}}$ is locally the graph of a $\C$-spacelike and $C^2$ locally uniformly convex function of the first $d$ coordinates. By Lemma~\ref{lem projection is bilipschitz}, $\pi \circ \dev_{\vert\widetilde{S}}$ is injective, so that $\dev(\widetilde{S})$ is globally the graph of a $\C$-spacelike $C^2$ locally uniformly convex function $f:U \to \R$ on some open subset $U \subseteq\R^d$.

Now, let us show that $U$ is the whole $\R^d$. We recall that the Blaschke metric $h$ on $\widetilde{S}$ given by the immersion $\dev_{\vert\widetilde{S}}:\widetilde{S} \to \R^{d+1}$ descends to the closed manifold $S = \widetilde{S}/\pi_1S$, so that $(\widetilde{S},h)$ is complete. By Theorem~\ref{theo Trudinger--Wang}, $(\widetilde{S},\dev_{\vert\widetilde{S}}^*g_{euc})$ is also complete. As bi-Lipschitz maps preserve completeness, Lemma~\ref{lem projection is bilipschitz} implies that $(U = \pi \circ \dev (\widetilde{S}),g_{euc\vert\R^d})$ is complete, so that $U$ is the whole $\R^d$.

Finally, let us prove that $\dev_{\vert\widetilde{S}}:\widetilde{S} \to \R^{d+1}$ is an embedding. By Lemma~\ref{lem projection is bilipschitz}, $\pi \circ \dev_{\vert\widetilde{S}}$ is injective on $\widetilde{S}$, thus $\dev$ is also injective on $\widetilde{S}$. As $(\widetilde{S},\dev_{\vert\widetilde{S}}^*g_{euc})$ is complete, the Hopf--Rinow Theorem implies that $\dev_{\vert\widetilde{S}}$ is proper. Thus, $\dev_{\vert\widetilde{S}}$ is a proper injective immersion, that is $\dev_{\vert\widetilde{S}}$ is an embedding.
\end{proof}

As we have parametrised $\A^{d+1}$ as $\R^{d+1}$, we can also see $\SA(\A^{d+1})$ as $\SA(\R^{d+1}) = \SL(\R^{d+1}) \ltimes \R^{d+1}$. Consider the linear part of the representation $\rho: \pi_1S \to \SA(\R^{d+1})$, it is a representation $\rho_{lin}: \pi_1S \to \SL(\R^{d+1})$. Then, the following holds.

\begin{proposition}
\label{prop Gauss on S is equivariant diffeo}
In setting of Proposition~\ref{prop devS is diffeo}, let $\G: \dev(\widetilde{S}) \to \Omega^*$ be the Gauss map of the $\C$-spacelike hypersurface $\graph(f) = \dev(\widetilde{S})$. The map $\phi \coloneqq \G \circ \dev: \widetilde{S} \to \Omega^*$ is a diffeomorphism onto its image $\Theta^*\coloneqq \G \circ \dev(\widetilde{S})$, which is an open convex subset of $\Omega^*$ preserved by the dual projective action (described at the beginning of Section~\ref{sec Affine deformations of divisible convex cones}) of $\rho_{lin}(\pi_1S)$, the linear part of $\rho$ on $\Omega^* \simeq\P(\C^*)$ .
Moreover, it is $\rho_{lin}$-equivariant, that is for all $p \in \widetilde{S}$ and $\gamma \in \pi_1 S$, 
\begin{equation*}
\phi(\gamma \cdot p)=\rho_{lin}(\gamma) * \phi(p) \, .
\end{equation*}

\end{proposition}
\begin{proof}
As $\graph(f)$ is a $\C$-spacelike $C^2$ locally uniformly convex hypersurface, the Gauss map $\G$ is defined as 
\begin{equation*}
\function{\G}{\graph(f)}{\Omega^*}{\bp{x,f(x)}}{\grad f (x)}
\end{equation*}
and is a $C^1$ diffeomorphism onto its image $\Theta^*$ which is an open subset of $\Omega^*$. This subset $\Theta$ is the interior of the support of the convex function $f^*$, the Legendre--Fenchel transform of $f$, and thus it is convex.

By Proposition~\ref{prop devS is diffeo}, the restriction $\dev_{\vert\widetilde{S}}$ is a diffeomorphism onto $\graph(f)$ so that $\phi = \G \circ \dev$ is a diffeomorphism onto $\Theta^*$.

As $\widetilde{S}$ is preserved by $\pi_1S=\pi_1M$ and $\dev$ is $\rho$-equivariant, $\dev(\widetilde{S})=\graph(f)$ is preserved by $\rho(\pi_1S)$. Using the equivariance of the Gauss map $\G$, we get that $\Theta^* = \G \circ \dev(\widetilde{S})$ is preserved by $\rho_{lin}(\pi_1S)$ and that $\phi = \G \circ \dev$ is $\rho_{lin}$-equivariant.
\end{proof}

\begin{corollary}
\label{cor rholin injective}
Let $\rho :\pi_1M \to \SL(\R^{d+1})\ltimes\R^{d+1}$ be a holonomy map of a GHCC affine spacetime $(M,\C^M)$. The linear part of $\rho$ is an injective representation $\rho_{lin}: \pi_1M \to \SL(\R^{d+1})$. That means, by Remark~\ref{rem necessary condition for affine deformation}, that $\rho(\pi_1M)$ is an affine deformation of its linear part $\rho_{lin}(\pi_1S)$.
\end{corollary}
\begin{proof}
As $S$ is a closed manifold, the action of $\pi_1S$ on $\widetilde{S}$ is free, and thus using the $\rho_{lin}$-equivariance of the diffeomorphism $\phi$ between $\widetilde{S}$ and $\Theta^*$ (Proposition~\ref{prop Gauss on S is equivariant diffeo}), $\rho_{lin}$ must be injective.
\end{proof}
 
\begin{corollary}
\label{cor theta divisible}
In the setting of Proposition~\ref{prop Gauss on S is equivariant diffeo}, the subgroup $\rho_{lin}(\pi_1S) < \SL(\R^{d+1})$ is discrete and the open convex domain $\Theta^*$, seen as a projective convex domain inside $\Omega^* \simeq\P(\C^*)$, is divisible by the dual projective action of $\rho_{lin}(\pi_1S) $.
\end{corollary}
\begin{proof}
As $S$ is a closed manifold, the action of $\pi_1S$ on $\widetilde{S}$ is free and properly discontinuous, and thus using the $\rho_{lin}$-equivariance of the diffeomorphism $\phi$ between $\widetilde{S}$ and $\Theta^*$ (Proposition~\ref{prop Gauss on S is equivariant diffeo}), we get that the dual projective action of $\rho_{lin}(\pi_1S) $ on $\Theta^*$ is free and properly discontinuous so $\rho_{lin}(\pi_1S) $ is discrete. Moreover we also have that the quotients $S=\widetilde{S}/\pi_1S$ and $\Theta^*/\rho_{lin}(\pi_1S)$ are diffeomorphic. As $S$ is compact, that means that $\Theta^*$ is divisible by $\rho_{lin}(\pi_1S) $.
\end{proof}

Corollary~\ref{cor theta divisible} indicates that we are in a setting where we can apply Lemma~\ref{lem inclusion and divisibility} to $\Theta^*\subseteq \Omega^*$. That implies that $\Theta^*=\Omega^*$, i.e. that the Gauss map $\G$ in Proposition~\ref{prop Gauss on S is equivariant diffeo} is surjective. Thus, the epigraph of $f$ is a $\rho(\pi_1M)$-invariant $\C$-convex domain (Definition~\ref{def C-convex}). As $\rho_{lin}(\pi_1M) = \rho_{lin}(\pi_1S)$ divides $\C^*$, the cone over $\Omega^*$, by Proposition~\ref{prop dual convex is divisible}, it also divides $\C$. To summarise, we have the following. 

\begin{proposition}
\label{prop dev embeds u.l.c. Cauchy surface}
Let $(\dev,\rho)$ be pair of developing and holonomy maps of a GHCC affine spacetime $(M,\C^M)$. Parametrise $\A^{d+1}$ as $\R^{d+1}$ and consider the proper convex cone $\C \coloneqq \dif \, \dev (\C^M) \subset \R^{d+1}$. The holonomy map $\rho$ is injective and its image is an affine deformation of its linear part $\Gamma\coloneqq \rho_{lin}(\pi_1M)$, which divides the cone $\C$. Introducing $\tau: \Gamma \to \R^{d+1}$ the cocycle such that $\rho(\pi_1M) = \Gamma_\tau$, if $S$ is a compact $C^2$ locally uniformly future-convex Cauchy surface of $(M,\C^M)$, the developing map $\dev$ embeds $\widetilde{S}$ in $\R^{d+1}$ as the boundary of a $\C$-spacelike and $C^2_+$ $\tau$-convex domain.
\end{proposition}

\subsection{ Properties of developing maps on GHCC affine spacetimes}
\label{subsec Properties of developing maps on a GHCC affine spacetime}

In this subsection, we extend the results in Proposition~\ref{prop dev embeds u.l.c. Cauchy surface} to properties on developing maps of GHCC affine spacetimes, in order to prove Theorem~\ref{theo GHCC embedding}, which is a classification theorem for GHCC implying Theorem~\ref{Theorem Intro dev}.

\begin{proposition}
\label{prop dev(M) included in Dtau}
Let $(\dev,\rho)$ be pair of developing and holonomy maps of a GHCC affine spacetime $(M,\C^M)$. Parametrise $\A^{d+1}$ as $\R^{d+1}$, consider the proper convex cone $\C \coloneqq \dif \, \dev (\C^M) \subset \R^{d+1}$ and, by Proposition~\ref{prop dev embeds u.l.c. Cauchy surface}, introduce $\tau$ a cocycle such that $\rho(\pi_1M) = \Gamma_\tau$ is an affine deformation of the subgroup $\Gamma \coloneqq \rho_{lin}(\pi_1M)$ dividing $\C$. The image $\dev(\widetilde{M})$ is contained in the maximal $\tau$-convex domain $D_\tau$ of $(\R^{d+1},\C)$.
\end{proposition}
\begin{proof}
Let us set $S$ to be compact $C^2$ locally uniformly future-convex Cauchy surface of $(M,\C^M)$. Notice that, using Proposition~\ref{prop maximal domain = Cauchy development} and Proposition~\ref{prop dev embeds u.l.c. Cauchy surface}, we have $D(\dev(\widetilde{S})) = D_\tau$.

Let $X =\dev(p) \in \dev(\widetilde{M})$ and $X(t)$ be an inextensible $\C$-causal curve in $\R^{d+1}$ parametrised by $t \in (-\infty,+\infty)$ and such that $X(0)=X$. As $\dev$ is a local diffeomorphism around $p$, we have that locally around $t=0$ the path $X(t)$ lifts to a $\C$-causal curve $p(t)$ in $\widetilde{M}$ such that $\dev(p(t))=X(t)$. 

Let $t^- \in \R_{-}^* \cup \{-\infty\}$ and $t^+ \in \R_{+}^* \cup \{+\infty\}$ be such that $(t^-,t^+)$ is the maximal interval on which the developing map lifts $X(t)$ to a $\C$-causal curve $p(t)$. By definition of $t^-$ and $t^+$, $\{ p(t) \ \vert \ t \in (t^-,t^+) \}$ is an inextensible $\C$-causal curve in $\widetilde{M}$, thus there exist a unique $t_s \in (t^-,t^+)$ such that $p(t_s)$ belongs to the Cauchy surface $\widetilde{S}$. Then, $\dev(p(t_s))=X(t_s) \in \dev(\widetilde{S})$.

Thus, any inextensible $\C$-causal curve going through $X = \dev(p)$ intersects $\dev(\widetilde{S})$, meaning precisely that $\dev(p) \in D(\dev(\widetilde{S})) = D_\tau$.
\end{proof}

\begin{proposition}
\label{prop dev is injective}
A developing map $\dev: \widetilde{M} \to \A^{d+1}$ of a GHCC affine spacetime $(M,\C^M)$ is necessarily injective.
\end{proposition}
\begin{proof}
Let us work in the setting of Proposition~\ref{prop dev(M) included in Dtau} and prove it by contradiction. Assume there are $p \neq q \in \widetilde{M}$ such that $\dev(p)=\dev(q)=X \in D_\tau$ (by Proposition~\ref{prop dev(M) included in Dtau}). Let $X(t)$ be an inextensible $\C$-causal curve with $X(0) = X$. As $X \in D_\tau = D (\dev(\widetilde{S}))$, there is a unique $t_s \in \R$ such that $X(t_s) \in \dev(\widetilde{S})$. Following what was done in the proof of Proposition~\ref{prop dev(M) included in Dtau}, there are real numbers $t^-<\min(0,t_s)$ and $t^+>\max(0,t_s)$ such that for all $t \in (t^-, t^+)$,
\begin{equation*}
X(t)= \dev\bp{p(t)} = \dev\bp{q(t)} \, ,
\end{equation*}
where $p(t),q(t)$ are $\C^M$-causal curves in $\widetilde{M}$ and $p(t_s),q(t_s) \in \widetilde{S}$. As 
\begin{equation*}
X(t_s)=\dev\bp{p(t_s) } = \dev\bp{q(t_s)} \, ,
\end{equation*}
using the injectivity of the developing map on $\widetilde{S}$ (Proposition~\ref{prop dev embeds u.l.c. Cauchy surface}), we get that $p(t_s)=q(t_s)$ (thus $t_s\neq 0$).

If $t_s<0$, let $T = \max \{ t \in [t_s,0] \ \vert \ p(t)=q(t)\}$ which is well-defined because of the continuity of $p(t)$ and $q(t)$ and that $p(t_s)=q(t_s)$. By assumption $p(0) \neq q(0)$ and thus $t_s \leq T<0$. Then for all $\varepsilon >0$ such that $T< T + \varepsilon < 0$, we have 
\begin{equation*}
p(T +\varepsilon) \neq q(T+ \varepsilon) \quad \text{and} \quad X(T+\varepsilon) = \dev\bp{p(T + \varepsilon) } = \dev\bp{q(T + \varepsilon) } \, ,
\end{equation*} 
which contradicts the fact that $\dev$ is a local diffeomorphism at $p(T)=q(T)$ (see Figure~\ref{fig Developing}).

The case where $t_s>0$ can be treated in a similar way.
\end{proof}

\begin{figure}[ht]
\centering
\includegraphics{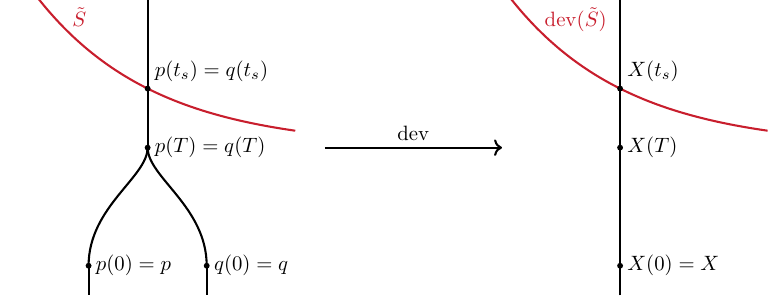}
\caption{Illustration of the situation in the proof of Proposition~\ref{prop dev is injective}.}
\label{fig Developing}
\end{figure}

\begin{theorem}
\label{theo GHCC embedding}
Let $(M,\C^M)$ be a GHCC affine spacetime. Any pair $(\dev,\rho)$ of a developing map
\begin{equation*}
\dev: \widetilde{M} \longrightarrow \R^{d+1} \, ,
\end{equation*}
and a holonomy map 
\begin{equation*}
\rho: \pi_1M \longrightarrow \SA(\R^{d+1}) \, ,
\end{equation*}
given by its equiaffine structure satisfies the following:
\begin{itemize}
\item the holonomy map $\rho$ is a group isomorphism onto its image $ \rho(\pi_1M)$ which is an affine deformation $\Gamma_\tau$ of a subgroup $\Gamma < \SL(\R^{d+1})$ dividing the cone $\C \coloneqq \dif \, \dev(\C^M)$, 
\item the developing map $\dev$ is an isomorphism of affine spacetimes onto its image which is included in $(D_\tau,\C)$, where $D_\tau $, the domain associated with $\Gamma_\tau \coloneqq \rho(\pi_1M)$ and $\C$, is given by Theorem~\ref{Theorem Intro domain},

\item the developing map $\dev$ induces an embedding $(M,\C^M) \lhook\joinrel\to (D_{\tau} / \Gamma_\tau, \C)$ preserving the equiaffine and cone structures, and mapping any $C^2$ locally uniformly $\C^M$-future-convex and compact Cauchy surface of $(M,\C^M)$ to a $C^2$ locally uniformly $\C$-future-convex and compact Cauchy surface of $(D_{\tau} / \Gamma_\tau, \C)$.
\end{itemize}
\end{theorem}
\begin{proof}
The first point is a consequence of Proposition~\ref{prop dev embeds u.l.c. Cauchy surface} and the second one of Propositions~\ref{prop dev(M) included in Dtau} and~\ref{prop dev is injective}. Finally, using those two first points and the equivariance of the developing map we get the third one.
\end{proof}

In Theorem~\ref{theo GHCC embedding}, if moreover $(M,\C^M)$ is a MGHCC, then the definition of \emph{maximality of globally hyperbolic affine spacetimes} (Definition~\ref{def Maximal}) implies that the embedding $(M,\C^M) \lhook\joinrel\to (D_{\tau} / \Gamma_\tau, \C)$ given by the developing map must be surjective, proving Theorem~\ref{Theorem Intro dev}.

\subsection{\texorpdfstring{Maximality of the quotient spacetimes $D_\tau / \Gamma_\tau$}{Maximality of the quotient spacetimes D-tau / Gamma-tau}}
\label{subsec Maximality of the quotient spacetimes}

In this subsection, we use the preceding work in order to prove the maximality of spacetimes built with Theorem~\ref{Theorem Intro domain}.

Let $\Gamma_\tau < \SL (\R^{d+1}) \ltimes \R^{d+1}$ be an affine deformation of a group $\Gamma < \SL (\R^{d+1})$ dividing an open proper convex cone $\C$. We consider the associated spacetimes $(D_{\tau}^{\C} / \Gamma_\tau , \C)$ given by Theorem~\ref{Theorem Intro domain}. In Section~\ref{sec Quotient of the maximal domain} we have proved that it is a GHCC affine spacetime. Using the results from subsection~\ref{subsec Properties of developing maps on a GHCC affine spacetime} we have the following.

\begin{theorem}[Theorem~\ref{Theorem Intro domain}~\ref{theo12}]
\label{theo maximality of spacetime}
Let $\Gamma_\tau < \SL (\R^{d+1}) \ltimes \R^{d+1}$ is an affine deformation of a group $\Gamma < \SL (\R^{d+1})$ dividing an open proper convex cone $\C$. The GHCC affine spacetime $(D_{\tau} / \Gamma_\tau, \C)$ is a maximal globally hyperbolic spacetime.
\end{theorem}

\begin{proof} 
Let $i: (D_{\tau}/ \Gamma_\tau, \C) \lhook\joinrel\to (M',\C')$ be an embedding preserving the equiaffine and cone structures, and mapping a $C^2$ locally uniformly $\C$-future-convex compact Cauchy surface $S$ of $(D_{\tau} / \Gamma_\tau,\C)$ to a Cauchy surface $S'$ of $(M',\C')$.

First, notice that, because $i$ preserves the equiaffine structure, $S'$ is a $C^2$ locally uniformly $\C'$-future-convex Cauchy surface. Thus, $(M',\C')$ is also a GHCC affine spacetime. The embedding $i$ lifts to an embedding $\tilde{i}: (D_{\tau}, \C) \lhook\joinrel\to (\widetilde{M'},\C')$ also preserving the equiaffine and cone structures, and mapping a $C^2$ locally uniformly $\C$-future-convex Cauchy surface $\widetilde{S}$ of $D_{\tau}$ to a $C^2$ locally uniformly $\C'$-future-convex Cauchy surface $\widetilde{S'}$ of $\widetilde{M'}$. We want to show that $i$ is surjective.

Now, let $(\dev, \rho)$ be a couple of developing and holonomy maps for the GHCC affine spacetime $(M',\C')$. Theorem~\ref{theo GHCC embedding} states that, $\rho(\pi_1 M')$ is an affine deformation $\Gamma''_{\tau''}$ of its linear part $\Gamma'' < \SL(\R^{d+1})$, which divides a proper convex cone $\C''$. It also states that the developing map $\dev$ on $(\widetilde{M'},\C')$ embeds it into $(D''_{\tau''}/ \Gamma''_{\tau''}, \C'')$.

Consider the following embedding $\tilde{j} = \dev \circ \tilde{i}: (D_{\tau},\C) \lhook\joinrel\to (D_{\tau''},C'')$, which descends to an embedding $j: (D_{\tau}/ \Gamma_\tau, \C) \lhook\joinrel\to (D_{\tau''}/ \Gamma_{\tau''}, \C'')$. It preserves the equiaffine structure between two convex domains of $\R^{d+1}$ so it is of the form $\tilde{j}: X~\longmapsto~AX$, where $A \in \SL(\R^{d+1})$. Thus, up to post-composing we can assume that $\tilde{j}$ is a restriction of the identity. 

As the embedding $\tilde{j}$ is a restriction of the identity, we clearly have $D_\tau \subseteq D_{\tau''}$. Moreover as $\tilde{j}$ preserves the cone structure, we also have $\C= \C''$. 

Using the immersion $i : (D_{\tau}/ \Gamma_\tau, \C) \lhook\joinrel\to (M',\C') $ we can see $\pi_1 (D_{\tau}/ \Gamma_\tau)\simeq \Gamma_\tau$ as included in $\pi_1M'$. Actually, as $S$ is a Cauchy surface of $M$ and $S'=i(S)$ is a Cauchy surface of $M'$, by Proposition~\ref{prop Geroch Cauchy surfaces}, we have an identification
\begin{equation*}
 \Gamma_\tau \simeq \pi_1(D_{\tau}/ \Gamma_\tau) \simeq \pi_1S \overset{i}{\simeq}\pi_1 S'\simeq \pi_1M' \, .
\end{equation*}
Then, using that identification, the embedding $\tilde{j} = \dev \circ \tilde{i}$ is $\rho$-equivariant in the sense that for all $(\gamma,\tau_\gamma) \in \Gamma_\tau \simeq \pi_1M'$ and $X\in D_\tau$, we have 
\begin{equation*}
\tilde{j} (\gamma X + \tau_\gamma) = \rho(\gamma,\tau_\gamma) \cdot \widetilde{X} \, ,
\end{equation*}
where $\rho(\gamma,\tau_\gamma)\in \SA(\R^d)$. As $\tilde{j}$ is a restriction of the identity and $D_\tau$ is open, we get that the affine transformations $(\gamma,\tau_\gamma)$ and $\rho(\gamma,\tau_\gamma)$ are equal. Thus, $\Gamma''_{\tau''} = \rho(\pi_1M')=\Gamma_\tau$.

Finally, that implies that $\tilde{j} = \Id_{\vert D_{\tau}}$ is surjective. Hence, as $\tilde{j} = \dev \circ \tilde{i}$ and $\dev$ is a diffeomorphism onto its image, $\tilde{i}$ is also surjective and so is $i$.
\end{proof}

\begin{remark}
Let $(M,\C)$ be a MGHCC affine spacetime, like in subsection~\ref{subseclength of causal curves}, using affine spheres in the tangent spaces of $M$, one can again define a norm $F$ on $\C$-null and $\C$-timelike vectors, a length $l_F$ of $\C$-causal curve and a $\C$-distance $\rho$ satisfying the time inequality. Theorem~\ref{Theorem Intro dev} implies that again this norm endows $(M, \C)$ with a timelike space structure in the sense of Busemann \cite{Busemann_67}, and that the cosmological time on $(M,\C)$ satisfies 
\begin{align*}
\mathcal{T}(p) & = \sup \left\lbrace \rho(q,p) \st q \in J^-(p) \right\rbrace \\
& = \sup \left\lbrace l_F(\gamma) \st \gamma \ \text{is a future directed $\C$-causal curve from a point of $J^-(p)$ to $p$} \right\rbrace .
\end{align*}
\end{remark}

\appendix

\section{Examples of dividing groups admitting affine deformations}
\label{sec Examples of affine deformations}

This appendix is devoted to the construction of examples proving the following proposition.

\begin{proposition}
In any dimension $d \geq 2$, there exists a discrete subgroup $\Gamma < \SL(\R^{d+1})$ dividing a proper convex cone $\C \subset \R^{d+1}$ over a proper convex domain $\Omega \subset\P(\R^{d+1})$ which is not projectively equivalent to an ellipse (giving the Klein ball model of the hyperbolic space).
\end{proposition}

\subsection*{Bending of holonomy groups}

An amalgamated product $\Gamma = \Gamma'*_\Lambda \Gamma'' < \SL(\R^{d+1})$ can be algebraically deformed through the following general procedure called \emph{bending} and deeply studied by Johnson and Millson \cite{Johnson_Millson_87}. Let $(A_s)_{s\in\R}$ be a continuous path of elements of $\SL(\R^{d+1})$ such that $A_0$ is the identity and every $A_s$ commutes with all the elements of $\Lambda$. The bending procedure consists in continuously deforming $\Gamma$ by looking at the family $\Gamma_s = \Gamma'*_\Lambda (A_s\Gamma''A_s^{-1})$.

\subsection*{Bulging the holonomy groups of hyperbolic manifolds}

Let $\Gamma_0 < \SO_0(d,1)$ be the holonomy group of a closed hyperbolic $d$-dimensional manifold $M = \B^d/\Gamma_0$ (using the Klein ball projective model of the hyperbolic space) containing a totally geodesic embedded compact connected hypersurface $N$ separating $M$ into two connected components: $M \setminus N = M' \sqcup M''$. We shall deform $\Gamma_0$ through a particular bending procedure called \emph{bulging} (following the nomenclature in \cite{Blayac_Viaggi_24}).

Using the inclusion $i:\pi_1 N \hookrightarrow \pi_1 M$ and taking the image by the holonomy of $\pi_1 N$ we get a subgroup $\Lambda < \Gamma_0$ preserving a hyperplane $H$ intersecting the Minkowski light cone $\C_0$. Then $H \cap (\B^d \times \{1\})$ is a connected preimage of $N$ by the projection $\B^d \times \{1\} \to M= \B^d/\Gamma_0$ and a universal cover over $N$. For $k \in \{1,2\}$, using the inclusion $i^{(k)}:\pi_1 M^{(k)} \hookrightarrow \pi_1 M$ and setting $\Gamma_0^{(k)}$ to be the holonomy of $\pi_1 M^{(k)}$, by the Seifert--van Kampen Theorem we can see the holonomy group of $M$ as the amalgamated product $\Gamma_0 = \Gamma_0' *_\Lambda \Gamma_0''$.

Let $X \in \R^{d,1}=\R^{d+1}$ be a unit Minkowski spacelike vector orthogonal to the hyperplane $H$. For later constructions, it is worth noting that $X$ is a common eigenvector, with associated eigenvalue $1$, of all elements of $\Lambda$. For $s \in \R$, let $A_s \in \SL(\R^{d+1})$ be the element that restricts to the homothety $e^s I_n$ on $H$ and sends $X$ to $e^{-d s}X$ (in a well-chosen basis it has form $\mathrm{Diag}(e^s,\dots,e^s,e^{-ds})$ and thus its determinant is indeed $1$). As $A_s \gamma A_s^{-1} = \gamma$, for all $\gamma \in \Lambda$, we can continuously deform $\Gamma_0$ by bulging by considering $\Gamma_s \coloneqq \Gamma_0' *_\Lambda (A_s \Gamma_0'' A_s^{-1}) < \SL(\R^{d+1})$. Then, by the work of Koszul \cite{Koszul_70} and Benoist \cite{Benoist_05}, each group $\Gamma_s$ happens to also divide an open proper convex cone $C_s$, which is not affinely equivalent to the Minkowski light cone if $s\neq 0$ \cite[Section 8.2]{Benoist_08}. 

\subsection*{Bending to get affine deformations of groups obtained by bulging}

Let us now prove that such a group $\Gamma \coloneqq\Gamma_{s_0}$ (we fix $s_0$ for the rest of this section) obtained by bulging admits a non-trivial cocycle $\tau \in \mathrm{H}^1(\Gamma, \R^{d+1})$, i.e. that $\Gamma$ admits a non-trivial affine deformation $\Gamma_\tau < \SL(\R^{d+1})\ltimes \R^{d+1}$. We shall proceed by seeing $\SL(\R^{d+1})$ as included in $\SL(\R^{d+1})\ltimes \R^{d+1}$ in order to bend the amalgamated product $\Gamma = \Gamma'*_\Lambda \Gamma''$ (where $\Gamma'=\Gamma_0'$ and $\Gamma''= A_s \Gamma_0'' A_s^{-1})$. 

Remembering that there is $X \in \R^{d+1}$, a common eigenvector, with associated eigenvalue $1$, of all elements of $\Lambda$, we set for all $s \in \R$, $B_s = (I,sX) \in \SL(\R^{d+1})\ltimes \R^{d+1}$, where $I$ is the identity in $\SL(\R^{d+1})$. Then notice that for all $\gamma \in \Gamma$ and $s \in \R$,
\begin{equation}
\label{eq linear part of bending}
B_s(\gamma,0)B_s^{-1} = \bp{\gamma, s(X - \gamma X)} \, .
\end{equation}
If moreover $\gamma \in \Lambda$, as $X$ is an eigenvector of $\gamma$ with associated eigenvalue $1$, it becomes 
\begin{equation*}
B_s(\gamma,0)B_s^{-1} = (\gamma, 0) \, .
\end{equation*}
Thus, we can consider the bending of $\Gamma = \Gamma_{\tau=0} < \SL(\R^{d+1})\ltimes \R^{d+1}$ by the family $(B_s)_{s \in \R}$. We get a family of groups $G_s = \Gamma' *_\Lambda (B_s \Gamma'' B_{s}^{-1})$ in $\SL(\R^{d+1})\ltimes \R^{d+1}$. Equality \eqref{eq linear part of bending} implies that we have actually built a family of affine deformations $(\Gamma_{\tau_s})_{s \in \R}$ of $\Gamma$. 

To conclude, let us prove that for $s \neq 0$, such a deformation is non-trivial. It suffices to prove it for $s=1$. In order to lighten expressions we shall denote $\tau_1$ by $\tau$. By contradiction, assume that $\tau$ is a coboundary, that is that we can set $V \in \R^{d+1}$ such that for all $\gamma \in \Gamma$, $\tau(\gamma) = (I - \gamma)V$. For all $\gamma' \in \Gamma'$, we have 
\begin{equation*}
\tau(\gamma') = (I - \gamma')V = 0 \, ,
\end{equation*}
thus $V$ is either $0$ or an eigenvector of all $\gamma' \in \Gamma'$ with associated eigenvalue $1$.

As $\Lambda$ is conjugate to a subgroup of the form
\begin{equation*}
\vp{\begin{array}{c|c}
G & 0 \\
\hline 0 & 1 
\end{array}} < \SL(\R^{d+1}) \, ,
\end{equation*}
where $G < \SO(d-1,1)$ is the holonomy group of a closed hyperbolic $(d-1)$-dimensional manifold $N$, its only common eigenvector is $X$. Thus, $V$ is necessarily in the span of $X$. 

We claim that for all $\gamma \in \Gamma \setminus \Lambda$, $X$ is not an eigenvector of $\gamma$ with associated eigenvalue $1$. Then, for $\gamma' \in \Gamma' \setminus \Lambda$, as $X$ is not a an eigenvector with associated eigenvalue $1$ of $\gamma'$, if $V \neq 0$, then
\begin{equation*}
\tau(\gamma') = (I - \gamma')V \neq 0 \, ,
\end{equation*}
so necessarily $V=0$. But then for $\gamma'' \in \Gamma'' \setminus \Lambda$, as $X$ is not a an eigenvector with associated eigenvalue $1$ of $\gamma''$,
\begin{equation*}
\tau(\gamma'') = (I - \gamma'')X \neq 0 \, ,
\end{equation*}
and thus $V \neq 0$. That is a contradiction.

Now let us prove our claim. Let $\gamma \in \Gamma \setminus \Lambda$, if $X$ was an eigenvector of $\gamma$ with associated eigenvalue $1$ we would have that $\gamma$ preserves the hyperplane $H$, thus any straight path from a point $P \in \C \cap H$ to $\gamma P$ would descends to a loop in $N$ with holonomy $\gamma$, which is not in $\Lambda$, the image of $\pi_1 N$ by the holonomy .

\section{The projective approach to affine deformations}
\label{sec Projective approach}
This appendix offers a more projective approach to affine deformations in order to relate the work of this article to the notion of convex cocompact representation \cite{DGK_17} .

Let $\Gamma$ be a torsion-free discrete subgroup of $\SL (\R^{d+1})$ dividing an open proper convex cone $\C$ in $\R^{d+1}$. Seeing the affine space $\R^{d+1}$ as an affine chart of the $(d+1)$-dimensional real projective space $\P(\R^{d+2})$:
\begin{equation*}
\namelessembedding{\R^{d+1}}{\P(\R^{d+2})}{X}{(X:1)} \, ,
\end{equation*}
we can see $\SA (\R^{d+1}) = \SL(\R^{d+1}) \ltimes \R^{d+1}$ as a Lie subgroup of $\PGL(\R^{d+2})$:
\begin{equation*}
\namelessembedding{\SL(\R^{d+1}) \ltimes \R^{d+1}}{\PGL(\R^{d+2})}{ (\gamma,\tau)}{\left[\begin{matrix} \gamma & \tau \\ 0 & 1 \end{matrix}\right]} \, .
\end{equation*}
As $\Gamma$ divides the cone $\C = \{ t(x,1) \ \vert \ x \in \Omega, t>0 \}$, the domain $\Omega_\P$ defined by
\begin{equation*}
\Omega_\P \coloneqq \left\lbrace (x:1:0) \st x \in \Omega \right\rbrace \subseteq \P(\R^{d+2})
\end{equation*}
is preserved by $\Gamma$ seen a subgroup of $\PGL(\R^{d+2})$. Let $\C_\P$ be the cone over $\Omega_\P$ with vertex at $[e_{d+2}] = (0:\dots:0:1)\in \P(\R^{d+2})$, i.e.
\begin{equation*}
\C_\P \coloneqq \Omega_\P \cup \left\lbrace (tx:t:1) \st x \in \Omega,t \neq 0 \right\rbrace = \Omega_\P \cup \left\lbrace (X:1) \st X \in \C \right\rbrace .
\end{equation*}
Then $\C_\P$ is also preserved by $\Gamma$ and moreover, noticing that 
\begin{equation*}
\C_\P = \left\lbrace (x:1:\lambda) \st x \in \Omega,\lambda \in \R \right\rbrace,
\end{equation*}
we easily see that $\C_\P$ is convex and homeomorphic to $\Omega \times \R$. Its polar cone in the dual projective space $\P(\R^{d+2})^*$, identified with $\P(\R^{d+2})$ using the usual scalar product on $\R^{d+2}$, is homeomorphic to $\Omega^* \times \R$ and can be expressed as 
\begin{equation}
\label{eq projective polar cone as a pipe}
\C_\P^* = \left\lbrace (y:-1:\mu) \st y \in \Omega^*,\mu \in \R \right\rbrace .
\end{equation}
Let $\Gamma_\tau$ be an affine deformation of $\Gamma$. An element $ (\gamma, \tau_\gamma) \in \Gamma_\tau$ acts:
\begin{itemize}
\item on $\P(\R^{d+2})$ by 
\begin{equation*}
(\gamma, \tau_\gamma) \cdot (X:1) = (\gamma X + \tau_\gamma: 1) \quad \text{and} \quad (\gamma, \tau_\gamma) \cdot (X:0) = (\gamma X:0) \, ,
\end{equation*}
\item on $\P(\R^{d+2})^*$ by 
\begin{equation*}
(\gamma, \tau_\gamma) \star (Y:1) = (\gamma \star Y: 1 + \tau_{\gamma^{-1}} \cdot Y) \quad \text{and} \quad (\gamma, \tau_\gamma) \star (Y:0) = (\gamma \star Y:0) \, ,
\end{equation*}
where $\gamma \star Y = \gamma^{\top}Y$. That dual action comes from the action of $\SA^* (\A^{d+1})$, the group of \emph{coaffine transformations} \cite{Bobb_Farre_24} seen as a subgroup of $\PGL^*(\R^{d+1})$:
\begin{equation*}
\namelessembedding{\SL^*(\R^{d+1}) \ltimes (\R^{d+1})^*}{\PGL^*(\R^{d+2})}{(\eta,\theta)}{\left[\begin{matrix} \eta & 0 \\ \theta & 1 \end{matrix}\right]} \, ,
\end{equation*} 
in which $\Gamma_\tau$ is seen as
\begin{equation*}
\left\lbrace \left[\begin{matrix} \gamma^{-\top} & 0 \\ (\tau_{\gamma^{-1}})^\top & 1 \end{matrix}\right] \in \PGL^*(\R^{d+1}) \st \gamma \in \Gamma \right\rbrace .
\end{equation*}
\end{itemize}
If $s$ is the support function of a $\tau$-convex domain, using the identification \eqref{eq projective polar cone as a pipe}, we can see its graph as embedded in $\P(\R^{d+2})^*$:
\begin{equation*}
\graph_\P(s) \coloneqq \left\lbrace \bp{y:-1:-s(y)} \st y \in \Omega^*\right\rbrace = \left\lbrace \bp{Y:-\tilde{s}(Y)} \st Y \in \C^*\right\rbrace .
\end{equation*}
Then the $\tau$-equivariance \eqref{eq action on C* support function} of $s$ precisely tells that $\graph_\P(s)$ is $\Gamma_\tau$-invariant. The existence of the unique boundary function $g_\tau: \partial \Omega^* \to \R$ (Proposition~\ref{prop existence of boundary function gtau}) thus tells us that there exist a $\Gamma_\tau$-invariant limit set $\Lambda_{\tau}^* \subset \partial \C_\P^*$ defined by
\begin{equation*}
\Lambda_{\tau}^* \coloneqq \graph_\P(g_\tau) \coloneqq \left\lbrace \bp{y:-1:-g_\tau(y)} \st y \in \Omega^*\right\rbrace .
\end{equation*} 

\begin{definition} [{\cite[Definition 1.1]{DGK_17}}]
Let $G < \PGL(\R^{d+2})$ be an infinite discrete subgroup.
\begin{itemize}
\item Let $W$ be a $G$-invariant properly convex open subset of $\P(\R^{d+2})$. The action of $G$ on $W$ is \emph{strongly convex cocompact} if $W$ is strictly convex with $C^1$ boundary and that for some $p \in W$, the convex hull in $W$ of the \emph{orbital limit set} $\overline{\Gamma_\tau \cdot p}\cap \partial W$ is non-empty
and has compact quotient by $G$.
\item The group $G$ is \emph{strongly convex cocompact in $\P (\R^{d+2})$} if it admits a strongly convex cocompact action on some properly convex open subset $W$ of $\P(\R^{d+2})$.
\end{itemize}
\end{definition}

\begin{proposition}
When $\Gamma$ is hyperbolic, $\Gamma_\tau$ is strongly convex cocompact in $\P (\R^{d+2})^*$, and $\Lambda_{\tau}^*$ is its orbital limit set.
\end{proposition}

\begin{remark}
\label{rem hyperbolic group gives C1 domain}
By the work of Benoist \cite[Theorem~1.1]{Benoist_04}, if $\Gamma$ is hyperbolic, then the domain $\Omega^*$ is strictly convex and $\partial \Omega^*$ is $C^1$.
\end{remark}

\begin{proof}
For the fact that $\Gamma_\tau$ is strongly convex cocompact, a way to proceed is to realise that $\Gamma_\tau \simeq \Gamma$ is hyperbolic and show that the action of $\Gamma_\tau$ is projective-Anosov, as in \cite[Lemma 3.6]{Bobb_Farre_24}, and then use Theorem~1.4 from \cite{DGK_17}. In fact, we can explicitly build a $C^1$ strictly convex $\Gamma_\tau$-invariant domain of $\P (\R^{d+2})^*$ on which the action of $\Gamma_\tau$ is strongly convex cocompact. In Proposition~\ref{prop exists tau-support function extending to boundary} we proved the existence of $s_-$ and $s_+$, two smooth $\tau$-equivariant function on $\Omega^*$ which are respectively $C^2$ locally uniformly convex and $C^2$ locally uniformly concave satisfying that for all $y \in \partial \Omega^*$,
\begin{equation}
\label{eq gradient of s+_}
\left\vert \grad s_\pm(y')\right\vert \xrightarrow[y' \to y]{} +\infty \, .
\end{equation}
Consider the following convex domain in $\P (\R^{d+2})^*$:
\begin{equation*}
W^* \coloneqq \left\lbrace (y:-1:-\mu) \st y \in \Omega^*, s_-(y) < \mu < s_+(y)\right\rbrace \subset \C_\P^* \, .
\end{equation*}
We claim that the dual action of $\Gamma_\tau$ on $W^*$ is strongly convex cocompact.

We decompose the boundary of $W^*$ as 
\begin{equation*}
\partial W^* = \graph_\P(s_-) \cup \Lambda_{\tau}^* \cup \graph_\P(s_-) \, .
\end{equation*}
The strict convexity of $\Omega^*$ and the respective strict convexity and concavity of $s_-$ and $s_+$ tells us that there is no line segment in $\partial W^*$, i.e. that $W^*$ is strictly convex. 

By Proposition~\ref{prop C1 convex function}, in order to prove that $\partial W^*$ is $C^1$, we only have to prove that at each point of its boundary $W^*$ admits only one supporting hyperplane. Using the smoothness of $s_-$ and $s_+$, that is clear for points in $ \graph_\P(s_+) \cup \graph_\P(s_-)$. Let $P = (y:-1:-g_\tau(y))$ be a point of $\Lambda_{\tau}^* \subset \partial W^*$, then condition \eqref{eq gradient of s+_} implies that any supporting hyperplane to $W^*$ at $P$ is a supporting hyperplane (at $P$) to the convex domain $\C^*_\P = \{ (y:-1:\mu) \ \vert \ y \in \Omega^*,\mu \in \R \}$ which has $C^1$ boundary because of Remark~\ref{rem hyperbolic group gives C1 domain}, and is thus unique. Hence, we have that $\partial W^*$ is $C^1$.

Let $P = (y:-1:-\mu)$ be a point of $W^*$. Using the well-known dynamic of the action of $\Gamma$ on $\Omega^*$ \cite[Proposition~5.1]{Benoist_04} and the fact that any $\tau$-invariant function extends continuously to $g_\tau$, we get that the proximal limit set of the action of $\Gamma_\tau$ on $W^*$ is $\overline{\Gamma_\tau \cdot P}\cap \partial W^* = \Lambda_{\tau}^*$.

Finally, $\mathrm{CH}_\P^*(\tau)$,the convex hull of $\Lambda_\tau^*$,  is foliated by $\Gamma_\tau$-invariant graphs of the linear combinations $((1-t)s_-+ts_+)_{t \in [0,1]}$. Thus, the action of $\Gamma_\tau$ on $\mathrm{CH}_\P^*(\tau)$ is cocompact, and  $\mathrm{CH}_\P^*(\tau)$ is the universal cover of the convex core $\mathrm{CC}^*_\P(\tau) \coloneqq \mathrm{CH}^*_\P(\tau) / \Gamma_\tau \simeq \Omega^* / \Gamma \times [0,1]$.
\end{proof}

\begin{remark}
Notice that when $\Gamma$ is not hyperbolic, we still have a limit set $\Lambda_\tau^*$, but it does not come from a projective-Anosov property: there is no limit map $\partial \Gamma_\tau \to \Lambda_\tau^*$.
\end{remark}

\end{spacing}

\bibliography{Aff_spacetimes}
\bibliographystyle{alpha}

\end{document}